\documentclass[12pt]{amsart}

\usepackage[utf8]{inputenc}

\usepackage{hyperref}
\usepackage{graphicx}
\usepackage{amssymb}
\usepackage{float}
\usepackage{fullpage}
\usepackage{physics}
\usepackage{comment}
\usepackage{tikz}
\usetikzlibrary{positioning, arrows.meta, shapes.geometric}
\usepackage{mathtools}
\newtheorem{theorem}{Theorem}[section]
\newtheorem{lemma}[theorem]{Lemma}
\newtheorem{corollary}[theorem]{Corollary}
\newtheorem{proposition}[theorem]{Proposition}

\theoremstyle{definition}
\newtheorem{definition}{Definition}

\theoremstyle{remark}
\newtheorem{remark}{Remark}

\numberwithin{equation}{section}

\newcommand{\cond}{\operatorname{cond}}
\newcommand{\B}{\mathcal{B}}
\newcommand{\A}{\mathcal{A}}

\newcommand{\C}{{\mathbb C}}
\newcommand{\R}{{\mathbb R}}

\newcommand{\Z}{{\mathbb Z}}
\newcommand{\N}{{\mathbb N}}

\newcommand{\eps}{\varepsilon}

\newcommand{\SL}{\operatorname{SL}_2}
\newcommand{\GL}{\operatorname{GL}_2}
\newcommand{\MM}{\operatorname{M}}

\renewcommand{\Re}{\operatorname{Re}}
\renewcommand{\Im}{\operatorname{Im}}
\renewcommand{\exp}{\operatorname{exp}}
\renewcommand{\P}{\mathbb{P}}

\renewcommand{\a}{\mathrm{a}}
\newcommand{\n}{\mathrm{n}}
\renewcommand{\k}{\mathrm{k}}

\renewcommand{\c}{\mathfrak c}
\renewcommand{\d}{\mathrm d}
\renewcommand{\n}{\mathrm{n}}
\newcommand{\m}{\mathfrak{m}}
\newcommand{\w}{\mathrm w}

\newcommand{\sgn}{\operatorname{sgn}}
\renewcommand{\pmod}[1]{\,\,(\mathrm{mod}\,{#1})}

\title{Twisted correlations of the divisor function via discrete averages of $\SL(\R)$ Poincar\'e series}
\author{Lasse Grimmelt}
\address{Mathematical Institute, University of Oxford\\ Radcliffe Observatory Quarter, Woodstock Rd\\
Oxford OX2 6GG\\
UK}
\email{lasse.grimmelt@maths.ox.ac.uk}
\author{Jori Merikoski}
\address{Mathematical Institute, University of Oxford\\ Radcliffe Observatory Quarter, Woodstock Rd\\
Oxford OX2 6GG\\
UK}
\email{jori.merikoski@maths.ox.ac.uk}
\date{}

\begin{document}
\begin{abstract}
We prove a theorem that allows one to count solutions to determinant equations twisted by a periodic weight with high uniformity in the modulus. It is obtained by using spectral methods of $\SL(\R)$ automorphic forms to study Poincar\'e series over congruence subgroups. By keeping track of interactions between multiple orbits we get advantages over the widely used sums of Kloosterman sums techniques. We showcase this with applications to correlations of the divisor functions twisted by periodic functions and the fourth moment of Dirichlet $L$-functions on the critical line. 
\end{abstract}

\maketitle

\tableofcontents

\section{Introduction}
The \emph{Kloostermania} techniques introduced by Deshouillers and Iwaniec \cite{deshoullieriwaniec}, building on the works of Bruggeman \cite{brug} and Kuznetsov \cite{kuz}, have produced remarkable applications. This is evidenced by \cite{BFI,motohashi,petrowyoung,young}, among many other works. In this paper we introduce a new approach that allows us to better exploit certain geometric and arithmetic aspects. Inspired by the works of Duke, Friedlander, and Iwaniec \cite{DFIprimes} as well as Bruggeman and Motohashi \cite{BM}, our method focuses on the present large symmetry under the action of a congruence subgroup of $\SL(\Z)$.

For a concrete example, consider the determinant equation
\begin{align} \label{eq:determinanteqintro}
    ad - bc = 1 \quad \text{with divisibility conditions} \quad q_1 |b, \,\, q_2|d,
\end{align}
where $q_1,q_2$ are fixed nonzero integers and 
$a,b,c,d$ are integer variables. This can be seen as variant of a correlation sum of the divisor function $\sum_n d(n)d(n+1)$. The standard approach for counting the number of solutions up to a given height to \eqref{eq:determinanteqintro} is to apply the Poisson summation formula. This reduces matters to bounding sums of Kloosterman sums at the pair of cusps $(\infty,1/q_1)$ for the Hecke congruence subgroup $\Gamma_0(q_1 q_2)$. In many situations the black-box results developed by Deshouillers and Iwaniec \cite{deshoullieriwaniec} are sufficient for this task. 

The approach we take here is to more immediately make use of the underlying congruence subgroup symmetry. Arranging the solutions to \eqref{eq:determinanteqintro} into a set of matrices
\begin{align*}
  \mathcal{M} = \bigg \{ \mqty( a & b  \\c & d )\in \SL(\Z): \, q_1 | b, \, q_2 | d  \bigg \},
\end{align*}
this set has an action by the congruence subgroup
\begin{align}
   \Gamma =  \Gamma_2(q_1,q_2) := \bigg\{\mqty(a & b \\c & d) \in \SL(\Z): \, q_1 |b, \, \, q_2 |c  \bigg\},
\end{align}
that is, for all $\gamma \in \Gamma$ we have $\gamma \mathcal{M} = \mathcal{M}$. This allows us to use spectral methods directly to count the number of solutions to \eqref{eq:determinanteqintro}, instead of stepping via Kloosterman sums. The group $\Gamma_2(q_1,q_2)$ is a conjugate of the standard Hecke congruence subgroup $\Gamma_0(q_1q_2) = \Gamma_2(1,q_1q_2)$.

The direct approach turns out to have an advantage in the case that the action of $\Gamma$ on a set $\mathcal{M}$ has a large number of orbits. This naturally occurs if one replaces the congruence conditions $q_1 | b, \, q_2 | d$ by a more general left $\Gamma$-invariant function $\alpha(\smqty(a &b \\ c& d))$. In this case we are able to keep track of the distribution of $\Gamma \backslash \mathcal{M}$ inside the ambient space $\Gamma \backslash \SL(\R)$ as well as exploit cancellations in $\alpha$. We even allow $\mathcal{M} \subseteq \SL(\R)$ to have non-integer entries, provided that $\Gamma$ acts on $\mathcal{M}$ with finitely many orbits.

We now formalise concept of an invariant weight. For functions $f:\GL(\R) \to \C$ and $g,h \in \GL(\R)$ define the left translation operator $l_h: f(g) \mapsto f(h g )$.
\begin{definition}[Automorphic function]\label{def:groupaction}
For any subgroup $\Gamma$ of $\SL(\R)$ and for any left $\Gamma$-invariant subset $\mathcal{M} \subseteq \GL(\R)$ we define the set of $\Gamma$-automorphic functions on $\mathcal{M}$ by
\begin{align*} 
    \mathcal{A}(\Gamma \backslash \mathcal{M}):= \{\alpha:\mathcal{M} \to \C: \, l_\gamma \alpha = \alpha \,\, \text{for all} \,\, \gamma \in \Gamma\}.
\end{align*}
\end{definition}
Further, naturally in the context of an analytic counting problem we require that the variables are weighted with a smooth function. 
\begin{definition}[Smooth dyadically supported function]\label{def:SandS}
  For $n,J \in \Z_{>0}$, $\delta > 0,$ and $X_1,\dots,X_n \in \R_{>0}$  define the space of $J$ times differentiable dyadically supported functions. 
\begin{align*}
    C_\delta^J (X_1,\dots,X_n) := \{f\in& C^J(\R^n): \, f(x_1,\dots,x_n) \,\, \text{supported on} \,  \\
    &(|x_1|,\dots,|x_n|) \in [X_1,2X_1] \times \cdots \times [X_n,2X_n], \\
    &  \|\partial_{x_1}^{J_1} \cdots \partial_{x_n}^{J_n} f\|_\infty \leq  \prod_{i \leq n} (\delta  X_i)^{-J_i} \quad \text{for all} \quad 0 \leq J_1+\cdots+J_n  \leq J  \}.
\end{align*}
\end{definition}
With this we can state a simplified version of our central result that already captures the core new features. We will later give more technical and stronger versions, see Theorems  \ref{thm:Gmainblackbox} and \ref{thm:twisteddetwbound}. 
\begin{theorem} \label{thm:mainblackbox}
Let $\mathcal{M} \subseteq \SL(\R)$,\, $q_1,q_2\in \Z_{>0}$ and denote $\Gamma = \Gamma_2(q_1,q_2)$. Let $\d g$ denote integration with respect to the Haar measure of the group $\SL(\R)$. Assume that $\mathcal{M}$ is left 
$\Gamma$-invariant with finitely many orbits $\Gamma \backslash \mathcal{M}$ and let $\alpha \in \mathcal{A}(\Gamma \backslash \mathcal{M})$.  Let $A,C,D,\delta > 0$  with $AD > \delta$ and let $f \in C^{7}_\delta (A,C,D)$. Denote $F:\SL(\R) \to \C, \, F(\smqty(a & b \\ c& d)):= f(a,c,d)$. Then for any $\eps > 0$ we have
\begin{align*}
   \bigg|  \sum_{g \in \mathcal{M}}  \alpha(g) F(g)  -  \frac{1}{|\Gamma \backslash \SL(\R) |} \sum_{\tau \in \Gamma \backslash \mathcal{M}} \alpha(\tau)  \int_{\SL(\R)} F(g) \, \d g \bigg|  \ll_\eps \delta^{-O(1)} Z^\eps  \sqrt{ AD \, \mathcal{K} \, \mathcal{R}},
\end{align*}
where $Z = \max\{A^{\pm 1},C^{\pm 1},D^{\pm 1}\}$,
\begin{align*}
\mathcal{K}& = \sum_{\substack{g =\smqty(a & b \\c & d) \in \mathcal{M}^{-1} \mathcal{M} \\  |a|+|b|C/D+|c|D/C+|d| \,\leq 6}}\Bigl|\sum_{ \substack{\tau\in \Gamma \backslash \mathcal{M} 
 \\ \tau g \in \mathcal{M}}} \alpha(\tau) \overline{\alpha(\tau g)} \Bigr|, \quad \quad\text{and} \\
\mathcal{R}& =    \left(1+ (AD)^{2\theta}\bigg( \frac{ C}{ A q_2}\bigg)^{2\theta} \right)\left(1+\bigg(\frac{C}{A q_2}\bigg)^{1-2\theta} \right) + \frac{A}{Cq_1}, \quad \theta=7/64.
\end{align*}
\end{theorem}
The error term in Theorem \ref{thm:mainblackbox} consist of three parts $  \sqrt{AD} $, $\mathcal{K}$, and $\mathcal{R}$.  The first and the third term are directly controlled by the ranges of the variables $a,c,d$. The term $\sqrt{AD}$  corresponds to square-root cancellation in terms of the height. 

The term $\mathcal{R}$ depends on the vertical skewness $C/A$ of the ranges as well as spectral data of automorphic forms for the group $\Gamma_0(q_1q_2) $,  via the spectral large sieve. In particular, denoting $q=q_1q_2$, it depends on $\theta_q := \max\{0,\Re(\sqrt{1/4-\lambda_1(q)})\}$ with $\lambda_1(q)$ denoting the smallest positive eigenvalue for $\Gamma_0(q)$. In the statement we have utilised the bound $\theta_q \leq 7/64$  of Kim and Sarnak \cite{KimSarnak}. Assuming the Selberg eigenvalue conjecture $\theta_q=0$ the term $\mathcal{R}$ would simplify to a more symmetrical
\begin{align*}
    \mathcal{R} \ll  1 + \frac{C}{A q_2} +\frac{A}{C q_1}.
\end{align*}

The term $\mathcal{K}$ is different in that it depends more intricately on the underlying geometry of $\Gamma \backslash \mathcal{M}$ and arithmetic of the weight $\alpha$. For this reason, an estimate for $\mathcal{K}$ needs to be provided separately for each intended application. This should not be seen as a defect -- it is what gives the strength of our method for new applications. In general terms, $\mathcal{K}$ depends on the horizontal skewness $C/D$ and on how well-spaced the set of orbits $\Gamma \backslash \mathcal{M}$ (weighted with $\alpha$) is within the fundamental domain $\Gamma \backslash SL_{2}(\R)$. For $C\asymp D$ and $\alpha$ bounded we expect that
 \begin{align*}
     \mathcal{K} \ll |\Gamma \backslash \mathcal{M}| + \frac{|\Gamma \backslash \mathcal{M}|^2}{q_1q_2},
 \end{align*}
provided that $\Gamma \backslash \mathcal{M}$ does not form clusters in the fundamental domain. Furthermore, even if $C/D$ or $D/C$ are large, we can hope to exploit oscillations or sparseness in the weight $\alpha$ to recover this heuristic upper bound. We will sketch this for a concrete example in Section \ref{sec:applsketch} and carry out rigorously in Sections \ref{sec:proofapplperiodicdivisor} and \ref{sec:proofapplLfunc}. 

\subsection{Applications} \label{subsec:appli}
We obtain the following applications, which follow from a more technical version of Theorem \ref{thm:mainblackbox}, see Theorem \ref{thm:twisteddetwbound}. Theorem \ref{thm:mainblackbox} would suffice for these in the range $|h| \ll 1$.  For simplicity we restrict to the case of a square-free modulus $q$ and $\gcd(h,q)=1$, in the general case for some applications there may be additional losses depending on $q/\mathrm{rad}(q)$ and $\gcd(h,q^\infty)$.  

As a first application we consider a general statement concerning sums over $d(n)d(n+h)t(n)$, where $t(n)$ is a periodic function. Define for square-free $q$ and for $h,r\in \Z$  with $\gcd(h,q)=1$
\begin{align*}
    \omega(r,h;q) &:= \prod_{p| q}  \omega(r,h;p), \\
        \omega(r,h;p) &:= \begin{cases}
        \frac{p-1}{p(p+1)}, \quad  &\text{if } r(r+h) \not \equiv
        0 \pmod{q} \\
           \frac{2p-1}{p(p+1)}, \quad  &\text{if }  r(r+h) \equiv
        0 \pmod{q}
    \end{cases}
\end{align*}
and note that $r\mapsto   \omega(r,h;q)$ defines a probability measure on $\Z/q\Z$.
\begin{theorem}[Divisor correlations with a periodic weight] \label{thm:divisorperiodic}
    Let $h,q\in \Z$ with $q$ square-free and $\gcd(h,q)=1$. Let $X,\eta > 0$ and suppose that $-X/2\leq h  \leq X^{1+\eta}$.  Let $G:\R\to\C$ be a smooth  function supported on $[1,2]$ that satisfies $G^{(j)} \ll_j X^{j \eta }$ for all $j \geq 0$. Let $t:\Z/q\Z \to \C$ be a function. Then for some binary quadratic polynomial $P_h$, only depending on $h$, we have
    \begin{align*}
        \sum_{\substack{ n }} G(n/X) d(n) d(n+h)t(n) =&\sum_{r\pmod{q}} t(r)\omega(r,h;q) \int G(u/X)  P_h( \log u,\log(u+h)) \,\d u
        \\
        &\hspace{20pt}+O\bigg( X^{1/2+O(\eta)} q^{1/2}\|t\|_2   (|h|^{\theta}+(X/q)^{\theta})  \bigg),
    \end{align*}
where  $\|t\|_2^2 = \sum_{r \pmod{q}} |t(r)|^2$ and $\theta= 7/64$ denotes the best approximation to the Ramanujan-Peterson and Selberg eigenvalue conjectures.
\end{theorem}
The polynomial $P_h$ can be explicitly computed, see \cite{Motohashidivisor}, for instance. We now state three corollaries that are all immediate consequences of Theorem \ref{thm:divisorperiodic}. The first corollary concerns divisor correlations in a fixed arithmetic progression. The expected main term is given by
\begin{align}\label{eq:MTdivisorap}
    \mathrm{MT}_G(X,r,h,q)= \omega(r,h;q) \int G(u/X) P_h( \log u,\log(u+h)) \,\d u
\end{align}
where $P_h$ is the quadratic polynomial as in Theorem \ref{thm:divisorperiodic}.
\begin{corollary} \label{cor:fixedap}
Assume the requirements of Theorem \ref{thm:divisorperiodic}, let $r\in \Z$ and $\mathrm{MT}_G(X,r,h,q)$ given by \eqref{eq:MTdivisorap}. Then we have
    \begin{align*}
        \sum_{\substack{ n\equiv r \pmod{q}}} G(n/X) d(n) d(n+h)=  \mathrm{MT}_G(X,r,h,q)+ O\bigg( X^{1/2+O(\eta)} q^{1/2}  (|h|^{\theta}+(X/q)^{\theta})  \bigg).
    \end{align*}
\end{corollary}
As far as we aware, this is the first result of this type that goes beyond the range $q\leq X^{1/6-\eps}$ that can be obtained by opening one of the divisor functions and using the fact that $d(n)$ is well distributed in arithmetic progressions to moduli at most $X^{2/3-\eps}$. Assuming $\theta=0$ we obtain the range $q \leq X^{1/3-\eps}$.  Unconditionally, the result is non-trivial if
\begin{align*}
   q & \leq X^{\frac{1-2\theta}{3-2\theta } - \eps} , \quad \quad  |h| \leq X/q \\
  q  &\leq X^{1/3-\eps} |h|^{ -2\theta/3} \quad \quad X/q < |h| \leq X^{1+\eta}.
\end{align*}
Note that by $\theta \leq 7/64$ \cite{KimSarnak} we have $X^{\frac{1-2\theta}{3-2\theta }}\geq X^{0.281\ldots} $ and uniformly in the considered $h$ range we have $X^{1/3} |h|^{ -2\theta/3} \geq  X^{0.261\ldots + O(\eta)}$. In particular, in this range we have 
\begin{align*}
        \sum_{\substack{ n \leq X \\ n\equiv r \pmod{q}}}  d(n) d(n+h)=   \omega(r,h;q)  \int_1^X   P_h( \log u,\log(u+h)) \,\d u(1+O(X^{-\eps'})).
    \end{align*}

If we consider an average over congruences to a fixed modulus, we can do better, as the following corollary shows. This is again an immediate consequence of Theorem \ref{thm:divisorperiodic}.
\begin{corollary}\label{cor:averageap} Assume the requirements of Theorem \ref{thm:divisorperiodic} and let $ \mathrm{MT}_G(X,r,h,q)$ given by \eqref{eq:MTdivisorap}. Then we have
    \begin{align*}
       \sum_{r(q)}  \left| \sum_{\substack{ n\equiv r \pmod{q}}}G(n/X)d(n) d(n+h)- \mathrm{MT}_G(X,r,h,q)\right| \ll X^{1/2+O(\eta)} q\bigg(|h|^{\theta}+(X/q)^{\theta}  \bigg).
    \end{align*}
\end{corollary}

Assuming $\theta=0$ this result is non-trivial as long as $q\leq X^{1/2-\eps}$. Similarly as in the case of a fixed congruence, possible exceptional eigenvalues and larger choices of $h$ will reduce the range of $q$. Unconditional we still get a non-trivial result for 
\begin{align*}
   q & \leq X^{\frac{1-2\theta}{2-2\theta } - \eps} , \quad \quad  |h| \leq X/q \\
  q  &\leq X^{1/2-\eps} |h|^{ -\theta} , \quad \quad X/q < |h| \leq X^{1+\eta}.
\end{align*}
Here, again by $\theta \leq 7/64$ \cite{KimSarnak}, we have $X^{\frac{1-2\theta}{2-2\theta }}\geq X^{0.438\ldots} $ and uniformly in the considered $h$ range we have $X\leq X^{1/2} |h|^{ -\theta} \geq  X^{0.390\ldots + O(\eta)}$.

As final direct consequence of Theorem \ref{thm:divisorperiodic} we consider twists by an exponential phase with denominator $q$.
\begin{corollary} \label{cor:minorarc}
   Assume the requirements of Theorem \ref{thm:divisorperiodic} and let $r\in \Z$ with $\gcd(r,q)=1$. Then 
    \begin{align*}
        \sum_{n} G(n/X)d(n)d(n+h) e(n r/q) \ll \frac{d(q)}{q}X+X^{1/2+O(\eta)} q\bigg(|h|^{\theta}+(X/q)^{\theta}\bigg).
    \end{align*}
\end{corollary}

In a different direction, we get new results for the fourth moment of Dirichlet $L$-functions along the critical line. We first consider the case of two different characters.
\begin{theorem} \label{cor:Lfunction}
Assume that $q_1,q_2$ are square-free and let $\chi_1,\chi_2$ be different primitive Dirichlet characters to moduli $q_1,q_2$. Let $\eta>0$ and $\omega:(0,\infty)\to \C$ 
be a smooth compactly supported function with $\omega^{(j)} \ll_{j} T^{ j \eta }$ for all $j \geq 0$. Then for some quadratic polynomial $P_{\chi_1,\chi_2}$ we have 
    \begin{align*}
        \int_\R |L(1/2+it,\chi_1) |^2|L(1/2+it,\chi_2)|^2\omega(t/T) \d t= &\int_\R P_{\chi_1,\chi_2}(\log t) \omega(t/T) \d t  \\
        &+O\bigg(  \,(q_1q_2)^{3/4} \, T^{1/2+\theta+O(\eta)} \bigg).
    \end{align*}
\end{theorem} 
If we instead have a single character, then we can do better than just setting $q_1=q_2$ in the previous statement.
\begin{theorem} \label{cor:Lfunctiononechar}
Assume that $q$ is square-free and let $\chi$ be a primitive Dirichlet character to the modulus $q$. Let $\eta>0$ and $\omega:(0,\infty)\to \C$ 
be a smooth compactly supported function with $\omega^{(j)} \ll_{j} T^{ j \eta }$ for all $j \geq 0$. Then for some quartic polynomial $P_{\chi}$ we have
    \begin{align*}
        \int_\R |L(1/2+it,\chi) |^4\omega(t/T) \d t= &\int_\R P_{\chi}(\log t) \omega(t/T) \d t  +O\bigg(    \, q \, T^{1/2+\theta+ O(\eta)} \bigg).
    \end{align*}
\end{theorem}
These two results should be compared to \cite[Theorem 1.4, Theorem 1.2]{berke} where Topacogullari obtained error terms of respective strength (for square-free moduli)
\begin{align*}
    &O((q_1+q_2)^{1/2}(q_1q_2)^{3/4-3\theta /2}T^{1/2+\theta}) \quad \text{and} \quad O(q^{2-3\theta}T^{1/2+\theta}).
\end{align*}
In particular, in the case of a single character we essentially save a full factor of $q$ in the error term. See also the work of Kaneko \cite{Kaneko_2022} who obtained a full Motohashi-type spectral expansion for one character to a prime modulus, but did not improve upon the error term in the setting of \cite[Theorem 1.2]{berke}. Our main technical result (Theorem \ref{thm:twisteddetwbound}) is also applicable to a product of four $L$-functions with four different characters.

\subsection{Proof sketch of Theorem \ref{thm:mainblackbox}} As mentioned above, our argument borrows ideas from two main sources, the works of Duke, Friedlander, and Iwaniec \cite{DFIprimes} as well as Bruggeman and Motohashi \cite{BM}. More precisely, we start as in \cite{BM} and parametrize a given determinant problem over $\SL(\R)$ with the help of a congruence subgroup. In this way we get a type of  Poincar\'e series that is suitable for spectral expansion. Afterwards, in the spirit of \cite{DFIprimes}, we apply the Cauchy-Schwarz inequality and reverse the spectral expansion. The main new idea is to apply this on average over multiple orbits and to keep track of this data throughout the argument.

Consider the set up in Theorem \ref{thm:mainblackbox} with $\alpha \in \mathcal{A}(\Gamma \backslash \mathcal{M})$ and $F(\smqty(a & b \\ c& d)) = f(a,c,d)$. For simplicity let us assume that $\Gamma= \Gamma_0(q)$. The general case is reduced to this by conjugation by $\a[q_1] = \smqty(\sqrt{q_1} & \\ & 1/\sqrt{q_1})$. Using the $\Gamma$-invariance of $\alpha$ we can write
\begin{align*}
    \sum_{g \in \mathcal{M}} \alpha(g) F(g) = \sum_{\tau \in \Gamma \backslash \mathcal{M}} \alpha(\tau) P_f(\tau), \quad \quad P_f(\tau):= \sum_{\gamma \in \Gamma} F(\gamma \tau).
\end{align*}
Here $P_f(\tau)$ is a type of Poincar\'e series on $\SL(\R)$, which is smooth and invariant under the left action by $\Gamma$.  Hence, applying the spectral expansion on $\Gamma \backslash \SL(\R)$ we get (omitting the continuous spectrum for brevity)
\begin{align*}
   P_f(\tau) = \frac{1}{|\Gamma \backslash \SL(\R)|} \langle P_f,1 \rangle_{\Gamma \backslash \SL(\R)} + \sum_{\varphi \in \B(q)} \langle P_f, \varphi \rangle_{\Gamma \backslash \SL(\R)} \varphi(\tau), 
\end{align*}
where $\B(q)$ denotes an orthonormal basis of automorphic forms on $\Gamma \backslash \SL(\R)$ and the inner product is given by
\begin{align*}
\langle f_1, f_2 \rangle_{\Gamma \backslash \SL(\R)} := \int_{\Gamma \backslash \SL(\R)} f_1(g) \overline{f_2}(g) \, \d g.
\end{align*}

The projection onto the constant function produces the main term in Theorem \ref{thm:mainblackbox} so that it remains to bound
\begin{align}\label{eq:introspecexp}
  \sum_{\tau \in  \Gamma \backslash \mathcal{M} } \alpha(\tau)  \sum_{\varphi \in \B(q)} \langle P_f, \varphi \rangle_{\Gamma \backslash \SL(\R)} \varphi(\tau) \leq    \sum_{\varphi \in \B(q)} |\langle P_f, \varphi \rangle_{\Gamma \backslash \SL(\R)}| \bigg| \sum_{\tau \in  \Gamma \backslash \mathcal{M} } \alpha(\tau )\varphi(\tau)   \bigg|.
\end{align}
Importantly, we here keep the sum over $\tau \in  \Gamma \backslash \mathcal{M}$ alive in the application of the triangle inequality. In the basis $\B(q)$ we sum over roughly $q$ many elements and using the triangle inequality destroys any hope of detecting cancellation between these harmonics. By exploiting the averaging over $\tau$ we are able to gain back some of this loss.

The inner products $\langle P_f, \varphi \rangle_{\Gamma \backslash \SL(\R)}$ are computed by unfolding and inserting a Fourier expansion for the automorphic form $\varphi$. This yields a sum roughly of the  shape
\begin{align}\label{eq:introinnerprod}
 \sqrt{AD} \sum_{\varphi \in \B(q)} h(\varphi)\bigg|  \frac{1}{\sqrt{C/A}} \sum_{n \leq C/A} \varrho_{\varphi}(n) \bigg| \bigg| \sum_{\tau \in  \Gamma \backslash \mathcal{M} } \alpha(\tau )\varphi(\tau)   \bigg|,
\end{align}
where $\varrho_{\varphi}(n)$ denote the Fourier coefficients of $\varphi$ and $h(\varphi)$ is a non-negative function which decays quickly in  the spectral parameter. Note that by $L^2$-normalization both $\varrho_{\varphi}(n)$ and $\varphi(\tau)$ are typically of size  $\approx q^{-1/2}$ in absolute value. 

Plugging this into \eqref{eq:introspecexp} and applying the Cauchy-Schwarz inequality we get
\begin{align*}
   \eqref{eq:introspecexp} \leq \sqrt{AD \, \mathcal{K}\, \mathcal{R}},
\end{align*}
where
\begin{align*}
    \mathcal{R} = \frac{1}{C/A} \sum_{\varphi \in \B(q)} h(\varphi)\bigg|   \sum_{n \leq C/A} \varrho_{\varphi}(n) \bigg|^2 
\end{align*}
and 
\begin{align}\label{eq:introKfirst} 
    \mathcal{K} =\sum_{\varphi \in \B(q)}   h(\varphi)\bigg| \sum_{\tau \in  \Gamma \backslash \mathcal{M} } \alpha(\tau )\varphi(\tau)   \bigg|^2 =  \sum_{\tau_1,\tau_2 \in  \Gamma \backslash \mathcal{M} } \alpha(\tau_1 ) \overline{\alpha(\tau_2)} \sum_{\varphi \in \B(q)}  h(\varphi) \varphi(\tau_1) \overline{\varphi(\tau_2) }. 
\end{align}
The term $\mathcal{R}$ may be bounded by the spectral large sieve of Deshouillers and Iwaniec \cite{deshoullieriwaniec}. 

To approach $\mathcal{K}$ we identify it as the spectral expansion of an automorphic kernel on $\SL(\R)$ (also known as pre-trace formula). Indeed, for a suitable function $h(\varphi)$ this gives the existence of a smooth function $k:\SL(\R) \to \C$ supported on a small neighborhood of the identity, such that
\begin{align*}
   \sum_{\varphi \in \B(q)} h(\varphi) \varphi(\tau_1) \overline{\varphi(\tau_2) } = \sum_{\gamma \in \Gamma} k(\tau_2^{-1} \gamma \tau_1).
\end{align*}
Thus, denoting $g=\tau_2^{-1} \gamma \tau_1$ we get
\begin{align*}
    \mathcal{K} &=  \sum_{\gamma \in \Gamma} \sum_{\tau_1,\tau_2 \in  \Gamma \backslash \mathcal{M} } \alpha(\tau_1 ) \overline{\alpha(\tau_2)}  k(\tau_2^{-1} \gamma \tau_1)  \\
   \nonumber  &= \sum_{g \in \mathcal{M}^{-1} \mathcal{M}}  k(g)\sum_{\substack{\tau_1,\tau_2 \in  \Gamma \backslash \mathcal{M}  \\ \tau_2 g \tau_1^{-1} \in \Gamma }} \alpha(\tau_1 ) \overline{\alpha(\tau_2)}   \\
   \nonumber  & \leq  \sum_{g \in \mathcal{M}^{-1} \mathcal{M}}  |k(g)| \bigg|\sum_{\substack{\tau \in \Gamma \backslash \mathcal{M} \\ \tau g \in \mathcal{M} }} \alpha(\tau g ) \overline{\alpha(\tau)} \bigg|,
\end{align*}
which is almost of the shape as proposed in Theorem \ref{thm:mainblackbox}. 

However, we have oversimplified in the above sketch in one crucial aspect. Namely, in the case that $\max\{ C/D,D/C\}$ is large, the smooth weight $F$  lives on a narrow part of the $\mathrm{K}$-component of the Iwasawa decomposition $\SL(\R)= \mathrm{N}\mathrm{A}\mathrm{K}$. Then the spectrum blows up in terms of the $\mathrm{K}$ types and the sum over $\varphi \in \B(q)$ actually runs over roughly $q \max\{ C/D,D/C\}$ many elements. This can be resolved by first rescaling the smooth weight 
\begin{align*}
    P_f(\tau) = \sum_{\gamma \in \Gamma} F(\gamma \tau) =   \sum_{\gamma \in \Gamma} F(\gamma \tau \a[D/C] \a [C/D]  ) =   P_{f_0}(\tau \a[D/C]),
\end{align*}
where 
\begin{align*}
    f_0(g) = f(g \a [C/D])  \in C_\delta^{7}(A_0,C_0,D_0)
\end{align*}
with
\begin{align*}
    C_0 = D_0 = \sqrt{CD}, \quad A_0D_0=AD, \quad \text{and} \quad C_0/A_0 = C/A.
\end{align*}
Alternatively, this can be seen as applying the spectral expansion to the original $P_f$ after the change of basis $\varphi(\tau) \mapsto \varphi(\tau \a[D/C])$. 

Then by the above argument we need to bound
\begin{align} \label{eq:comparetokuz,tauscalingmatrix}
\begin{split}
  \sqrt{AD} &\sum_{\varphi \in \B(q)} h(\varphi) \bigg|  \frac{1}{\sqrt{C/A}} \sum_{n \leq C/A} \varrho_{\varphi}(n) \bigg| \bigg| \sum_{\tau \in  \Gamma \backslash \mathcal{M} } \alpha(\tau )\varphi(\tau \a[D/C])   \bigg|   
\end{split}
\end{align}
and we arrive at
\begin{align*}
     \mathcal{K} &=  \sum_{\gamma \in \Gamma} \sum_{\tau_1,\tau_2 \in  \Gamma \backslash \mathcal{M} } \alpha(\tau_1 ) \overline{\alpha(\tau_2)}  k( \a[D/C]^{-1}\tau_2^{-1} \gamma \tau_1 \a[D/C]) \\
     & \leq   \sum_{g \in \mathcal{M}^{-1} \mathcal{M}}  |k( \a[C/D] g \a[D/C])| \bigg|\sum_{\substack{\tau \in \Gamma \backslash \mathcal{M} \\ \tau g \in \mathcal{M} }} \alpha(\tau g ) \overline{\alpha(\tau)} \bigg|.
\end{align*}
Here we can show that for $g= \smqty(a & b \\ c& d)$ the weight $|k( \a[C/D] g \a[D/C])|$ vanishes outside 
\begin{align*}
    |a|  + |b| C/D +|c| D/C + |d| \leq 6.
\end{align*}
This gives the shape of $\mathcal{K}$ in Theorem \ref{thm:mainblackbox}.

\subsection{Sketch of applications} \label{sec:applsketch}
To illustrate how the method may be applied, we sketch the proof of Corollary \ref{cor:fixedap} in the case $h=1$, $\gcd(r(r+1),q)=1$, using Theorem \ref{thm:mainblackbox} and assuming $\theta=0$. 

We have the following blueprint for applying our main theorem.
\begin{enumerate}
    \item Identify the largest symmetry by a congruence subgroup $\Gamma$ and the orbits $\Gamma \backslash \mathcal{M}$.
    \item Handle very skewed ranges trivially.
     \item Compute an upper bound for $\mathcal{K}$.
    \item Apply theorem.
\end{enumerate}

Consider the case of  Corollary \ref{cor:fixedap} with $h=1$. We write
\begin{align*}
    d(n) = \sum_{n=bc} 1, \quad \quad \quad d(n+1) = \sum_{n=ad} 1
\end{align*}
 and insert a smooth dyadic partition for $a,c,d$ via $f \in C^{7}_{\delta}(A,C,D)$ for some $A,C,D \ll X$ with $AD \asymp X$ and $\delta \asymp \eta$. Denote also $B:= X/C$.  This reduces matter to evaluating 
\begin{align*}
    \sum_{\substack{ad-bc = 1 \\ bc \equiv r \pmod{q} }} f(a,c,d).
\end{align*} 
Here we take
\begin{align*}
    \mathcal{M} = \SL(\Z), \quad \alpha(\smqty(a & b \\ c& d)) = \mathbf{1}_{bc \equiv r \pmod{q}}.
\end{align*}
We then observe that $\alpha$ is invariant under the action of $\Gamma = \Gamma_2(q,q)$, since for any $ \smqty( a_0 & b_0 q \\  c_0 q & d_0) \in \Gamma_2(q,q)$ we have that
\begin{align*}
    \mqty( a_0 & b_0 q \\  c_0 q & d_0)  \mqty(a & b \\ c& d) = \mqty(\ast & a_0 b + q b_0 d \\ d_0 c + q c_0 a & \ast) 
\end{align*}
maps 
\begin{align*}
   bc \mapsto (a_0 b + q b_0 d) (d_0 c + q c_0 a) \equiv a_0 d_0 bc \equiv bc \pmod{q}.
\end{align*} 
Note that $\Gamma$ is in some sense too sparse a group. The original problem lives on a set of density roughly $1/q$ whereas $\Gamma$ has density roughly $1/q^2$ in $\SL(\Z)$. Consequently $\alpha$ has density roughly $1/q$ on the set of orbits $\Gamma \backslash \SL(\Z)$.

By symmetry (swapping $ad \leftrightarrow bc$, $a \leftrightarrow d$, $b \leftrightarrow c$) we may assume that 
\begin{align*}
    C \leq D \leq A \leq B.
\end{align*}
The very skewed ranges where
\begin{align*}
    \max \{ A,D \} \geq q C
\end{align*}
may be counted trivially (Poisson summation). In the complementary range we have
\begin{align}
     1\leq \frac{A}{C}, \frac{D}{C} \leq q.
\end{align}

To apply Theorem \ref{thm:mainblackbox} we need an upper bound for
\begin{align*}
  \mathcal{K} =   \sum_{\substack{g =\smqty(a & b \\c & d) \in \SL(\Z) \\  |a|+|b|C/D+|c|D/C+|d| \,\leq 6}}\Bigl|\sum_{ \substack{\tau\in \Gamma \backslash  \SL(\Z) 
 }} \alpha(\tau) \overline{\alpha(\tau g)} \Bigr|.
\end{align*}
The contribution from $g = I$ is bounded by
\begin{align*}
    \sum_{\tau \in \Gamma \backslash  \SL(\Z) } |\alpha(\tau)|^2 \ll q 
\end{align*}
and we can never hope to get a better bound. The contribution from $ad \neq 1$ satisfies a similar bound $\ll q$ by Cauchy-Schwarz and using the divisor bound for $c |ad-1$. Finally, since $1 \leq D/C \leq q$ and since for generic $g \in \SL(\Z)$ the function $\tau \mapsto \alpha(\tau) \overline{\alpha(\tau g)}$ has density $1/q \cdot 1/q = 1/q^2$, we have
\begin{align*}
    \sum_{\substack{g =\smqty(\pm 1 & b \\0 & \pm 1) \in \SL(\Z) \\ 0< |b| \,\leq 6 D/C}}\Bigl|\sum_{ \substack{\tau\in \Gamma \backslash  \SL(\Z) 
 }} \alpha(\tau) \overline{\alpha(\tau g)} \Bigr| \ll  D/C \ll  q,
\end{align*}
so that $\mathcal{K} \ll  q$. In fact, it is easy to check that for $\gcd(r(r+1),q)=1$ we have  
\begin{align*}
    |\sum_{ \substack{\tau\in \Gamma \backslash  \SL(\Z) 
 }} \alpha(\tau) \overline{\alpha(\tau g)} |  \ll \gcd(b,q), \quad \quad g =\mqty(\pm 1 & b \\0 & \pm 1).
\end{align*}

Using $1 <A/C \leq q$ and our assumption that  $\theta=0$, we get
\begin{align*}
    \mathcal{ R} \ll 1 + \frac{C}{Aq} + \frac{A}{Cq} \ll 1.
\end{align*}
Thus, applying Theorem \ref{thm:mainblackbox} we get a main term of size roughly $X/q$ and an error term that is bounded by
\begin{align*}
    \sqrt{AD \, \mathcal{K} \, \mathcal{R}} \ll X^{1/2} q^{1/2}.
\end{align*}
This is dominated by the main term as long as $q$ is a bit smaller than $X^{1/3}$.

The proof of the more general Theorem \ref{thm:divisorperiodic} is a  bit more involved. One has to first decompose $t(n)$ into functions $t^{\flat}(n;q_0),q_0|q$, which are \emph{balanced} in the sense that they have mean value 0 over any coset of a non-trivial subgroup of $\Z/q_0\Z$. Applying the method separately for each $t^{\flat}(n;q_0)$ essentially gives us that $\mathcal{K} \ll q \|t\|_2^2$. Without this decomposition we would only be able to get $\mathcal{K} \ll q \|t\|_1^2$,  potentially losing a factor of $q$ for general $t(n)$. The key property that is leveraged here is that the weight $\alpha(\smqty(a & b \\ c &d)) = t(ad)$ is constant along certain one dimensional linear fibers in the two dimensional space $\Gamma_2(q,q) \backslash \SL(\Z)$.

For Theorem \ref{cor:Lfunction} we apply the approximate functional equation,  use the group $\Gamma_2(q_1,q_2)$, and check that
\begin{align*}
    \alpha(\smqty(a &b \\ c& d)) = \chi_1(a) \overline{\chi_1}(b) \overline{\chi_2}(c) \chi_2(d)
\end{align*}
is left-invariant. In fact, here we average also over the determinant
\begin{align*}
    \sum_{h \sim H} \sum_{ad-bc=h}     \alpha(\smqty(a &b \\ c& d))f(a,c,d), \quad H \ll \frac{AD}{T}, \, T \ll AD,BC \ll T \sqrt{q_1q_2} .
\end{align*}
This and a good dependency on the size of $h$ in other applications is facilitated by the more technical Theorem \ref{thm:twisteddetwbound}. In this case a strong bound for $\mathcal{K}$ is a consequence of an elementary character sum bound, which captures the oscillations in $\alpha$ from the characters.

\subsection{Comparison to the literature}
Consider the determinant equation 
\begin{align*}
    \sum_{\substack{ad-bc = 1 \\ q |c}} f(a,c,d)
\end{align*}
with variables weighted by a smooth $f(a,c,d)$ supported around $(A,C,D)$. Applying Poisson summation to the variables $a$ and $d$ one obtains a main term and an error term roughly of the form
\begin{align*}
   \frac{AD}{C^2} \sum_{m \leq C/D} \sum_{n \leq C/A} \sum_{\substack{c \equiv 0 \pmod{q}}} F(c/C) S(m,n;c), \quad S(m,n;c) := \sum_{\substack{a,d \pmod{c} \\ad \equiv 1\pmod{c}}}e_c(ma+nd).
\end{align*}
By Kuznetsov's formula for $\Gamma_0(q)$, this is bounded by  a sum of the form (omitting the continuous spectrum) 
\begin{align*}
  \frac{AD}{C}   \sum_{\varphi \in \B(q)} h(\varphi)  \bigg| \sum_{n \leq C/A} \varrho_\varphi(n) \bigg| \bigg| \sum_{m \leq C/D} \varrho_\varphi(m) \bigg|. 
\end{align*}
This can essentially be recovered from \eqref{eq:comparetokuz,tauscalingmatrix}, since for $\tau = I$ we morally have by the Fourier expansion (in the regular spectrum)
\begin{align*}
    \varphi( \a [D/C])  \approx \sqrt{D/C} \sum_{m \leq C/D}  \varrho_\varphi(m).
\end{align*}
A more general form with Kloosterman sums at cusps may also be recovered when $\tau$ is a scaling matrix of a cusp. Therefore, for a fixed orbit $\tau$ our argument produces exactly the same results as applying \cite{deshoullieriwaniec}. It should be noted that our method is simpler to apply in the sense that we never actually need to consider cusps and scaling matrices.

In similar spirit, if after spectral expansion of the Poincar\'e series we followed further the strategy in the work of Bruggeman and Motohashi \cite{BM}, we would need to employ what they call the \emph{Kirillov scheme}. This could be done if the orbits were identified suitably with scaling matrices and would yield the identical result as the application of Kuznetsov's formula we just considered. Indeed, this essentially amounts to retracing the proof of Kuznetsov's formula as presented in the book by Cogdell and Piatetski-Shapiro \cite{CoPi}.

It should be mentioned that Musicantov  and Zehavi \cite{musiczeha} also consider a version of \cite{DFIprimes} on $\SL(\R)$, but the setting is otherwise similar to \cite{DFIprimes}.

\subsection{Further developments}

Our underlying principle of making direct use of the group symmetry of the considered problem appears to be quite general. We intend to work on extending the method for other groups besides congruence subgroups on $\SL(\Z)$.

In a joint project with Duker Lichtman we will apply the results of this paper to the distribution of primes in arithmetic progressions to large moduli. For instance, extending \cite{BFI}, we will consider  averages of primes in simultaneous arithmetic progressions
\begin{align*}
    \sum_{v \sim V} \eta_v \max_{\gcd(b,v)=1}\sum_{\substack{q \sim Q \\ \gcd(q,v)=1} } \lambda_q \bigg(  \sum_{\substack{n \leq x \\ n \equiv a \pmod{q} \\ n \equiv b \pmod{v}}} \Lambda(n) - \frac{x}{\varphi(qv)}\bigg),
\end{align*}
where $\lambda_q$ is a factorable weight. This type of problem is motivated by the Maynard-Tao sieve with potential applications to bounded gaps between primes \cite{maynardsieve,polymath}.

Our results are uniform in determinants up to $|h| \leq X^{1+\eta}$ with a small $\eta>0$. For practical applications this is of the same strength as the recent work of Assing, Blomer, and Li \cite{ABL}. The sums of Kloosterman sums techniques allow one to handle even larger determinants, for instance, see Meurman \cite{meurman}. While we feel that better uniformity in $h$ should be achievable with more work, this appears to be a delicate problem. Similarly, it should be possible to extend the method to handle more oscillatory weight functions $f$. 

In Theorem \ref{thm:divisorperiodic} we restrict to $\gcd(h,q)=1$ to obtain uniformity in $|h|$. For general periodic weights there seems to be an obstruction for $h$ which have a large common factor with $q$. Since in \eqref{eq:swaoheckeandT} we require that we can swap the order of the Hecke orbits and $T=\Gamma \backslash \SL(\Z)$, and the part $\gcd(h,q^\infty)$ of the Hecke orbits needs to be handled inside $\mathcal{K}$. This gives a worse dependency on $\gcd(h,q^\infty)$. It is not clear if this problem can be overcome.

\subsection{Notations and structure of the paper}

The structure of the paper is as follows. In the next two sections we summarise the necessary background material. The content is standard, but we have not found it gathered in the precise composition required for us. In Section \ref{sec:geomliegroup} we discuss coordinate systems, differential operators and integration for $G=\SL(\R)$. In Section \ref{sec:Spectraltheoryof} we import the spectral expansion of $L^2(\Gamma\backslash G,\chi)$ in terms of automorphic forms, consider their Fourier expansions, and recall Hecke operators. 

Our next task is to prove technical results that are required for the proofs of the main theorems. In Section \ref{sec:decayandfourier} we prove decay properties of certain inner products. Together with the Fourier expansion, this is part of the rigorous treatment of \eqref{eq:introinnerprod}. The following two sections consider the automorphic kernel that is essential in our treatment of the $\mathcal{K}$ term (see  \eqref{eq:introKfirst}). In Section \ref{sec:specexpauto} we give a proof of the spectral expansion of the kernel. In Section \ref{sec:meanvalueauto} we apply this to give a mean value theorem for automorphic forms, which is what we ultimately use to bound the $\mathcal{K}$ part.

At this point we have gathered all the background tools to state and prove the most technical version of our main result, Theorem \ref{thm:Gmainblackbox}. This is done in Section \ref{sec:discaverage}. Combining Theorem \ref{thm:Gmainblackbox} with the spectral large sieve that we import in Section \ref{sec:spectrallarge}, we prove Theorem \ref{thm:mainblackbox} in Section \ref{sec:proofmainblack}.

In Section \ref{section:twist} we state and prove Theorem \ref{thm:twisteddetwbound} which is the main technical result. Like Theorem \ref{thm:mainblackbox} it follows from Theorem \ref{thm:Gmainblackbox} and the spectral large-sieve, but it is much stronger with respect to the shift $h$.  

Finally, in the last two sections, we prove the applications stated in the introduction. In section \ref{sec:proofapplperiodicdivisor} we consider correlations of the divisor function with a periodic twist and show Theorem \ref{thm:divisorperiodic}. Section \ref{sec:proofapplLfunc} deals with the $L$-functions applications, Theorems \ref{cor:Lfunctiononechar} and \ref{cor:Lfunction}.

\begin{center}
 Dependency graph of results:
\end{center}

 \makebox[\textwidth]{%
\parbox{1.5\textwidth}{
\centering
   \begin{tikzpicture}[
    node distance=1cm and 1cm,
    result/.style={rectangle, rounded corners, draw, align=center, minimum height=20pt, minimum width=30pt, fill=white},
    >=Stealth % This sets the style for arrow tips to 'Stealth', which is a filled arrowhead
]

% Group A nodes
\node[result] (A1) {Prop. \ref{prop:specdecompL2}};
\node[result, below=of A1] (A2) {Prop. \ref{prop:spectralexpofkernel}};
\node[result, above=1cm of A1] (A3) {Prop. \ref{prop:fourierexpansionbasis}};

% Group B nodes
\node[result, right=of A1] (B1) {Prop. \ref{prop:boundbykernel}};
\node[result, above=2cm of B1] (B2) {Prop. \ref{prop:eisenstein0prop} };
\node[result, above=of B2] (B3) {Prop. \ref{prop:Jacquetinnerproduct}};
\node[result, below=of B1] (B4) {Prop. \ref{prop:speclasie}};
\node[result, below=of B4] (B5) {Prop. \ref{prop:exclasie}};

% Group C nodes
\node[result, right=of B1 ] (C2) {Cor. \ref{cor:Gmainblackbox}};
\node[result, above=of C2] (C1) {Thm. \ref{thm:Gmainblackbox}};

% Group D nodes
\node[result, right=of C1] (D1) {Thm. \ref{thm:mainblackbox}};
\node[result, right=of C2] (D2) {Thm. \ref{thm:twisteddetwbound}};

% Group E nodes
\node[result, right=of D1] (E1) {Thm. \ref{thm:divisorperiodic}};
\node[result, right=of D2] (E2) {Prop. \ref{prop:divisorperiodicweak}};

% Group F nodes
\node[result, right=of E2] (F1) {Cor. \ref{cor:minorarc}};
\node[result, above=of F1] (F2) {Cor. \ref{cor:averageap}};
\node[result, above=of F2] (F3) {Cor. \ref{cor:fixedap}};
\node[result, below=of F1] (F4) {Thm. \ref{cor:Lfunction}};
\node[result, below=of F4] (F5) {Thm. \ref{cor:Lfunctiononechar}};

% Edges
\draw[->] (E1) -- (F1);
\draw[->] (E1) -- (F2);
\draw[->] (E1) -- (F3);
\draw[->] (E2) -- (E1);
\draw[->] (D2) -- (E2);
\draw[->] (D2) -- (F4);
\draw[->] (D2) -- (F5);
\draw[->] (C1) -- (D1);
\draw[->] (C1) -- (D2);
\draw[->] (C2) -- (D1);
\draw[->] (C2) -- (D2);
\draw[->] (B4) -- (C2);
\draw[->] (B5) -- (C2);
\draw[->] (B1) -- (C1);
\draw[->] (B2) -- (C1);
\draw[->] (B3) -- (C1);
\draw[->] (A1) -- (A2);
\draw[->] (A2) -- (B1);
\draw[->] (A3) -- (B1);
\draw[->] (A1) -- (C1);
\draw[->] (A3) -- (C1);
\end{tikzpicture}
}}

\subsection{Acknowledgements}
We are grateful to James Maynard and Jared Duker Lichtman for helpful discussions. This project has
received funding from the European Research Council (ERC) under the European Union's Horizon 2020 research and innovation programme (grant agreement No 851318).

\section{Geometry of the Lie group $G=\SL(\R)$} \label{sec:geomliegroup}
In this section and the following section  we provide an overview of the background material that goes into Theorems \ref{thm:Gmainblackbox} and \ref{thm:twisteddetwbound}. We also refer the readers to the book of Bump \cite{bump1997automorphic} and the excellent survey of Motohashi \cite{motohashielements}.
\subsection{Coordinates}
We start by introducing two coordinate systems for the Lie group 
\begin{align*}
 G:=  \SL(\R)  = \bigg\{ \mqty(a&b \\ c& d) : a,b,c,d \in \R, ad-bc=1  \bigg\} 
\end{align*}
following mostly the notation in \cite{motohashielements}. The relevant subgroups for us are
\begin{align*}
    \mathrm{N} :=& \bigg\{ \mathrm{n}[x] =  \mqty(1&x \\ & 1) : x \in \R  \bigg\}  \\
      \mathrm{A} := & \bigg\{ \mathrm{a}[y]= \mqty(\sqrt{y }& \\ & 1/\sqrt{y}) : y \in (0,\infty)  \bigg\} \\
        \mathrm{K} := &\bigg\{ \mathrm{k}[\theta] =  \mqty(\cos \theta & \sin \theta \\ -\sin \theta & \cos \theta) : \theta \in \R / 2\pi \Z \bigg\}  
\end{align*}
We will refer to the groups $ \mathrm{N}, \mathrm{A}, \mathrm{K}$ and their elements $\n,\a,\k$ without further notice and these symbols are reserved for this purpose.

We then have the \emph{Iwasawa decomposition}
\begin{align*}
    G= \mathrm{N} \mathrm{A} \mathrm{K}  ,
\end{align*}
where the representation $g= \smqty(a& b \\ c & d) = \mathrm{n}[x] \mathrm{a}[y] \mathrm{k}[\theta]$ is uniquely defined for $\theta \in [0,2\pi)$ via
\begin{align} \label{eq:matrixtoiwasawa}
   x= \frac{ac+bd}{c^2+d^2},\quad y= \frac{1}{c^2+d^2}, \quad \theta = \arctan \bigg(-\frac{c}{d}\bigg)
\end{align}
and 
\begin{align} \label{eq:iwasawatomatrix}
 \mathrm{n}[x] \mathrm{a}[y] \mathrm{k}[\theta] = \mqty( \sqrt{y} \cos \theta  - \frac{x}{\sqrt{y}} \sin \theta   &  \sqrt{y} \sin \theta  + \frac{x}{\sqrt{y}} \cos \theta \\ - \frac{1}{\sqrt{y}} \sin \theta &\frac{1}{\sqrt{y}} \cos \theta ).
\end{align}
The first two Iwasawa coordinates $(x,y)$ correspond to the upper half-plane $\mathbb{H}:=\{x+iy: y > 0\}$, which is isomorphic to the set of cosets $G/\mathrm{K}$, in the sense that the group action on $\mathbb{H}$ via $\smqty(a & b \\ c&d): z \mapsto \frac{az+b}{cz+d}$ is the matrix multiplication from the left on $G/\mathrm{K}$.

We also have the \emph{Cartan decomposition}
\begin{align*}
    G= \mathrm{K} \mathrm{A} \mathrm{K},
\end{align*}
where the representation $g = \mathrm{k}[\varphi] \mathrm{a}[e^{-\varrho}]\mathrm{k}[\vartheta] $ is unique provided that $\varrho > 0$ and $\varphi \in [0,\pi)$. We have
\begin{align} \label{eq:cartaninverserho}
    \mathrm{k}[\varphi] \mathrm{a}[e^{\varrho}]\mathrm{k}[\vartheta] = \mathrm{k}[\varphi + \pi/2] \mathrm{a}[e^{-\varrho}]\mathrm{k}[\vartheta - \pi/2] = \mathrm{k}[\varphi + \pi] \mathrm{a}[e^{\varrho}]\mathrm{k}[\vartheta - \pi].
\end{align}
The  Cartan coordinates are best described by the isomorphism of groups via conjugation by the Cayley matrix $\mathcal{C} = \smqty(1 & -i \\ 1 & i)$, which diagonalizes $\mathrm{K}$, 
\begin{align*}
    \SL(\R) &\cong \operatorname{SU}_{1,1}(\R) := \bigg\{ \mqty( \alpha & \beta \\ \overline{\beta} & \overline{\alpha}): \alpha,\beta \in \C, \, |\alpha|^2-|\beta|^2 = 1    \bigg\}, \\
    g &\mapsto \mathcal{C} g \mathcal{C}^{-1}.
\end{align*}
From this we can compute for $\smqty(a& b \\ c& d)= \mathrm{n}[x] \mathrm{a}[e^t] \mathrm{k}[\theta]=  \mathrm{k}[\varphi] \mathrm{a}[e^{-\varrho}]\mathrm{k}[\vartheta]  $

\begin{align} 
\label{eq:iwasawatocartan}  &\begin{cases} \alpha = \frac{1}{2}(a+d + i(b-c))=e^{i\theta}  (\cosh(t/2) + i e^{-t/2} x /2) = e^{i(\varphi + \vartheta)} \cosh(\varrho/2),  \\ \beta =\frac{1}{2}(a-d - i(b+c) ) =e^{-i\theta} (\sinh(t/2) - i e^{-t/2} x /2) =e^{i(\varphi - \vartheta)} \sinh(-\varrho/2) .\end{cases}
\end{align}
The first two Cartan coordinates $(\varphi,e^{-\varrho})$ correspond to the polar coordinates of the Poincar\'e disk model of the upper half-plane
\begin{align*}
    \mathbb{H} \cong \{ |z| < 1\}: z \mapsto \frac{z-i}{z+i}.
\end{align*}
Using \eqref{eq:cartaninverserho} and \eqref{eq:iwasawatocartan} we compute that for 
\begin{align*}
    \mathrm{a}[e^t] \mathrm{n}[x] &= \mathrm{k}[\varphi] \mathrm{a}[e^{-\varrho}] \mathrm{k}[\vartheta]
\end{align*}
we have
\begin{align} \label{eq:cartanAbel}
\begin{split}
\cosh \varrho = & \cosh t + e^{t} x^2  /2, \\ 
  e^{i (\varphi + \vartheta)} = & \frac{\cosh(t/2)}{\cosh(\varrho/2)} =\frac{\cosh(t/2) +i e^{t/2} x/2  }{|\cosh(t/2) +i e^{t/2} x/2 |},  \\
  e^{i (\varphi-\vartheta)} =& \frac{\sinh(t/2)}{\sinh(\varrho/2)} =    \frac{\sinh(t/2) -i e^{t/2} x/2  }{|\sinh(t/2) -i e^{t/2} x/2 |}.
  \end{split}
\end{align}
Denoting 
\begin{align*}
u(z,w):= \frac{|z-w|^2}{4 \Im(z) \Im(w)}, \quad z,w\in \mathbb{H},    
\end{align*}
 we have $\cosh(\varrho) = 2u(gi,i) +1$ and for $g= \smqty(a & b \\ c& d) \in G$
 \begin{align} \label{eq:ulowerbound}
   u(gi, i) = \tfrac{1}{4}  (a^2+b^2+c^2+d^2-2).
 \end{align}

\subsection{Differential operators}
Following the notation in \cite{motohashielements},
we introduce matrices
\begin{align*}
    \mathrm{X}_1 := \mqty(\,\,&1 \\ \,\, &),  \quad \mathrm{X}_2 := \mqty(1& \\ &-1), \quad  \mathrm{X}_3 := \mqty( & 1 \\ -1 & ),
\end{align*}
which are the infinitesimal generators for the subgroups
\begin{align*}
    \mathrm{N} = \{\exp( t \mathrm{X}_1): t \in \R\}, \,    \mathrm{A} = \{\exp( t \mathrm{X}_2): t \in \R\}, \, \mathrm{K} = \{\exp( t \mathrm{X}_3): t \in \R\}.
\end{align*}
We then define the right Lie differentials
\begin{align*}
    \mathrm{x_j}f(g) := \partial_t f(g \,\exp(t \mathrm{X}_j)) |_{t=0}.
\end{align*}
By construction these operators are left-invariant under the action of the group $G$ on itself. That is, denoting
\begin{align*}
    l_h: f(g) \mapsto f(hg), \quad r_h: f(g) \mapsto f(gh),
\end{align*}
we have $\mathrm{x}_j l_h=l_h\mathrm{x}_j$ for $j\in\{1,2,3\}$.

We construct the \emph{raising and lowering operators}
\begin{align*}
    \mathrm{e}^+ :=& 2 \mathrm{x}_1 + \mathrm{x}_2-i\mathrm{x}_3, \\
      \mathrm{e}^- :=& -2i \mathrm{x}_1 + \mathrm{x}_2+i\mathrm{x}_3.
\end{align*}
In the Iwasawa coordinates  $ \mathrm{n}[x] \mathrm{a}[y] \mathrm{k}[\theta]$ we have
\begin{align*}
\mathrm{e}^+  &= e^{2i\theta} (2iy\partial_x + 2 y \partial_y-i\partial_\theta),\\
\mathrm{e}^- &= e^{-2i\theta}(-2iy\partial_x + 2 y \partial_y+i\partial_\theta),\\
\mathrm{x}_3 &= \partial_\theta.
\end{align*}
The \emph{Casimir operator} is then defined as
\begin{align}
    \Omega := -\frac{1}{4}  \mathrm{e}^{+}  \mathrm{e}^{-} + \frac{1}{4}\mathrm{x}_3^2 - \frac{1}{2} i \mathrm{x}_3.
\end{align}
The Casimir element is not only left-invariant but also right-invariant under the action of the group $G$, that is,
\begin{align}\label{eq:Omegalrcommute}
    \Omega l_h =  l_h \Omega, \quad  \Omega r_h =  r_h \Omega.
\end{align}
In the Iwasawa coordinates $ \mathrm{n}[x] \mathrm{a}[y] \mathrm{k}[\theta]$ the Casimir operator $\Omega$   is given by
\begin{align} \label{eq:casimiriwasawa}
    \Omega = -y^2 (\partial_x^2 + \partial_y^2) + y \partial_x \partial_\theta,
\end{align}
where the first part $-y^2 (\partial_x^2 + \partial_y^2)$ is the Laplace operator on the upper half-plane $\mathbb{H}$.
In the Cartan coordinates $\mathrm{k}[\varphi] \mathrm{a}[e^{-\varrho}]\mathrm{k}[\vartheta]$, denoting  $\cosh \varrho =2u+1,$ we have
\begin{align} \label{eq:casimircartan}
\begin{split}
    \Omega =&  - \partial_\varrho^2 - \frac{1}{\tanh \varrho} \partial_\varrho -\frac{1}{4\sinh^2 \varrho} \partial_\varphi^2  +\frac{1}{2 \sinh\varrho \tanh \varrho} \partial_\varphi \partial_\vartheta - \frac{1}{4 \sinh^2 \varrho}  \partial_\vartheta^2 \\
        =& -u(u+1) \partial_u^2 - (2u+1) \partial_u - \frac{1}{16u(u+1)}  \partial_\varphi^2   +   \frac{2u+1}{4 u(u+1)}\partial_\varphi \partial_\vartheta -  \frac{1}{16u(u+1)} \partial_\vartheta^2 .
\end{split}
\end{align}
From this we see that $\Omega$ commutes with taking inverses: If
\begin{align*}
    \iota: f(g) \mapsto f(g^{-1}),
\end{align*}
then
\begin{align}\label{eq:Omegaiotacommute}
    \Omega \iota = \iota \Omega 
\end{align}
since
\begin{align*}
    (\k[\varphi] \a[e^{-\varrho}] \k[\vartheta])^{-1} =   \k[-\vartheta]\a[e^{\varrho}] \k[-\varphi] =\k[\pi/2-\vartheta]\a[e^{-\varrho}] \k[-\pi/2-\varphi].
\end{align*}

Let $\ell \in \Z$ and $\nu \in \C$. The functions (defined using the Iwasawa coordinates)
\begin{align} \label{eq:phibasicdefinition}
    \phi_{\ell}(g,\nu):=y^{\nu+1/2} e^{ i\ell \theta}
\end{align}
 are eigenfunctions of the Casimir operator $\Omega$ on $G$
\begin{align*}
    \Omega   \phi_{\ell}(g,\nu) = \bigg(\frac{1}{4}-\nu^2\bigg) \phi_{\ell}(g,\nu).
\end{align*}
Furthermore, the functions $ \phi_{\ell}(g,\nu)$ are of \emph{right-type} $\ell$, by which we mean that
\begin{align*}
 \phi_{\ell}(g \mathrm{k}[\theta],\nu) = e^{i \ell \theta}   \phi_{\ell}(g,\nu).
\end{align*}
The notion of functions with \emph{left-type} $\ell$ is defined similarly.
\subsection{Integration}

The invariant Haar measure $\d g$ on $G$ is defined in terms of the Iwasawa and Cartan coordinates via (note the swap between $\a$ and $\n$)
\begin{align} \label{eq:integrationonG}
\begin{split}
     \d \mathrm{n}[x] = \d x, \quad & \d \mathrm{a}[y] = \frac{\d y}{y}, \quad \d \mathrm{k}[\theta] = \frac{\d \theta}{2\pi}, \\
   \int_G  f(g)\d g =& \int_{\mathrm{A} \times \mathrm{N} \times \mathrm{K}} f(\a \n \k)  \d \a\d \n \d \k  =    \int_{\R\times \R_{>0} \times \R/2\pi\Z} f(\n[x] \a[y] \k[\theta]) \frac{\d x \d y \d \theta}{2\pi y^2} \\
   =&\int_{\mathrm{K} \times \mathrm{K}} \int_{\R_{>0}} f(\k_1 \a[e^{-\varrho}] \k_2)  \sinh(\varrho) \d \varrho \,\d \k_1 \d \k_2
\end{split}
\end{align}
Then
\begin{align}
    \int_{\Gamma_0(q)\backslash G} \d g= [\SL(\Z): \Gamma_0(q)] \int_{\SL(\Z)\backslash \SL(\R)} \d g= \frac{\pi}{6} q \prod_{p|q}(1+p^{-1}). \, 
\end{align}
In terms of the matrix coordinates we have (by using \eqref{eq:matrixtoiwasawa} to compute the Jacobian determinant)
\begin{align*}
    \d g = \frac{\d x \d y \d \theta}{2\pi y^2} = \frac{1}{\pi} \frac{\d a \d c \d d}{c}.
\end{align*}
Thus, using $\zeta(2) = \pi^2/6$ we have for any integrable $F:G \to\C$
\begin{align} \label{eq:integralinmatrix}
   \frac{1}{ |\Gamma_0(q)\backslash G|}    \int_G  F(g)\d g = \frac{1}{  \zeta(2) \, q \, \prod_{p|q}(1+p^{-1})}   \int_{\R^3}  F(\smqty(a & \ast \\ c & d)) \frac{\d a \d c \d d}{c}.
\end{align}

The inner product on $L^2(G)$ is defined by
\begin{align*}
    \langle f_1,f_2\rangle_G := \int_G f_1(g) \overline{f_2(g)} \d g.
\end{align*}
Then we have
\begin{align} \label{eq:integrationbypartsraisinglowering}
   \langle 
 \mathrm{e}^{\pm} f_1,f_2\rangle_G  = -   \langle 
 f_1, \mathrm{e}^{\mp}f_2\rangle_G
\end{align}
 and
 \begin{align} \label{eq:omegaselfadjoint}
      \langle 
\Omega f_1,f_2\rangle_G   =   \langle 
 f_1, \Omega f_2\rangle_G, 
 \end{align}
that is, $\Omega$ is  a symmetric operator.

\section{Spectral theory of $L^2(\Gamma \backslash G,\chi)$} \label{sec:Spectraltheoryof}

Fix $q \in \Z_{>0}$ and denote $\Gamma = \Gamma_0(q)$. While we consider the more general groups $\Gamma_2(q_1,q_2)$, these are isomorphic to $\Gamma_0(q_1 q_2)$ by conjugation with $\a[q_1]$. Let $\chi$ be a Dirichlet character to the modulus $q$. This defines a character on $\Gamma$ via
\begin{align*}
    \chi(\smqty(a &b \\ c & d)) := \chi(d) = \overline{\chi}(a).
\end{align*}
We let $\kappa\in \{0,1\}$ denote the parity of $\chi$, that is, $\chi(-1) = (-1)^\kappa.$ 

The group $\Gamma$ acts on $G$ by matrix multiplication from the left. We denote 
\begin{align*}
    \langle f_1,f_2 \rangle_{\Gamma \backslash G}:=   \int_{\Gamma \backslash G} f_1(g) \overline{f_2(g)} \d g
\end{align*}
and
\begin{align*}
    L^2(\Gamma \backslash G,\chi) := \{f: G \to \C: f(\gamma g) = \chi(\gamma) f(g), \quad  \langle f,f \rangle_{\Gamma \backslash G} < \infty\}.
\end{align*}
In this section we consider the decomposition of this space into an orthogonal basis of automorphic forms and the Fourier expansion of these basis elements. The automorphic forms described below of course depend on $q$ and $\chi$ but it is customary to omit this dependency in the notation.

\subsection{Spectral expansion}

 We denote by $\mathfrak{C}(q,\chi)$ a set of inequivalent representatives for the singular cusps of $\Gamma \setminus G$ for the character $\chi$. The notion of singular cusps and scaling  matrices coincide (with the identification $G/\mathrm{K} \cong \mathbb{H}$) with the upper half-plane setting ( see \cite[Section 4.1]{Drappeau}, for instance). 
 Let $ \phi_{\ell}(g,\nu)$ be as  in \eqref{eq:phibasicdefinition} and suppose that $\ell \equiv \kappa \pmod{2}$. For $\mathfrak{c}\in \mathfrak{C}(q,\chi)$, we define the Eisenstein series for $\R(\nu)>1/2$ by
\begin{align*}
    E_{\mathfrak{c}}^{(\ell)}(g,\nu)=\sum_{\gamma\in \Gamma_{\mathfrak{c}}\setminus \Gamma } \overline{\chi}(\gamma) \phi_\ell (\sigma_{\mathfrak{c}}^{-1} \gamma g, \nu),
\end{align*}
where $\Gamma_{\mathfrak{c}}$ is the stabilising group and $\sigma_{\mathfrak{c}}$ a scaling matrix of the cusp $\mathfrak{c}$. The Eisenstein series are extended to $\nu \in \C$ by meromorphic continuation, and have only one simple pole at $\nu=1/2$ for $\Gamma=\Gamma_0(q)$ \cite[Lemma 3.7]{deshoullieriwaniec}. For general Fuchsian groups $\Gamma$ it is possible that there are other poles on $0< \nu < 1/2$, which would appear as terms in the spectral decomposition \cite[Chapter 7]{IwaniecBook}.   The residue at $\nu=1/2$ is
\begin{align*}
    \mathrm{Res}_{\nu=1/2}  E^{(0)}_{\mathfrak{c}}(z,\nu) = \frac{3}{\pi q  \prod_{p|q} (1+p^{-1})}.
\end{align*}
\begin{proposition}\label{prop:specdecompL2}\emph{(Spectral expansion of $L^2(\Gamma \backslash G,\chi)$)}.
  Let $q\in \Z_{>0}$, $\chi$ be a Dirichlet character to the modulus $q$ of parity $\kappa \in \{0,1\}$, and let $\Gamma=\Gamma_0(q)$.  There exists a countable set  $\mathcal{B}(q,\chi) $ and for $V \in \mathcal{B}(q,\chi),\ell\in \Z, \ell \equiv \kappa \pmod{2}$, complex numbers $\nu_V$, and smooth functions $\varphi^{(\ell)}_V \in L^2(\Gamma \backslash G,\chi)$ of right-type $\ell$ with \begin{align}\label{eq:varphieigenfuncOmega}
        \Omega \varphi^{(\ell)}_V=\bigg(\frac{1}{4}-\nu_V^2\bigg)\varphi^{(\ell)}_V
    \end{align} 
    such that the following holds. For any smooth and bounded $f\in L^2(\Gamma \backslash G,\chi)$ with $\Omega f$ bounded we have
    \begin{align*}
        f(g)= &\mathbf{1}_{\chi \,\mathrm{principal}}\frac{1}{|\Gamma \backslash G|} \langle f ,1 \rangle_{\Gamma \backslash G} +\sum_{V \in \mathcal{B}(q,\chi)} \sum_{\ell\equiv \kappa \pmod{2}}\langle f ,\varphi_V^{(\ell)} \rangle_{\Gamma \backslash G}  \varphi_V^{(\ell)}(g)\\
        &+\sum_{\mathfrak{c}\in \mathfrak{C}(q,\chi)} \sum_{\ell\equiv \kappa \pmod{2}} \frac{1}{4 \pi i} \int_{(0)} \langle f, E^{(\ell)}_{\mathfrak{c}}(*,\nu) \rangle_{\Gamma \backslash G}  E^{(\ell)}_{\mathfrak{c}}(g,\nu) \d \nu.
    \end{align*}
\end{proposition}
Note that while the Eisenstein series are not square-integrable, for $f \in L^{2}(\Gamma \backslash G,\chi)$ the inner product  $\langle f, E^{(\ell)}_{\mathfrak{c}}(*,\nu) \rangle_{\Gamma \backslash G} $ exists and is finite.
Proposition \ref{prop:specdecompL2} follows from \cite[Section 4]{DFIartin}, by identifying for a fixed $\ell$ the basis elements $u_j$ of weight $k_{\text{DFI}}=\ell$ in  \cite[(4.50))]{DFIartin} via
\begin{align*}
  \varphi_V^{(\ell)}(\n[x] \a[y] \k[\theta])=  u_j(x+iy) e^{i\ell \theta},
\end{align*}
and similarly for the Eisenstein series.

We now describe the functions $\varphi_V^{(\ell)}$ in the discrete spectrum. The objects $V$ appearing are the irreducible subspaces of $L^2(\Gamma\backslash G,\chi)$, and the Casimir operator is constant on each $V$, that is, there is some $\nu_V$ such that for all smooth $f \in V$ we have $\Omega f = (\frac{1}{4}-\nu_V^2) f$. Each $V$ splits into subspaces according to the right-action by $\mathrm{K}$
\begin{align*}
    V = \bigoplus_{\ell \equiv \kappa \pmod{2}} V^{(\ell)},
\end{align*}
where $V^{(\ell)}$ is a one or zero dimensional subspace consisting of functions of right-type $\ell$. The raising and lowering operators define maps $\mathrm{e}^{\pm}: V^{(\ell)} \to  V^{(\ell \pm 2)}$. We pick generators $\varphi^{(\ell)}_V$ for each $V^{(\ell)}$ such that
\begin{align*}
    \langle \varphi^{(\ell)}_V, \varphi^{(\ell)}_V \rangle_{\Gamma \backslash G} =1
\end{align*}
to get an orthonormal basis.
 
The cuspidal spectrum $\mathcal{B}(q,\chi)$ splits into three parts
\begin{enumerate}
    \item Regular spectrum: $\nu_V \in i \R$
    \item Exceptional spectrum: $\nu_V \in (0,1/2)$
    \item Discrete series: $\nu_V =  \frac{k-1}{2}, \quad k > 0, \quad k \equiv \kappa \pmod{2}$.
    \end{enumerate}
In the first two cases all of the $V^{(\ell)}$ are non-trivial. The Selberg eigenvalue conjecture states that the exceptional spectrum is empty. The best bound towards this is by Kim-Sarnak \cite{KimSarnak}, which states that for congruence subgroups we have $\nu_V \in (0,7/64)$ in the possible exceptional spectrum. We let $\theta_q$ denote $\max_V\{  \Re(\nu_V) < 1/2\}$.

For the discrete series $\nu_V = k-1/2$ we call $k$ the \emph{weight} and we have either
\begin{align*}
    V = \bigoplus_{\substack{\ell  \geq k\\ \ell \equiv \kappa\pmod{2}}} V^{(\ell)} \quad \text{or} \quad   V= \bigoplus_{\substack{\ell  \leq - k\\  \ell \equiv \kappa\pmod{2}}} V^{(\ell)}.
\end{align*}
In either case the edge function $\varphi_V^{(\pm k)}$ is annihilated by the lowering/raising operator, that is, $\mathrm{e}^{\mp} \varphi_V^{(\pm k)} = 0$. For classical holomorphic modular forms $u(x+iy)$ of weight $k$ we have a correspondence
\begin{align*}
    \varphi_V^{(k)}(\n[x] \n[y] \k[\theta]) = y^{k/2} u(x+iy) e^{i k \theta}.
\end{align*}

\subsection{Fourier expansion of automorphic forms}

To compute the projections $\langle f ,\varphi_V^{(\ell)} \rangle_{\Gamma \backslash G}$ we use the Fourier expansion of the basis elements $\varphi_V^{(\ell)}$.  To state the Fourier expansion, we use the Jacquet operator $\mathcal{A}^{\pm}$. For a function $\phi$ on $G$ it is defined as
\begin{align*}
    \mathcal{A}^\pm \phi(g)=\int_{-\infty}^\infty e(\mp \xi) \phi( \w \n[\xi] g) \d \xi,
\end{align*}
where 
\begin{align*}
    \w=\k[\pi/2]=\mqty(0 & 1 \\ -1 & 0) 
\end{align*}
is the Weyl element of $G$. 

We will use the Jacquet operator only for the choice $\phi(g)=\phi_\ell (g,\nu)$ as given in \eqref{eq:phibasicdefinition}. In that case we can express it in terms of Whittaker functions  (see \cite[eq. (16.1)]{motohashielements})
\begin{align}\label{eq:Jacq=Whit}
       \mathcal{A}^\pm \phi_\ell(g,\nu)=(-1)^{\ell/2} \pi^{\nu+1/2} e^{i\ell  \theta} e(\pm x) \frac{W_{\pm \ell/2, \nu}(4\pi y)}{\Gamma(\nu \pm \ell/2 +\frac{1}{2})},
\end{align}
where in Iwasawa coordinates $g=\n[x]\a[y] \k[\theta]$ (the choice of the sign of $(-1)^{\ell/2}$ for odd $\ell$ will not be important for us). We record here the following integral expression (see \cite[eq. (16.2)]{motohashielements}) which will be useful later 
\begin{align}\label{eq:Jacquetintegral}
    \mathcal{A}^{\pm} \phi_{\ell}(g,\nu)=e^{i\ell  \theta} e(\pm x) y^{1/2-\nu} \int_{-\infty}^\infty \frac{e(-y\xi) }{(\xi^2+1)^{\nu+1/2}} \left(\frac{\xi-i}{\xi+i} \right)^{\pm \ell/2} \d \xi,
\end{align}
where for odd $\ell$ we choose the branch cut  $\R_{> 0}$ for $(\frac{\xi-i}{\xi+i} )^{\pm \ell/2}$.
We also note that
\begin{align}\label{eq:Jacdiffop}\begin{split}
    \Omega \mathcal{A}^{\pm} \phi_{\ell}(g,\nu) =(1/4-\nu^2) \mathcal{A}^{\pm} \phi_{\ell}(g,\nu), \quad   \partial_x^2 \mathcal{A}^{\pm} \phi_{\ell}(g,\nu) = -4 \pi^2 \mathcal{A}^{\pm} \phi_{\ell}(g,\nu), \\ \partial_\theta^2 \mathcal{A}^{\pm} \phi_{\ell}(g,\nu) = -  \ell^2 \mathcal{A}^{\pm} \phi_{\ell}(g,\nu). 
\end{split}
\end{align}

We gather the statements of \cite[(17.3), (17.5), (18.2), (21.11), (24.3)]{motohashielements} in the following, noting that for $\nu_V \in i \R $ we have $\left|\pi^{-2\nu_V}\frac{\Gamma(|\ell|/2+\nu_V+1/2)}{\Gamma(|\ell|/2-\nu_V+1/2)} \right|^{1/2} = 1$.
\begin{proposition}\label{prop:fourierexpansionbasis}
We have the Fourier expansions
\begin{align*}
    \varphi_V^{(\ell)}(g)=\left|\pi^{-2\nu_V}\frac{\Gamma(|\ell|/2+\nu_V+1/2)}{\Gamma(|\ell|/2-\nu_V+1/2)} \right|^{1/2}\sum_{n\neq 0} \frac{\varrho_V(n)}{\sqrt{|n|}} \mathcal{A}^{\sgn(n)}\phi_{\ell}(\mathrm{a}[|n|]g,\nu_V)
\end{align*}
and
\begin{align*}
    E_{\mathfrak{c},\ell}(g,\nu)=&1_{\mathfrak{c}=\infty}\phi_\ell(g,\nu)+\varrho_{\mathfrak{c,\nu}}^{(\ell)}(0)\phi_\ell(g,-\nu)\\
    &+\sum_{n\neq 0}\frac{\varrho_{\mathfrak{c},\nu}(n)}{\sqrt{|n|}} \mathcal{A}^{\sgn(n)}\phi_{\ell}(\mathrm{a}[|n|]g,\nu_V),
\end{align*}
where the coefficients  $\varrho_V(n)$ and $\varrho_{\mathfrak{c},\nu}(n)$ do not depend on $\ell$.
\end{proposition}
Of course the Fourier coefficients $\varrho_{V}(n)$ and $\varrho_{\mathfrak{c},\nu}(n)$ depend on both $q$ and $\chi$. As mentioned, this is suppressed in the notation. The Fourier coefficients can be written in terms of the upper half plane setup as in \cite{deshoullieriwaniec} for $\kappa = 0$ as follows (the supersript D-I refers to the notation in Deshouillers-Iwaniec)
\begin{align*}
    \varrho_V(n)&=\frac{\Gamma(\nu_V+1/2)}{2 \pi^{\nu_V+1/2}} \rho^{\text{D-I}}_j(n) \quad \text{with}\quad \nu_V=i \kappa^{\text{D-I}}_j, \quad \text{if } \nu_V\not\in \mathbb{N}-1/2;\\
     \varrho_V(n)&=\frac{(-1)^{k/2} \Gamma(|k|)^{1/2}}{2^{k}\pi^{(|k|+1)/2} n^{(k-1)/2}} \psi^{\text{D-I}}_{j, k}(n), \quad \text{ if } \nu_V=(k-1)/2, n>0 \text{ and } V \text{ holomorphic}; \\
       \varrho_V(n)&=\overline{ \varrho_{\overline{V}}(-n)}, \quad \text{ if } \nu_V=k-1/2, n<0 \text{ and } V \text{ anti-holomorphic}; \\
    \varrho_{\mathfrak{c},\nu}(n)&=|n|^{\nu}\varphi^{\text{D-I}}_{\mathfrak{c}\infty n}(1/2+\nu),\quad \text{ if }n\neq 0.
\end{align*}

\subsection{Hecke operators}

The Hecke operator on $L^2(\Gamma_0(q)\backslash G,\chi)$  is defined by
\begin{align} \label{eq:heckeopdef}
    \mathcal{T}_h f(g) :=  \frac{1}{\sqrt{h}} \sum_{ad=h} \chi(a) \sum_{b  \pmod{d}} f\bigg(\frac{1}{\sqrt{h}}\mqty(a & b  \\ & d)g\bigg).
\end{align}
It satisfies the multiplicativity relation
\begin{align*}
    \mathcal{T}_m \mathcal{T}_n = \sum_{d | \gcd(m,n)} \chi(d)\mathcal{T}_{mn/d^2}. 
\end{align*}
Then for $\gcd(h,q)=1$  the Hecke operators $\mathcal{T}_h$ commute with each other and $\Omega$,  and furthermore are normal (since $\langle \mathcal{T}_h f,g \rangle = \langle  f, \overline{\chi}(h)\mathcal{T}_h g \rangle $), so that we can choose a common orthonormal basis. Then 
\begin{align}\label{eq:heckecusp}
    \mathcal{T}_h \varphi_V^{(\ell)} (\tau) = \lambda_V(h) \varphi_V^{(\ell)} (\tau)
\end{align}
for the Hecke eigenvalues $\lambda_V(h)$. Similarly, for the Eisenstein series we have for $\gcd(h,q)=1$ \cite[(6.16)]{DFIartin}
\begin{align}\label{eq:heckeeis}
    \mathcal{T}_h E_{\mathfrak{c}}^{(\ell)} (\tau,\nu) = \lambda_{\mathfrak{c},\nu}(h) E_{\mathfrak{c}}^{(\ell)} (\tau,\nu)
\end{align}
with $\lambda_{\mathfrak{c},\nu}(h) $ given explicitly in \cite[(6.17)]{DFIartin}. For $\gcd(h,q)=1$ the Hecke eigenvalues satisfy the multiplicativity relation
\begin{align*}
  \lambda_V(h) \varrho_V(n) = \sum_{d | \gcd(h,n)} \chi(d) \varrho_V(hn/d^2),  
\end{align*}
and similarly for the Eisenstein series.

We have  $|\lambda_{\mathfrak{c},\nu}(h)| \leq d(h)  \ll_\eps |h|^{\eps}.$ The Ramanujan-Petersson conjecture states that also in the cuspidal spectrum
\begin{align*}
 |\lambda_{V}(h)| \ll_\eps |h|^{\eps}.
\end{align*}
For the discrete series, this was proved by Deligne \cite{Deligne1968-1969}. In general the best bound towards this is $ \ll_\eps |h|^{\vartheta_{q}+\eps}$ with $\vartheta_q \leq 7/64$ by Kim and Sarnak \cite{KimSarnak}. 
 
\subsection{Useful heuristics}
The automorphic forms $\varphi_V^{(\ell)}$ are $L^2$-normalized and since the volume of the fundamental domain satisfies $|\Gamma_0(q) \backslash G|  = q^{1+o(1)}$, we expect that for typical  $g \in G$ and $n \in \Z\setminus\{0\}$
\begin{align*}
  |\varphi_V^{(\ell)}(g)| = q^{-1/2+o(1)}, \quad \quad |\varrho_V(n) |   = q^{-1/2+o(1)}.
\end{align*}

With the help of Selberg's trace formula one can show that the eigenvalues $\nu_V$ in the cuspidal spectrum satisfy
\begin{align*}
    |\{ \nu_V \in i\R: |\nu_V| \leq K \}| = q^{1+o(1)} K^2,
\end{align*}
that is, there are roughly $q^{1+o(1)} K$ spectral parameters in $i[K,K+1]$. In our applications $K$ will be small, so that morally the number of harmonics used is  $q^{1+o(1)}$. We are unable to detect cancellation between the harmonics and thus always lose  at least a factor of $\sqrt{q}$ compared to heuristically optimal bounds in the error terms.

\section{Decay of spectrum and Fourier expansion}\label{sec:decayandfourier}
After applying the spectral expansion of Proposition \ref{prop:specdecompL2} we will need to bound inner products such as  $\langle P, \varphi_{V}^{(\ell)}\rangle_{\Gamma\backslash G}$, where $P$ is a kind of Poincar\'e series associated to some smooth function $F$ on the ambient space $G$. By unfolding  and inserting the Fourier expansion  (Proposition \ref{prop:fourierexpansionbasis}) we encounter then sums of the form
\begin{align*}
  \sum_{V,\ell}\sum_{n} \frac{\varrho_V(n)}{\sqrt{|n|}}  \langle F, \mathcal{A}^{\sgn(n)} \phi_\ell(\mathrm{a}[|n|] \cdot \,,\nu_V) \rangle_{G}.
\end{align*}
The main goal of this section is to prove the following bound, which simultaneously achieves two things -- decay in the parameters $n, \nu,\ell$ as well as summation by parts to decouple the smooth weight from the variable $n$. This morally speaking allows us to truncate the above sum at $|\nu_V| \ll \delta^{-1},|\ell| \ll \delta^{-1}, |n| \ll (\delta X)^{-1}$.
\begin{proposition}\label{prop:Jacquetinnerproduct}
    Let $\delta \in (0,1)$ and $X,Y > 0$, $\kappa \in \{0,1\}$. Let $a_{V,\ell},b_{\mathfrak{c},\nu,\ell,}:\Z_{>0} \to \C$ and let $F:G \to \C$,  $F(g)=F(\n[x] \a[y] \k[\theta])$, be  $(2J+1)$ times differentiable in all three variables, supported on $(x,y) \in [-X,X]\times [Y,2Y]$, satisfying for all $J_1,J_2,J_3 \geq 0$ with $J_1+J_2+J_3 \leq 2J+1$
\begin{align}\label{eq:JIPderivbound}
        \partial_x^{J_1} \partial_y^{J_2}  \partial_\theta^{J_3} F \ll (\delta X)^{-J_1} (\delta Y)^{-J_2} (\delta)^{-J_3}.
    \end{align}
    Define the decay function
    \begin{align} \label{eq:decaydef}
        D_J(t,\nu,\ell) :=\min_{J_0\in \{0,1,\dots,J\}} \frac{ (1+\log t)^2 (1+Y/X)^{2J_0} }{\delta^{2}(1+|\delta \nu| + |\delta t X| + |\delta\ell|)^{2J_0-1}  } 
    \end{align}
    Then for $\sigma \in \{\pm\}$ we have for any $J\geq0$   \begin{align}\label{eq:JIPMaas}
    \begin{split}
 \sum_{\substack{V \in \mathcal{B}(q,\chi) \\ \Re(\nu_V) < 1/2}}& \sum_{\ell \equiv \kappa \pmod{2} }  \sum_{n > 0} 
 \langle F, \mathcal{A}^{\sigma} \phi_\ell(\mathrm{a}[n] \cdot,\nu_V) \rangle_{G} \, \frac{a_{V,\ell}(n)}{\sqrt{n}} \\
 \ll & \max_{t > 0} \frac{X}{Y^{1/2}}\sum_{\substack{V\in \mathcal{B}(q,\chi) \\ \Re(\nu_V) < 1/2}} \sum_{\ell \equiv \kappa \pmod{2} }(  Yt  )^{- \Re(\nu_V)}    D_J(t,\nu_V,\ell) 
 \bigg| \sum_{0< n \leq t}  a_{V,\ell}(n)  \bigg|,
    \end{split} 
    \end{align}
    \begin{align}
        \label{eq:JIPHolom}
        \begin{split}
        \sum_{\substack{V \in \mathcal{B}(q,\chi) \\ \nu_V = (k-1)/2}} \sum_{\substack{\ell \equiv \kappa \pmod{2}\\ \sigma \ell \geq k}}  &\left|\pi^{-2\nu_V}\frac{\Gamma(|\ell|/2+\nu_V+1/2)}{\Gamma(|\ell|/2-\nu_V+1/2)} \right|^{1/2} \sum_{n > 0} 
 \langle F, \mathcal{A}^{\sigma} \phi_\ell(\mathrm{a}[n] \cdot,\nu_V) \rangle_{G} \, \frac{a_{V,\ell}(n)}{\sqrt{n}} \\
 & \ll \max_{t > 0}   \frac{X}{Y^{1/2}}  \sum_{\substack{V \in \mathcal{B}(q,\chi)\\ \nu_V = (k-1)/2}}  \sum_{\substack{\ell \equiv \kappa \pmod{2}\\ \sigma \ell \geq k}} D_J(t,\nu_V,\ell)\bigg| \sum_{0 < n \leq t}  a_{V,\ell}(n)  \bigg|,
        \end{split}
    \end{align}
and
 \begin{align}\label{eq:JIPEis}
 \begin{split}
       \sum_{\c \in \mathfrak{C}(q,\chi)} &\sum_{\ell \equiv \kappa \pmod{2}} \int_{(0)}  \sum_{n > 0} 
 \langle F, \mathcal{A}^{\sigma} \phi_\ell(\mathrm{a}[n] \cdot,\nu) \rangle_{G} \, \frac{b_{\c,\nu,\ell}(n)}{\sqrt{n}} \d \nu  \\
 &\ll \max_{t > 0}  \frac{X}{Y^{1/2}} \sum_{\c \in \mathfrak{C}(q,\chi)}\int_{(0)} \sum_{\ell \equiv \kappa \pmod{2}} D_J(t,\nu,\ell) \bigg| \sum_{0< n \leq t}  b_{\c,\nu,\ell}(n) \bigg| |\d \nu|.
 \end{split}
    \end{align}
\end{proposition}

After using the Fourier series of Proposition \ref{prop:fourierexpansionbasis}, Proposition \ref{prop:Jacquetinnerproduct} gives us a tool to handle all parts of the spectral expansion except the $n=0$ terms of the Eisenstein series. This is covered by the following. 

\begin{proposition} \label{prop:eisenstein0prop}
Suppose that the assumptions of Proposition \ref{prop:Jacquetinnerproduct} hold and let $D_J(t,\nu,\ell)$ be defined by \eqref{eq:decaydef}.  We have for any  complex coefficients $b_{\mathfrak{c},\nu,\ell}$
    \begin{align*}
     \sum_{\c \in \mathfrak{C}(q,\chi)} \sum_{\ell \equiv \kappa \pmod{2} } \int_{(0)}\langle F,\phi_\ell (\cdot,\nu)\rangle b_{\mathfrak{c},\nu,\ell}\ll\frac{X}{Y^{1/2}} \sum_{\c \in \mathfrak{C}(q,\chi)}\int_{(0)} \sum_{\ell \equiv \kappa \pmod{2} }D_J(1,\nu,\ell) |b_{\mathfrak{c},\nu,\ell}|  |\d \nu|.
    \end{align*}
\end{proposition}
To show Proposition \ref{prop:Jacquetinnerproduct} we use the following lemma that is an extension of \cite[eqs. (15.4), (15.5)]{motohashielements}.
\begin{lemma} \label{le:whittakerbounds}
Let $g= \mathrm{n}[x] \mathrm{a}[y] \mathrm{k}[\theta]$. 
We have the following bounds. 
 \begin{enumerate}
     \item For $\Re(\nu) \geq 0$ and $y \in (0,1)$ we have
    \begin{align*}
        \mathcal{A}^{\sigma} \phi_{\ell}(g,\nu) \ll_\eps \, (|\ell| + |\nu| + 1)y^{1/2-\Re(\nu)} (1+|\log y|).
    \end{align*}

    \item For $\Re(\nu)  \geq 0$ and $y > 0$ we have
    \begin{align*}
    \mathcal{A}^{\sigma} \phi_{\ell}(g,\nu) \ll  (|\ell|+|\nu|+1) y^{-1/2-\Re(\nu)} e^{-y/(|\ell| + |\nu| + 1)}
    \end{align*}
%    \item  For $\nu=k-1/2$, $\sigma \ell \geq k$, $y  \in (0,1)$ we have
 %   \begin{align*}
 %   \bigg(\pi^{-2 \nu_V}\frac{\Gamma(|\ell|+k)}{\Gamma(|\ell|-k+1)}  \bigg)^{1/2}  |\mathcal{A}^{\sigma} \phi_{\ell}(g,\nu) |  \ll  1
 %   \end{align*}
    \item  For $\nu=(k-1)/2$, $\sigma \ell \geq k$, $\ell \equiv k \pmod{2}$  we have
    \begin{align*}
          \pi^{-2 \nu_V}\frac{\Gamma(|\ell|/2+k/2)}{\Gamma(|\ell|/2-k/2+1)}  \int_0^\infty |\mathcal{A}^\sigma  \phi_{\ell}(\mathrm{a}[y],\nu) |^2 \frac{\d y}{y^2} \ll \, 1.
    \end{align*}
 \end{enumerate}  
 
\end{lemma}

\begin{proof}
Statements (1) and (2) are contained in \cite[eq. (15.5)]{motohashielements} in the case $\ell \equiv 0 \, \pmod{2}$. The proofs apply also for $\ell \equiv 1 \, \pmod{2}$.

To prove (3), for $\nu=(k-1)/2$ and $\sigma\ell\geq k$ we have
\begin{align*}
 \bigg(\pi^{-2 \nu_V}\frac{\Gamma(|\ell|/2+k/2)}{\Gamma(|\ell|/2-k/2+1)} & \bigg)^{1/2}   \mathcal{A}^\sigma  \phi_{\ell}(\mathrm{a}[y],\nu) \\
 =&  (-1)^{k/2} \pi^{1/2}\frac{\Gamma(|\ell|/2 - k/2 +1)^{1/2}}{\Gamma(|\ell|/2 + k/2)^{1/2}} e^{-2 \pi y} (4 \pi y)^{k/2} L^{(k-1)}_{(|\ell| - k)/2} (4 \pi y),
\end{align*}
where $L^{(\alpha)}_n$ is the Laguerre polynomial \cite[Chapter V (5.1.5)]{szegobook} defined by
\begin{align*}
 L^{(\alpha)}_n(x) :=  \frac{x^{-\alpha} e^{x}}{n!} \partial_{x}^{n} \bigg( e^{-x} x^{n+\alpha} \bigg).
\end{align*}
This can be seen from shifting the contour to $\Im (\xi) = -\infty$ in \eqref{eq:Jacquetintegral} and calculating the residue of the pole of order $(|\ell| - k)/2$ at $\xi = -i$. The Laguerre polynomials satisfy the orthogonality relation \cite[Chapter V, (5.1.1)]{szegobook}
\begin{align*}
    \int_0^\infty x^\alpha e^{-x} L^{(\alpha)}_{n}(x)L^{(\alpha)}_{m}(x) \d x = \mathbf{1}_{m=n} \frac{\Gamma(n+\alpha+1)}{\Gamma(n+1)}.
\end{align*}
We also have the recurrence relation \cite[Chapter V, (5.1.13)]{szegobook}
\begin{align*}
    L_{n}^{(\alpha)}(x) = \sum_{i=0}^{n}  L_{i}^{(\alpha-1)}(x).
\end{align*}
Thus, we get by expanding the square and by orthogonality
\begin{align*}
     \int_0^\infty x^{\alpha-1} e^{-x} L^{(\alpha)}_{n}(x)^2 \d x = \sum_{i=0}^n \frac{\Gamma(i+\alpha)}{\Gamma(i+1)}  \leq (n+1) \frac{\Gamma(n+\alpha)}{\Gamma(n+1)},
\end{align*}
so that in particular
\begin{align*}
        \int_0^\infty (4 \pi y)^{k-2} e^{-4 \pi y} L^{(k-1)}_{(|\ell| - k)/2}(4 \pi y)^2 \d y \ll  (|\ell|/2-k/2 +1 ) \frac{\Gamma(|\ell|/2 +k/2-1)}{\Gamma(|\ell|/2-k/2 +1)}.
\end{align*}
Thus, we have
\begin{align*}
  \pi^{-2 \nu_V}\frac{\Gamma(|\ell|/2+k/2)}{\Gamma(|\ell|/2-k/2+1)}  &\int_0^\infty  |\mathcal{A}^\sigma  \phi_{\ell}(\mathrm{a}[y],\nu) |^2 \frac{\d y}{y^2} \\
  =&  4^2\pi^3  \frac{\Gamma(|\ell|/2-k/2+1)}{\Gamma(|\ell|/2+k/2)} \int_0^\infty (4 \pi y)^{k-2} e^{-4 \pi y} L^{(k-1)}_{(|\ell| - k)/2}(4 \pi y)^2 \d y \\
\ll &\, (|\ell|/2-k/2 +1 )  \frac{\Gamma(|\ell|/2-k/2+1)}{\Gamma(|\ell|/2+k/2)}  \frac{\Gamma(|\ell|/2 +k/2-1)}{\Gamma(|\ell|/2-k/2 +1)} \\
\ll & \, 1.
\end{align*} 
\end{proof}

\subsection{Proof of Propositions \ref{prop:Jacquetinnerproduct} and \ref{prop:eisenstein0prop}}
Let us first consider \eqref{eq:JIPMaas}.
    Let \begin{align*}
        \psi_{\nu,\ell}(t) := t^{-1/2} \langle F, \mathcal{A}^{\sigma} \phi_\ell(\mathrm{a}[t] \cdot,\nu) \rangle_{G}
    \end{align*}
    so that the left-hand side of \eqref{eq:JIPMaas} is
    \begin{align*}
     \sum_{\substack{V \in \B(q,\chi) \\ \Re(\nu_V) < 1/2}} \sum_{\ell \equiv \kappa \pmod{2}}  \sum_{n > 0} 
\psi_{\nu_V,\ell}(n) a_{V,\ell}(n).
    \end{align*}
We define the second order differential operator
\begin{align*}
    \Xi:= 1+\delta^2(1/4-\Omega) + \bigg( \frac{\delta X}{2 \pi i} \partial_x\bigg)^2 +  \bigg( \frac{\delta }{ i} \partial_\theta\bigg)^2
\end{align*}
and observe that  by \eqref{eq:Jacdiffop} and by left-invariance of $\Omega$  for any $J\geq 0$
\begin{align} \label{eq:xiJbound}
  \Xi^J   \mathcal{A}^{\sigma} \phi_\ell(\a[t] g,\nu) = \Xi^J l_{\a[t]}  \mathcal{A}^{\sigma} \phi_\ell( g,\nu)  = (1+(\delta \nu)^2 + (\delta t X)^2 + (\delta\ell)^2)^J l_{\a[t]}\mathcal{A}^{\sigma} \phi_\ell( g,\nu)  
\end{align}
and that by \eqref{eq:JIPderivbound} 
\begin{align} \label{eq:xi0bound}
     \begin{split}
         \Xi^J F \ll_J (1 +Y/X)^{2J} \quad \quad \text{and} \\
 \partial_t l_{\a[t^{-1}]}  \Xi^J F \ll_J (\delta t^2)^{-1} (1 +Y/X)^{2J}.
     \end{split} 
\end{align}
By \eqref{eq:xiJbound} we get for any $J\geq 0$
\begin{align*}
    \psi_{\nu,\ell}'(t) &= (t^{-1/2} \partial_t- \tfrac{1}{2} t^{-3/2}) \langle F, l_{\a[t]}\mathcal{A}^{\sigma} \phi_\ell( \cdot,\nu) \rangle_{G} \\
    &= \frac{(t^{-1/2} \partial_t- \tfrac{1}{2} t^{-3/2})}{(1+(\delta \nu)^2 + (\delta t X)^2 + (\delta\ell)^2)^J } \langle F, \Xi^J l_{\a[t]}\mathcal{A}^{\sigma} \phi_\ell( \cdot,\nu) \rangle_{G} \\
    &=  \frac{(t^{-1/2} \partial_t- \tfrac{1}{2} t^{-3/2})}{(1+(\delta \nu)^2 + (\delta t X)^2 + (\delta\ell)^2)^J } \langle  l_{\a[-t]}\overline{\Xi^J} F,  \mathcal{A}^{\sigma} \phi_\ell( \cdot,\nu) \rangle_{G},
 \end{align*}
where we used the symmetry of $\Omega$ \eqref{eq:omegaselfadjoint} as well as the Iwasawa decomposition to apply integration by parts for $\partial_x$ and $\partial_\theta$, and the invariance of the measure $\d g$. Using \eqref{eq:xi0bound} we get
\begin{align*}
        \psi_{\nu,\ell}'(t)  \ll  \frac{(\delta^{-1} t^{-5/2} + t^{-3/2})(1+Y/X)^{2J}}{(1+|\delta \nu| + |\delta t X| + |\delta\ell|)^{2J} } \int_{G} \mathbf{1}_{\substack{x\ll Xt \\ y\asymp Yt }} \left| \mathcal{A}^{\sigma} \phi_\ell(\n[x]\a[y] ,\nu)\right| \d g.
\end{align*}
This gives us by (1) and (2) of Lemma \ref{le:whittakerbounds}
\begin{align} \label{eq:psibound}
\begin{split}
      \psi_{\nu,\ell}'(t)  &\ll  \frac{(\delta^{-1} t^{-5/2} + t^{-3/2})(1+Y/X)^{2J} (1+|\nu|+|\ell|)}{(1+|\delta \nu| + |\delta t X| + |\delta\ell|)^{2J} } 
 Xt \int_{\substack{y\asymp Yt }} y^{-3/2-\Re(\nu)} \d y \\ 
 & \ll  \frac{(\delta^{-1} t^{-2} + t^{-1})(1+Y/X)^{2J} (1+|\nu|+|\ell|)}{(1+|\delta \nu| + |\delta t X| + |\delta\ell|)^{2J}  } 
 \frac{X}{Y^{1/2}}  (Yt)^{-\Re(\nu)}.  
\end{split}
\end{align}
Defining
    \begin{align*}
        R_{V,\ell}(t):= \sum_{0<n \leq t} a_{V,\ell}(n),
    \end{align*}
    we bound the left-hand side of \eqref{eq:JIPMaas} with Abel summation and \eqref{eq:psibound}
    \begin{align*}
      & \sum_{\substack{V\in B(q,\chi)  \\ \Re(\nu_V) < 1/2}} \sum_{\ell \equiv \kappa \pmod{2}} \sum_{n > 0} 
\psi_{\nu_V,\ell}(n) a_{V,\ell}(n) =  -\int_{1/2}^\infty  \sum_{\substack{V \in \B(q,\chi)  \\ \Re(\nu_V) < 1/2}} \sum_{\ell \equiv \kappa \pmod{2}} \psi'_{\nu_V,\ell}(t)   R_{V,\ell}(t) \d t \\
&\leq \int_{1/2}^\infty  \sum_{\substack{V\in \B(q,\chi)  \\ \Re(\nu_V) < 1/2}} \sum_{\ell \equiv \kappa \pmod{2}} |\psi'_{\nu_V,\ell}(t) | | R_{V,\ell}(t)| \d t \\
&\ll \int_{1/2}^\infty  \frac{X }{Y^{1/2} }  
\sum_{\substack{V \in \B(q,\chi) \\ \Re(\nu_V) < 1/2}} \sum_{\ell \equiv \kappa \pmod{2}}\frac{(\delta^{-1} t^{-2} + t^{-1})(1+Y/X)^{2J} (1+|\nu|+|\ell|)}{(1+|\delta \nu| + |\delta t X| + |\delta\ell|)^{2J}  } 
 (Yt)^{-\Re(\nu)} | R_{V,\ell}(t)| \d t  \\
 &\ll \int_{1/2}^\infty  \frac{X }{Y^{1/2} }  
\sum_{\substack{V \in \B(q,\chi) \\ \Re(\nu_V) < 1/2}} \sum_{\ell \equiv \kappa \pmod{2}}\frac{ (1+\log t)^2 (1+Y/X)^{2J} }{\delta^{2}(1+|\delta \nu| + |\delta t X| + |\delta\ell|)^{2J-1}  } 
 (Yt)^{-\Re(\nu)} | R_{V,\ell}(t)| \frac{\d t}{t (1+\log t)^2} 
    \end{align*}
The claim \eqref{eq:JIPMaas} now follows by allowing the parameter $J$ to depend on $t$, taking the minimum over them, and then taking the maximum over $t$ since $\int_{1/2}^\infty \frac{\d t}{t (1+\log t)^2} < \infty$. The proof of \eqref{eq:JIPEis} follows by the same argument.

For the discrete series \eqref{eq:JIPHolom} with $\nu_V=(k-1)/2$, $\sigma \ell > k$, to get the analogue of \eqref{eq:psibound} we have to bound
\begin{align*}  \left|\pi^{-2\nu_V}\frac{\Gamma(|\ell|/2+\nu_V+1/2)}{\Gamma(|\ell|/2-\nu_V+1/2)} \right|^{1/2}  \int_{y\sim Yt} \left| \mathcal{A}^{\sigma} \phi_\ell(\a[y] ,\nu)\right| \frac{\d y}{y^2}.
\end{align*}
By Cauchy-Schwarz and (3) of Lemma \ref{le:whittakerbounds} this is bounded by
\begin{align*}
    \ll & \bigg( \int_{y\sim Yt} \frac{\d y}{y^2} \bigg)^{1/2} \bigg(  \left|\pi^{-2\nu_V}\frac{\Gamma(|\ell|/2+\nu_V+1/2)}{\Gamma(|\ell|/2-\nu_V+1/2)} \right|   \left| \int_0^\infty\mathcal{A}^{\sigma} \phi_\ell(\a[y] ,\nu)\right|^2 \frac{\d y}{y^2}\bigg)^{1/2} 
    \ll & (Yt)^{-1/2}.
\end{align*}
The rest of the argument is exactly the same.

Proposition \ref{prop:eisenstein0prop} follows by a similar argument but is technically much easier. We use the point-wise upper bound $|\phi_\ell(g,\nu)| \leq y^{1/2}$ instead of Lemma \ref{le:whittakerbounds}.

\section{Spectral expansion of an automorphic kernel on $G$}\label{sec:specexpauto}
Let $k:G \to \C$ be a smooth function with compact support. For $\ell_1 \equiv \ell_2 \pmod{2}$ we define the projections
\begin{align}\label{eq:projectionll}
    k_{\ell_1,\ell_2}(g) :=  \int_\mathrm{K} \int_\mathrm{K} k( \mathrm{k_1} g \mathrm{k_2}) e^{-i\ell_1 \theta(\mathrm{k_1})-i\ell_2 \theta(\mathrm{k_2})} \d \mathrm{k}_1 \d \mathrm{k}_2.
\end{align}
Then
\begin{align*}
    k_{\ell_1,\ell_2}(\mathrm{k}[\theta_1] g\mathrm{k}[\theta_2]) =     e^{i\ell_1 \theta_1+i\ell_2 \theta_2}k_{\ell_1,\ell_2}(g) 
\end{align*}
and we say that any function with this property is of type $(\ell_1,\ell_2)$, or left-type $\ell_1$ and right-type $\ell_2$. By the Cartan decomposition $g=\mathrm{k}[\varphi] \mathrm{a}[e^{-\varrho}]\mathrm{k}[\vartheta]$ a function of fixed type is uniquely determined by its values on $\mathrm{A}$. We have a Fourier series decomposition into functions of fixed type
\begin{align*}
    k(g) = \sum_{\ell_1,\ell_2 \in \Z}  k_{\ell_1,\ell_2}(g) = \sum_{\ell_1,\ell_2 \in \Z}  k_{\ell_1,\ell_2}(\mathrm{a}[e^{-\varrho}]) e^{i\ell_1 \varphi+i\ell_2 \vartheta}.
\end{align*}
Given $k$ as above, we define the invariant integral operator with kernel $k_{\ell_1,\ell_2}$ as
\begin{align*}
    L_{\ell_1,\ell_2} f (g):= \int_G k_{\ell_1,\ell_2}(g^{-1} h ) f(h) \d h.
\end{align*}

The main goal of this section is to prove the following proposition, which gives a spectral expansion of an automorphic kernel (also known as the pre-trace formula). For our purpose it is enough to consider the case $\ell_1=\ell_2$. In terms of the upper half-plane language this can be found in Hejhal's books \cite{Hejhal1,Hejhal2}. As these comprehensive books are not the easiest-to-use sources, we opt to provide a self-contained proof, generalising the argument in \cite[Chapter 1.8]{IwaniecBook}.

\begin{proposition} \label{prop:spectralexpofkernel}
    Let $k:G \to \C$ be smooth and compactly supported. Define
    \begin{align*}
        K^{(\ell)}(\tau_1,\tau_2) := \sum_{\gamma \in \Gamma} \overline{\chi}(\gamma)k_{\ell,\ell}(\tau_2^{-1} \gamma \tau_1).
    \end{align*}
   Then for $\ell \equiv \kappa \pmod{2}$
    \begin{align*}
        K^{(\ell)}(\tau_1,\tau_2) = & \frac{\mathbf{1}_{\chi \,\mathrm{principal}} \mathbf{1}_{\ell=0}}{|\Gamma \backslash G|} \int_G k(g) \d g  \\
      &+  \sum_{V\in \mathcal{B}(q,\chi)}
        \Phi_{\ell,\ell}(k;\nu_V) \varphi_V^{(\ell)}(\tau_1) \overline{\varphi_V^{(\ell)}(\tau_2)} \\
&+\sum_{\mathfrak{c}\in \mathfrak{C}(q,\chi)} \frac{1}{4 \pi i}  \int_{(0)}  \Phi_{\ell,\ell}(k;\nu)E^{(\ell)}_{\mathfrak{c}}(\tau_1,\nu) \overline{E^{(\ell)}_{\mathfrak{c}}(\tau_2,\nu)} \d \nu,
    \end{align*}
    where 
    \begin{align} \label{eq:hfromkdefinition}
\Phi_{\ell,\ell}(k;\nu) :=& \int_G  k_{\ell,\ell}(h) \overline{\phi_{\ell}(h,\nu)} \d h .
    \end{align}
\end{proposition}

Proposition \ref{prop:spectralexpofkernel} is a corollary of Proposition \ref{prop:specdecompL2} and the following.
\begin{proposition} \label{prop:integraloperator}
Let $\nu \in \C$, $k \in C_0^\infty(G)$, and let $f_\nu^{(\ell)}$ be an eigenfunction of $\Omega$ with eigenvalue $1/4-\nu^2$ and right-type $\ell$. Then we have 
\begin{align*}
    L_{\ell,\ell} \overline{f_\nu^{(\ell)} (g) }=  \Phi_{\ell,\ell}(k;\nu) \overline{ f_\nu^{(\ell)}(g)}.
\end{align*}
\end{proposition}

\begin{proof}[Proof of Proposition \ref{prop:spectralexpofkernel} assuming Proposition \ref{prop:integraloperator}]
 By Proposition \ref{prop:specdecompL2} applied to $\tau_1$ we have
\begin{align}\nonumber
  K^{(\ell)}(\tau_1,\tau_2) = &  \frac{\mathbf{1}_{\chi \,\mathrm{principal}}}{|\Gamma \backslash G|} \langle K^{(\ell)}(\cdot,\tau_2) ,1 \rangle_{\Gamma \backslash G} + \sum_{\substack{V\in \mathcal{B}(q,\chi)\\ j \equiv \kappa \pmod{2}}}  \langle    K^{(\ell)}(\cdot,\tau_2), \varphi_V^{(j)} \rangle  \varphi_V^{(j)}(\tau_1)   \\
 \label{eq:prop71proof1} &+\sum_{\mathfrak{c}\in \mathfrak{C}(q,\chi)} \sum_{j \equiv \kappa \pmod{2}} \frac{1}{4 \pi i} \int_{(0)} \langle  K^{(\ell)}(\cdot,\tau_2), E^{(j)}_{\mathfrak{c}}(\cdot,\nu) \rangle E^{(j)}_{\mathfrak{c}}(g,\nu) \d \nu.
\end{align}
By unfolding the first term matches the first term in Proposition \ref{prop:boundbykernel}. We use unfolding to compute
\begin{align*}
    \langle    K^{(\ell)}(\cdot,\tau_2), \varphi_V^{(j)} \rangle =& \int_{\Gamma \backslash G}  K^{(\ell)}(g,\tau_2) \overline{\varphi_V^{(j)}(g)} \d g = \sum_{\gamma \in \Gamma} \int_{\Gamma \backslash G} \overline{\chi}(\gamma) k_{\ell,\ell}(\tau_2^{-1} \gamma g) \overline{\varphi_V^{(j)}(g)} \d g \\
=& \sum_{\gamma \in \Gamma} \int_{\gamma (\Gamma \backslash G)}  k_{\ell,\ell}(\tau_2^{-1}  g) \overline{\varphi_V^{(j)}(g)} \d g =\int_{ G}  k_{\ell,\ell}(\tau_2^{-1}g) \overline{\varphi_V^{(j)}(g)} \d g.
\end{align*}
With the Haar measure decomposition $\d g=  \d \mathrm{a} \d\mathrm{n}\d\mathrm{k}$ we get
\begin{align*}
    \int_{ G}  k_{\ell,\ell}(\tau_2^{-1}g) \overline{\varphi_V^{(j)}(g)} \d g = &\int_{\mathrm{N}} \int_{\mathrm{A}} \int_{\mathrm{K}} k_{\ell,\ell}(\tau_2^{-1}\mathrm{a}\mathrm{n} \mathrm{k}) \overline{\varphi_V^{(j)}( \mathrm{a}\mathrm{n} \mathrm{k})}  \d \k \d \mathrm{a} \d \mathrm{n} \\
    =&\int_{\mathrm{N}} \int_{\mathrm{A}} k_{\ell,\ell}( \tau_2^{-1}\mathrm{a}\mathrm{n} ) \overline{\varphi_V^{(j)}(\mathrm{a} \mathrm{n} )} \int_{\mathrm{K}} e^{i(\ell - j )\theta(\mathrm{k})} \d \k \d \mathrm{a} \d \mathrm{n}.
\end{align*}
The innermost integral is $1_{\ell=j}$. We recall that the $\varphi_V^{(l)}$ are eigenfunctions by \eqref{eq:varphieigenfuncOmega} and apply Proposition \ref{prop:integraloperator} to get
\begin{align} \nonumber
   \langle    K^{(\ell)}(\cdot,\tau_2), \varphi_V^{(\ell)} \rangle &=  \int_{ G}  k_{\ell,\ell}(\tau_2^{-1}g) \overline{\varphi_V^{(\ell)}(g)} \d g \\
  \nonumber &=L_{\ell,\ell} \overline{\varphi_V^{(\ell)} (\tau_2)}\\
\label{eq:prop71proof2}   &= \Phi_{\ell,\ell}(k;\nu_V) \overline{ \varphi_{V}^{(\ell)}(\tau_2)}.
\end{align}
Similar computation holds for the Eisenstein series contribution. Combining \eqref{eq:prop71proof1} and \eqref{eq:prop71proof2} Proposition \ref{prop:spectralexpofkernel} follows.
\end{proof}

\subsection{Lemmas}
We now gather some lemmas for the proof of Proposition \ref{prop:integraloperator}. As mentioned, our argument is similar to \cite[Section 1.8]{IwaniecBook}.
Recall that $k_{\ell,\ell}(g)$  is of left and right type $\ell$  and
\begin{align*}
    L_{\ell,\ell} f (g) := \int_{G} k_{\ell,\ell}(g^{-1} h) f(h) \d h = \int_{G} k_{\ell,\ell}(g,h) f(h) \d h,
\end{align*}
where we denote
\begin{align*}
 k_{\ell,\ell}(g,h) := k_{\ell,\ell}(g^{-1} h). 
\end{align*}
Using the Cartan coordinates $g= \mathrm{k}[\varphi] \mathrm{a}[e^{-\varrho}]\mathrm{k}[\vartheta]$ and denoting for $u=u(g)$
\begin{align*}
  \cosh(\varrho) =  2u+1, \quad p_\ell (u) := k_{\ell,\ell}(\mathrm{a}[e^{-\varrho}])
\end{align*}
we have
\begin{align}\label{eq:KHpl}
    k_{\ell,\ell}(g) = e^{i\ell(\varphi + \vartheta)} p_\ell(u) =: H_{\ell}(g) p_\ell(u).
\end{align}
From \eqref{eq:casimircartan} we immediately get the following (compare to \cite[eq. (33.7)]{motohashielements}).
\begin{lemma}\label{le:omegaforfixedtypes}
Denoting
\begin{align*}
  \mathcal{D}_{u,\ell}:= u(u+1)\partial_{u}^2+(2u+1)\partial_{u}+\frac{\ell^2}{4(u+1)},
\end{align*}
we have
\begin{align*}
    \Omega     k_{\ell,\ell}(g)=H_{\ell}(g)   \mathcal{D}_{u,\ell}\, p_\ell (u).
\end{align*}
\end{lemma}

\begin{lemma} \label{le:omegahtog}
    We have 
    \begin{align*}
        \Omega_h k_{\ell,\ell}(g,h)=\Omega_g k_{\ell,\ell}(g,h)
    \end{align*}

\end{lemma}
\begin{proof}

Recall \eqref{eq:Omegalrcommute} and \eqref{eq:Omegaiotacommute}, that is, that $\Omega$ commutes with translations $r_h,l_h$ and the inversion $\iota$. Thus,
\begin{align*}
    \Omega_h k(g^{-1}h) = \Omega_h l_{g^{-1}} k(h) = l_{g^{-1}}  \Omega_h k(h) = (\Omega k) (g^{-1}h)
\end{align*}
and
\begin{align*}
    \Omega_g k(g^{-1}h) = \Omega_g  r_h\iota k(g) = r_h\iota \Omega_g k(g)  = (\Omega k) (g^{-1}h).
\end{align*}
\end{proof}
From this we get the following.
\begin{lemma}
 The operators $\Omega$ and $L_{\ell,\ell}$ commute.
\end{lemma}
\begin{proof}
By Lemma \ref{le:omegahtog} and \eqref{eq:omegaselfadjoint} we have
\begin{align*}
   \Omega  L_{\ell,\ell} f (g) &= \int  \Omega_g k_{\ell,\ell}(g,h) f(h) \d h  =\int  \Omega_h k_{\ell,\ell}(g,h) f(h) \d h  \\
 & 
   =\int   k_{\ell,\ell}(g,h)\Omega_h f(h) \d h  =     L_{\ell,\ell} \Omega f (g).
\end{align*}
\end{proof}

Given $h\in G$ we define
\begin{align*}
  \mathcal{R}_{h,\ell}  f(g) := \int_0^{2\pi} e^{i\ell \theta}  f(h \mathrm{k}[\theta] h^{-1} g) \frac{\d \theta}{2\pi}
\end{align*}
and state some simple facts for this operator.
\begin{lemma}\label{lem:f_g(g)=f(g)}
    If $f$ is of right-type $-\ell$,  then
    \begin{align*}
          \mathcal{R}_{g,\ell} f(g) = f(g).
    \end{align*}
    Furthermore, denoting
    \begin{align} \label{eq:FdefinitionR}
        F(g_1,g_2) :=  \frac{1}{f(g_2)} \int_0^{2\pi} e^{i\ell \theta} f(g_2 k[\theta] g_1) \frac{\d \theta}{2\pi},
    \end{align}
    we have
    \begin{align}\label{eq:lemf_g=f(g)eq1}
        F(h^{-1} g, h)=   \mathcal{R}_{h,\ell}f(g) f(h)^{-1} 
    \end{align}
    and $F(g_1,g_2)$ has type $(-\ell,-\ell)$ in $g_1$ and right-type $0$ in $g_2$.
\end{lemma}
\begin{proof}
 Since $f$ has right-type $-\ell$ we get 
 \begin{align*}
    \mathcal{R}_{g,\ell}  f(g) = \int_0^{2\pi} e^{i\ell \theta}  f(g \mathrm{k}[\theta] g^{-1} g) \frac{\d \theta}{2\pi} =   \int_0^{2\pi} f(g ) \frac{\d \theta}{2\pi} = f(g),
 \end{align*}
which proves the first claim. For the remaining part of the lemma, we observe that \eqref{eq:lemf_g=f(g)eq1} is immediate from the definitions and the type claim follows from
\begin{align*}
     F(\mathrm{k}[\theta_1] g_1\mathrm{k}[\theta_2] ,g_2 \mathrm{k}[\theta_3] ) = &\frac{1}{f(g_2 \mathrm{k}[\theta_3])} \int_0^{2\pi} e^{i\ell\theta} f(g_2  k[\theta +\theta_3 + \theta_1] g_1\mathrm{k}[\theta_2]) \frac{\d \theta}{2\pi}\\
     = &e^{-i\ell (\theta_1 + \theta_2) } \frac{1}{f(g_2)}  \int_0^{2\pi} e^{i\ell \theta} f(g_2 k[\theta] g_1) \frac{\d \theta}{2\pi}\\
     =& e^{-i\ell (\theta_1 + \theta_2) }F(g_1,g_2).
\end{align*}
\end{proof}
We next show that $L_{\ell,\ell}$ and $ \mathcal{R}_{h,\ell}$ commute.
\begin{lemma} \label{le:Llfg}
    We have
    \begin{align*}
        L_{\ell,\ell}   \mathcal{R}_{h,\ell} =   \mathcal{R}_{h,\ell} L_{\ell,\ell}.
    \end{align*}
    In  particular, we have
    \begin{align*}
        L_{\ell,\ell}   (\mathcal{R}_{g,\ell}f)(g) =   L_{\ell,\ell}   f(g).
    \end{align*}
\end{lemma}
\begin{proof}
    We have by invariance of the measure $\d h'$ and using $k(g_1,g_3^{-1}g_2) = k(g_3 g_1 ,g_2)$
    \begin{align*}
        L_{\ell,\ell}  \mathcal{R}_{h,\ell} f(g)&=\int_{G} k_{\ell,\ell}(g,h')\mathcal{R}_{h,\ell} f(h')\d h'\\
        &=\int_{G} k_{\ell,\ell}(g,h') \int_0^{2\pi} e^{i\ell\theta}f(h \mathrm{k}[\theta] h^{-1} h')\frac{\d \theta}{2\pi} \d h'\\
&=\int_0^{2\pi} e^{i\ell\theta}\int_{G} k_{\ell,\ell}(g,h') f(h \mathrm{k}[\theta] h^{-1} h') \d h' \frac{\d \theta}{2\pi}\\
        &= \int_0^{2\pi}  e^{i\ell\theta}\int_{G}k_{\ell,\ell}(g,(h \mathrm{k}[\theta] h^{-1})^{-1}h') f(h')  \d h' \frac{\d \theta}{2\pi}\\
        &=\int_0^{2\pi}  e^{i\ell\theta} \int_{G}  k_{\ell,\ell}(h \mathrm{k}[\theta] h^{-1} g,h') f(h')  \d h' \frac{\d \theta}{2\pi}\\
        &=\mathcal{R}_{h,\ell} 
 L_{\ell,\ell}f(g).
    \end{align*}
\end{proof}

Recall that
\begin{align*}
    \phi_{\ell}(g,\nu) = y^{\nu+1/2} e^{i\ell \theta}.
\end{align*}
\begin{lemma} \label{le:Uunique}
 Let $\ell \in \Z$, $\lambda \in \C$. There is a unique smooth function $U_{\lambda,\ell}:G \to \C$ which is of type $(\ell,\ell)$ satisfying
 \begin{align*}
  U_{\lambda,\ell}(1) = 1, \quad  \Omega U_{\lambda,\ell} = \lambda U_{\lambda,\ell}.
 \end{align*}
Denoting $\lambda = 1/4-\nu^2$, we have
 \begin{align}  \label{eq:Uformula}
   U_{\lambda,\ell}(g)  =  \int_0^{2\pi} \phi_\ell(\mathrm{k}[-\theta]g\mathrm{k}[\theta],\nu)  \frac{\d \theta}{2\pi}= \int_0^{2\pi} \phi_\ell(\mathrm{k}[-\theta]g\mathrm{k}[\theta],-\nu) \frac{\d \theta}{2\pi}.
 \end{align}
\end{lemma}
\begin{proof}
Recall that $H_\ell(g)=e^{i\ell (\varphi+\vartheta)}$ and that, just as as in \eqref{eq:KHpl}, we have for any $U_{\lambda,\ell}$ of type $(\ell,\ell)$ for some $P_{\lambda,\ell}$ that
\begin{align*}
    U_{\lambda,\ell}(g)=H_\ell(g)P_{\lambda,\ell}(u(g)).
\end{align*}
By Lemma \ref{le:omegaforfixedtypes} we see that $\Omega U_{\lambda,\ell}=\lambda U_{\lambda,\ell}$ is equivalent to 
\begin{align*}
(u(u+1) \partial_u^2 + (2u+1) \partial_u + \lambda +\frac{\ell^2}{4(u+1)}) P_{\lambda,\ell}(u) = 0.
\end{align*}
Multiplying this with $u/(u+1)$ (since $u > 0$) we get
\begin{align*}
    (u^2 \partial_u^2 + u\frac{(2u+1)}{u+1}\partial_u +\frac{u \lambda}{u+1} +\frac{u \ell^2}{4(u+1)^2})  P_{\lambda,\ell}(u)=0.
\end{align*}
Following the notation in \cite[Section 10.3]{whittakerwatson}), this second order differential equation is of the form 
\begin{align*}
   \bigg( (z-c)^2 \partial_z + (z-c) P(z) \partial_z + Q(z)\bigg) U(z) = 0
\end{align*}
where $c=0, P(z) = \frac{2z+1}{z+1}, Q(z) =\frac{z \lambda}{z+1} +\frac{z \ell^2}{4(z+1)^2} $. We then consider the indicial equation
\begin{align*}
    \alpha^2+(p_0-1)\alpha+q_0=0
\end{align*}
with
\begin{align*}
    p_0&=P(0)=1\\
    q_0&=Q(0)=0.
\end{align*}
In this case, the indicial equation has only the solution $\alpha=0$ and as such there exist one solution for the differential equation which is analytic at $u=0$ and one with a logarithmic singularity \cite[Section 10.32]{whittakerwatson}. By $U_{\lambda,\ell}(I)=1$ and the fact that for $g=I$ we have $u=0$, the function $U_{\lambda,\ell}$ is uniquely determined to originate from the analytic solution. The claim \eqref{eq:Uformula} is then clear since the integrals on the right-hand side define a type $(\ell,\ell)$ eigenfunction with eigenvalue $\lambda$. 
\end{proof}

\begin{lemma} \label{le:fhtoU}
Let $f$ be an eigenfunction of $\Omega$ with eigenvalue $\lambda$ with right-type $-\ell$. Then we have
    \begin{align*}
\mathcal{R}_{h,\ell}f(g) = U_{\lambda,-\ell}(h^{-1} g ) f(h),
    \end{align*}
    where $U_{\lambda,-\ell}$ is as in Lemma \ref{le:Uunique}.
\end{lemma}

\begin{proof}
By Lemma \ref{lem:f_g(g)=f(g)} we have
\begin{align*}
    \frac{\mathcal{R}_{h,\ell}f(g)}{f(h)} = F(h^{-1}g,h)
\end{align*}
with 
\begin{align}\label{eq:lem7.9F1h1}
    F(1,h)=\frac{\mathcal{R}_{h,\ell}f(h)}{f(h)}=1
\end{align}
and with $F$ having left and right type $-\ell$ in the first coordinate. Recall the definition of $F(g_1,g_2)$ in \eqref{eq:FdefinitionR} and observe that 
\begin{align*}
    \Omega_{g_1} F(g_1,g_2)=\lambda F(g_1,g_2)
\end{align*}
by the fact that the casimir element commutes with left group action. Thus, by \eqref{eq:lem7.9F1h1} and Lemma \ref{le:Uunique} we have for any $g_2$
\begin{align*}
  F(g_1,g_2)=U_{\lambda,-\ell}(g_1)
\end{align*}
and so
\begin{align*}
  \mathcal{R}_{h,\ell}f(g)=F(h^{-1}g,h) f(h) = U_{\lambda,-\ell}(h^{-1} g ) f(h).
\end{align*}
\end{proof}

\begin{lemma}\label{le:Llf=Lambdaf}
    Let $f \in L^2(G)$ be an eigenfunction of $\Omega$ with eigenvalue $\lambda$ with right-type $-\ell$. Then $f$ is also an eigenfunction of $L_{\ell,\ell}$ and we have
    \begin{align*}
        (L_{\ell,\ell} f)(g)= \Lambda(k,\lambda,\ell)f(g),
    \end{align*}
    where $\Lambda$ does not depend on $f$. Furthermore, for $\lambda=1/4-\nu^2$ we have
    \begin{align} \label{eq:Lambdaformula}
\Lambda(k,\lambda,\ell) = \int_G k_{\ell,\ell}(h) \overline{\phi_\ell(h,\nu)} \d h.
    \end{align}
\end{lemma}
\begin{proof}
    By Lemmas \ref{le:Llfg} and \ref{le:fhtoU}  we have
    \begin{align*}
        \Lambda  &= \frac{1}{f(g)} (L_{\ell,\ell}f)(g) = \frac{1}{f(g)}(\mathcal{R}_{g,\ell} L_{\ell,\ell} f)(g) = \frac{1}{f(g)}(L_{\ell,\ell}\mathcal{R}_{g,\ell}f)(g)   \\
       & =  \int_{G} k_{\ell,\ell} (g^{-1}h)  \frac{\mathcal{R}_{g,\ell}f(h)}{f(g)} \d h  \\
       & = \int_{G} k_{\ell,\ell} (g^{-1}h) U_{\lambda,-\ell}(g^{-1}h) \d h.
    \end{align*}
    This is independent of $f$ and $g$ by using the change of variables $h \mapsto gh$ and thus depends only on $k,\lambda,\ell$. The last claim follows by applying this with $f(g)= \overline{\phi_\ell(g,\nu)}$ and evaluating at $g=I$.
\end{proof}

\subsection{Proof of Proposition \ref{prop:integraloperator}}
\begin{proof}[Proof of Proposition \ref{prop:integraloperator}]

The function $\overline{f_{V,\ell}} (g)$ is an eigenfunction with eigenvalue $1/4-\nu^2$ and right-type $-\ell$. Therefore, by Lemma \ref{le:Llf=Lambdaf} and  by \eqref{eq:Lambdaformula} we get
\begin{align*}
    L_{k;\ell,\ell} \overline{f_{V,\ell}} (g) = \Phi_{\ell,\ell}(k;\nu) \overline{f_{V,\ell}} (g). 
\end{align*}

\end{proof}

\section{A mean value theorem for automorphic forms on $\Gamma \backslash G$} \label{sec:meanvalueauto}

The main goal in this section is to prove the following proposition. It is the essential ingredient to obtain the $\mathcal{K}$ terms in Theorems \ref{thm:mainblackbox} and \ref{thm:Gmainblackbox}.
\begin{proposition}\label{prop:boundbykernel}
Let $K,L \geq 1.$ There exists a function $k:G \to [0,1]$ satisfying
 \begin{align*}
     k(\smqty(a& b \\ c & d)) \leq \mathbf{1}_{|a|+|b|+|c|+|d| \leq 6}
 \end{align*}
 such that the following holds.  Let $T \subseteq G$ be a finite set and let $\alpha:T \to \C$ be a function. Then 
    \begin{align*}
     \sum_{\substack{V \in \B(q,\chi) \\  |\nu_V| \leq K }} \sum_{\substack{|\ell| \leq L  \\ \ell \equiv \kappa \pmod{2}}} \bigg| \sum_{\tau \in T }\alpha(\tau) \varphi_V^{(\ell)} (\tau)\bigg|^2& + \sum_{\mathfrak{c} \in \mathfrak{C}(q,\chi)} \sum_{\substack{|\ell| \leq L  \\ \ell \equiv \kappa \pmod{2}}} \int_{-iK}^{iK} \bigg| \sum_{\tau \in T }\alpha(\tau) E_{\mathfrak{c},\ell}(\tau,\nu)\bigg|^2 |\d \nu| \\
     & \ll K^2 L \sum_{\tau_1,\tau_2 \in T} \alpha(\tau_1) \overline{\alpha(\tau_2)} \sum_{\gamma \in \Gamma} \overline{\chi}(\gamma) k(\tau_2^{-1} \gamma \tau_1).
    \end{align*}
\end{proposition}

\begin{remark}
The above is a non-commutative analogue of the well-known bound
\begin{align*}
  \sum_{|k| \leq K} \bigg| \sum_{m} \alpha_m e( k x_m) \bigg|^2  \ll  K \sum_{m_1,m_2} |\alpha_{m_2} \alpha_{m_1}|F(x_{m_1} -x_{m_2}),
\end{align*}
where $F$ is a bounded function supported on $\|x_{m_1} -x_{m_2} \|_{\R / \Z} < 1/K$. In both of these the right-hand side measures the distribution of the points ($\tau$ or $x_m$) inside a fundamental domain ($\Gamma \backslash G$ or $\R/\Z$). Indeed this analogy is not too far fetched, as the proof below actually gives a stronger decay condition for $k$ for large $K,L$. We do not pursue this as it is inconsequential for our applications. For instance, the upper bound $   k(\smqty(a& b \\ c & d)) \leq \mathbf{1}_{|a|+|b|+|c|+|d| \leq 6}$ may be strengthened to  $   k(\smqty(a& b \\ c & d)) \leq \mathbf{1}_{|a|^2+|b|^2+|c|^2+|d|^2 \leq 2+1/K}$. 
\end{remark}

To prove Proposition \ref{prop:boundbykernel} we need some standard lemmas. We define the \emph{Abel transform} of a function $k:G \to \C$ by
\begin{align*}
 \mathcal{A}k: \mathrm{A} \to \C: \quad   \mathcal{A} k(\a) := y^{1/2} \int_{\N}    k(\a \n) \d \n, \quad \a=\a[y] 
\end{align*}
and the \emph{Mellin transform} of a function $f:\mathrm{A}\to \C$ for $s \in \C$
\begin{align*}
    \mathcal{M} f(s) := \int_\mathrm{A} f(\a) y^s \d \a = \int_0^\infty f(\a[y]) y^s \frac{\d y}{y} =  \int_\R f(\a[e^t]) e^{ts} \d t.  
\end{align*}
Recall the definition of the projection $k_{\ell,\ell}$ given in \eqref{eq:projectionll}.

\begin{lemma} \label{le:mellinabel}
Let $\Phi_{\ell,\ell}(k;\nu)$ be as in \eqref{eq:hfromkdefinition}. Then we have
\begin{align*}
    \Phi_{\ell,\ell}(k;\nu)=\mathcal{M}\A k_{\ell,\ell} (\overline{\nu}).
\end{align*}
\end{lemma}

\begin{proof}
  Recall $\phi_\ell( \a[y] \n[x] \k[\theta]) =\phi_\ell( \n[xy]\a[y]  \k[\theta])  =y^{1/2+\nu}e^{i \ell \theta}$ so that by \eqref{eq:integrationonG} 
    \begin{align*}
        \Phi_{\ell,\ell}(k;\nu) =& \int_G  k_{\ell,\ell}(h) \overline{\phi_{\ell}(h,\nu)} \d h= \int_{A} \int_N \int_K  k_{\ell,\ell}(\a \n \k) y^{1/2+\overline{\nu}}e^{-i \ell \theta}  \d \k \d \n  \d \a \\
        =&  \int_{A} y^{1/2}\int_N k_{\ell,\ell}(\a \n ) \d \n  y^{\overline{\nu}} \d \a=  \int_{A} \mathcal{A}k_{\ell,\ell}(\a) y^{\overline{\nu}} \d \a
        = \mathcal{M} \A k_{\ell,\ell} (\overline{\nu}).
    \end{align*}
\end{proof}

For $f_1,f_2 \in L^2(G)$ define the convolution
\begin{align*}
    f_1 \ast f_2 (h) := \int_G f_1( g^{-1} h) f_2 ( g ) dg =  \int_G f_1(g^{-1} ) f_2 (h g ) \d g .
\end{align*}
For functions  $f_1,f_2 \in L^2(\mathrm{A})$ we define the convolution via
\begin{align*}
   f_1 \ast_{\mathrm{A}} f_2 (\mathrm{a}) := \int_{\mathrm{A}} f_1(\mathrm{a}_0^{-1} \mathrm{a}) f_2(\mathrm{a}_0)  \d \mathrm{a}_0.
\end{align*}
These convolutions interact with the Abel transform in the following way.
\begin{lemma} \label{le:abelconvo}
Let $f_1,f_2:G\to\C$ and suppose that the left type of $f_1$ is the same as the right type of $f_2$.   Then
    \begin{align*}
        \mathcal{A}(   f_1 \ast f_2 ) =     \mathcal{A} f_1 \ast_{\mathrm{A}} \mathcal{A} f_2.
    \end{align*}
\end{lemma}
\begin{proof}
    We have for $\mathrm{a}= \mathrm{a}[y]$
    \begin{align*}
  y^{-1/2}   \mathcal{A}(   f_1 \ast f_2 ) (\mathrm{a}) =& \int_\mathrm{N} ( f_1 \ast f_2)(\mathrm{a} \mathrm{n}) \d \mathrm{n} \\
    =&  \int_\mathrm{N} \int_G f_1(g^{-1} \mathrm{a} \mathrm{n}) f_2 ( g ) \d g \d \mathrm{n} \\
    =& \int_\mathrm{N} \int_G f_1(g^{-1}  \mathrm{n}) f_2 (\mathrm{a} g ) \d g \d \mathrm{n}. 
    \end{align*}
    Using the Iwasawa decomposition $g=\a_0 \n_0 \k_0$ and the type assumption we get (denoting $\a_0=\a[y_0]$)
    \begin{align*}
       y^{-1/2}   \mathcal{A}(   f_1 \ast f_2 ) (\mathrm{a})&=  \int_\mathrm{N} \int_\mathrm{A} \int_\mathrm{N}  \int_\mathrm{K}   f_1(\mathrm{k}_0^{-1 }\mathrm{n_0}^{-1} \mathrm{a_0}^{-1}  \mathrm{n}) f_2 (\mathrm{a} \mathrm{a_0} \mathrm{n_0} \mathrm{k}_0)\d \mathrm{k}_0 \d \mathrm{n}_0 \d \mathrm{a_0} 
 \d \mathrm{n} \\ &=   \int_\mathrm{N} \int_\mathrm{A} \int_\mathrm{N}  f_1(\mathrm{n_0}^{-1} \mathrm{a_0}^{-1}  \mathrm{n}) f_2 (\mathrm{a} \mathrm{a_0} \mathrm{n_0} ) \d \mathrm{n}_0 \d \mathrm{a_0} 
 \d \mathrm{n}  \\
 &=\int_\mathrm{N} \int_\mathrm{A} \int_\mathrm{N} y_0^{-1}  f_1( \mathrm{n_0}^{-1}  \mathrm{n}\mathrm{a_0}^{-1} ) f_2 (\mathrm{a} \mathrm{a_0} \mathrm{n_0} ) \d \mathrm{n}_0 \d \mathrm{a_0} 
 \d \mathrm{n} \\
 &=\int_\mathrm{N} \int_\mathrm{A} \int_\mathrm{N} y_0^{-1} f_1( \mathrm{n} \mathrm{a_0}^{-1} ) f_2 (\mathrm{a} \mathrm{a_0} \mathrm{n_0} ) \d \mathrm{n}_0 \d \mathrm{a_0}  \d \mathrm{n}\\
 &=\int_\mathrm{N} \int_\mathrm{A} \int_\mathrm{N} f_1( \mathrm{a_0}^{-1}\mathrm{n}  ) f_2 (\mathrm{a} \mathrm{a_0} \mathrm{n_0} ) \d \mathrm{n}_0 \d \mathrm{a_0}  \d \mathrm{n} \\
 &=\int_\mathrm{N} \int_\mathrm{A} \int_\mathrm{N}  f_1( \mathrm{a_0}^{-1} \mathrm{a}  \mathrm{n}) f_2 (\mathrm{a_0} \mathrm{n_0} ) \d \mathrm{n}_0 \d \mathrm{a_0}  \d \mathrm{n}.
    \end{align*}
Thus, we have
\begin{align*}
    \mathcal{A}(   f_1 \ast f_2 ) (\mathrm{a}) =  &\int_\mathrm{A} y^{1/2} \int_\mathrm{N}  f_1( \mathrm{a_0}^{-1} \mathrm{a}  \mathrm{n}) \d \mathrm{n} \int_\mathrm{N}   f_2 (\mathrm{a_0} \mathrm{n_0} ) \d \mathrm{n}_0 \d \mathrm{a_0} \\
    =&   \int_\mathrm{A} \bigg((y/y_0)^{1/2}\int_\mathrm{N}  f_1( \mathrm{a_0}^{-1} \mathrm{a}  \mathrm{n}) \d \mathrm{n} \bigg)  \bigg(y_0^{1/2}\int_\mathrm{N}   f_2 (\mathrm{a_0} \mathrm{n_0} ) \d \mathrm{n}_0 \bigg)\d \mathrm{a_0} \\
= & \mathcal{A} f_1 \ast_{\mathrm{A}} \mathcal{A} f_2 (\mathrm{a}).  
\end{align*}

\end{proof}

The above lemmas allow us to construct a non-negative $\Phi$-function via self-convolution of a kernel.
\begin{lemma} \label{le:convosquare}
    Let $k_0\in L^2(G)$ be a real valued function with $k_0(\mathrm{k} g \mathrm{k}^{-1})= k_0(g)$ and  $k_0(g^{-1})=k_0(g)$. Then for $k=k_0 \ast k_0$ we have for $\nu \in \R \cup i\R$ 
    \begin{align*}
        \Phi_{\ell,\ell}(k,\nu) = |\Phi_{\ell,\ell}(k_0,\nu) |^2.
    \end{align*}
\end{lemma}
\begin{proof}
By $k_0(\mathrm{k} g \mathrm{k}^{-1})= k_0(g)$ we have
  \begin{align*}
   k_0(g) = \sum_{\ell \in \Z}    k_{0,\ell,\ell}(g) , \quad     k(g) = \sum_{\ell \in \Z}    k_{0,\ell,\ell} \ast  k_{0,\ell,\ell}(g),
  \end{align*}
  since for $\ell\neq \ell'$ we have by Cartan decomposition $k_{0,\ell',\ell'} \ast k_{0,\ell,\ell}\equiv 0$. Hence, we have $k_{\ell,\ell} =k_{0,\ell,\ell} \ast k_{0,\ell,\ell} $ and by Lemmas \ref{le:mellinabel} and \ref{le:abelconvo} we get
  \begin{align*}
       \Phi_{\ell,\ell}(k,\nu) = \mathcal{M} \A  (k_{0,\ell,\ell} \ast k_{0,\ell,\ell}) (\overline{\nu}) =  \mathcal{M} (\A  k_{0,\ell,\ell}  \ast_\mathrm{A} \A k_{0,\ell,\ell})) (\overline{\nu}) = (\mathcal{M} \A  k_{0,\ell,\ell}  (\overline{\nu}) )^2,
  \end{align*}
  where $\mathcal{M} \A  k_{0,\ell,\ell}  (\overline{\nu})=\mathcal{M} \A  k_{0,\ell,\ell}  (\nu)$ is real valued by the assumption that $k_0(g^{-1})=k_0(g)$.
\end{proof}
Next we have a simple statement relating sizes of Iwasawa coordinates and Cartan coordinates near the identity.
\begin{lemma} \label{le:yvsvarrho}
  Let $K > 1$ and 
    \begin{align*}
        \a[e^t] \n[x] = \k[\varphi] \a[e^{-\varrho}] \k[\vartheta].
    \end{align*}
    If $|\varrho|  \leq 1/K$, then $|t| \ll 1/K$ and  $|x| \ll 1/K$. Similarly, if $|t| \leq 1/K$ and  $|x| \leq 1/K$, then $|\varrho|  \ll 1/K$
\end{lemma}
\begin{proof}
    By \eqref{eq:cartanAbel} we have for $y= e^t$
    \begin{align*}
    1 \leq  \cosh t  \leq   \cosh t + e^{t} x^2  /2 =  \cosh \varrho  \leq 1+O(1/K^2).
    \end{align*}
   By  $\cosh t  = 1+t^2/2 + O(t^4)$ this implies that $ |t| \ll 1/K$ and $|x| \ll 1/K$. The other direction follows similarly.
\end{proof}
 \begin{proof}[Proof of Proposition \ref{prop:boundbykernel}]
 We construct a smooth compactly supported function $k:G\to\C$ such that $ \Phi_{\ell,\ell}(k,\nu)$ takes non-negative real values and
 \begin{align*}
      \Phi_{\ell,\ell}(k,\nu) \geq \mathbf{1}_{|\nu| \leq K} \mathbf{1}_{|\ell| \leq L}.
 \end{align*}
By Lemma \ref{le:convosquare} it suffices to choose $k=k_0*k_0$ such that $\Phi_{\ell,\ell}(k_0,\nu)\geq 1$ for $|\nu|\leq K, |\ell| \leq L$. 
 We let $C>0$ denote a large constant and choose
 \begin{align*}
     k_0(g) :=  C^4 K^2 L f(\varrho)  F(\varphi + \vartheta), \quad g = \k[\varphi]\a[e^{-\varrho}] \k[\vartheta],
 \end{align*}
where $f:\R \to [0,1]$ is a smooth even function with
\begin{align*}
   f(x)&=1, \, x \in  [-(2CK)^{-1},(2CK)^{-1}],\quad  \mathrm{supp}(f) \subseteq  [-(CK)^{-1},(CK)^{-1}],
\end{align*}
and $F:\R \to [0,1]$ is a $2\pi$-periodic even function with
\begin{align*}
       F(x)&=1, \, x \in  [-(2CL)^{-1},(2CL)^{-1}], \\
        F(x+\pi)&=(-1)^\kappa F(x),  \\
       \mathrm{supp}(F) &\subseteq [-(CL)^{-1},(CL)^{-1}] \cup \bigg(\pi + [-(CL)^{-1},(CL)^{-1}]\bigg).
\end{align*}
Then 
\begin{align*}
    k_{0,\ell,\ell}(g) = C^4 K^2 L \widehat{F}(\ell) e^{i\ell (\varphi + \vartheta)} f(\varrho), \quad \quad  \widehat{F}(\ell) = \frac{1}{2\pi}\int_{-\pi}^{\pi} F(\theta) e^{-i\ell \theta} \d \theta.
\end{align*}
By Taylor approximation we have for $|\ell| \leq L$ with $\ell \equiv \kappa \pmod{2}$
\begin{align*}
    \widehat{F}(\ell) &= \frac{1}{2\pi}\int_{-\pi}^{\pi} F(\theta) e^{-i\ell \theta} \d \theta =   \frac{1}{\pi}\int_{-(CL)^{-1}}^{(CL)^{-1}} F(\theta) (1+O( L \theta )) \d \theta   \\&=  \frac{1}{\pi}\int_{-(CL)^{-1}}^{(CL)^{-1}} F(\theta) (1+O( 1/C )) \d \theta  
    \geq  \frac{1}{\pi}\int_{-(2CL)^{-1}}^{(2CL)^{-1}} 1 \d \theta + O(C^{-2} L^{-1})  \geq (4C L)^{-1}
\end{align*}
if $C$ is sufficiently large. Furthermore, $  \widehat{F}(\ell) = 0$ for $\ell \not \equiv \kappa \pmod{2}$.

Similarly, for  $|\varrho| \leq (CK)^{-1}$ with $\a[e^t] \n[x] = \k_1 \a[e^{-\varrho}] \k_2$, $y=e^t$, we have by Lemma \ref{le:yvsvarrho} that for sufficiently large $C$
\begin{align*}
    \max\{ |t|, |x| \} \leq (C^{3/4} K)^{-1}.
\end{align*}
Therefore, for $|\nu| \leq K$ we get by Taylor approximation
\begin{align*}
    \int_{A}\int_N f(\varrho) y^{1/2+\nu} \d \n \d \a&=    \int_{\R}\int_\R f(\varrho) e^{t(1/2+\nu)} \d x \d t \\
    &=  \int_{-(C^{3/4}K)^{-1} }^{(C^{3/4}K)^{-1} }\int_{-(C^{3/4}K)^{-1} }^{(C^{3/4}K)^{-1} }f(\varrho) e^{t(1/2+\nu)} \d x \d t \\
    &=\int_{-(C^{3/4}K)^{-1} }^{(C^{3/4}K)^{-1} }\int_{-(C^{3/4}K)^{-1} }^{(C^{3/4}K)^{-1} }f(\varrho) (1+O(C^{-3/4})) \d x \d t \\
    & \geq \int_{-(C^{3/4}K)^{-1} }^{(C^{3/4}K)^{-1} }\int_{-(C^{3/4}K)^{-1} }^{(C^{3/4}K)^{-1} } \mathbf{1}_{|\varrho| \leq (2CK)^{-1}}  \d x \d t + O(C^{-9/4} K^{-2}) \\
    & \gg (CK)^{-2},
\end{align*}
where the last bound holds for $C$ sufficiently large. Combining the above, we have for $|\ell| \leq L, |\nu| \leq K$ that
\begin{align*}
    \Phi_{\ell,\ell}(k_0,\nu)=C^4 K^2 L \widehat{F}(l)\int_{A}\int_N f(\varrho) y^{1/2+\nu} \d \n \d \a  \geq C^{1/2} \geq 1,
\end{align*}
 if we, once more, assume that $C$ is sufficiently large. By positivity and Proposition \ref{prop:spectralexpofkernel} we get
 \begin{align*}
     &\sum_{\substack{V \in \B(q,\chi) \\  |\nu_V| \leq K }} \sum_{\substack{|\ell| \leq L  \\ \ell \equiv \kappa \pmod{2}}} \bigg| \sum_{\tau \in T }\alpha(\tau) \varphi_V^{(\ell)} (\tau)\bigg|^2 + \sum_{\mathfrak{c} \in \mathfrak{C}(q,\chi)} \sum_{\substack{|\ell| \leq L  \\ \ell \equiv \kappa \pmod{2}}}\frac{1}{4 \pi} \int_{-iK}^{iK} \bigg| \sum_{\tau \in T }\alpha(\tau) E_{\mathfrak{c},\ell}(\tau,\nu)\bigg|^2 |\d \nu| 
      \\
      \leq  & \,\frac{\mathbf{1}_{\chi \,\mathrm{principal}} \mathbf{1}_{\ell=0}}{|\Gamma \backslash G|} \int_G k(g) \d g+\sum_{\substack{V \in \B(q,\chi)  }} \sum_{\substack{\ell \equiv \kappa \pmod{2}}} \Phi_{\ell,\ell}(k,\nu)\bigg| \sum_{\tau \in T }\alpha(\tau) \varphi_V^{(\ell)} (\tau)\bigg|^2  \\
      & \hspace{80pt}+ \sum_{\mathfrak{c} \in \mathfrak{C}(q,\chi)} \sum_{\substack{ \ell \equiv \kappa \pmod{2}}} \frac{1}{4 \pi i} \int_{(0)} \Phi_{\ell,\ell}(k,\nu) \bigg| \sum_{\tau \in T }\alpha(\tau) E_{\mathfrak{c},\ell}(\tau,\nu)\bigg|^2 \d \nu \\
      = & \sum_{\tau_1,\tau_2 \in T} \alpha(\tau_1) \overline{\alpha(\tau_2)}\sum_{\gamma \in \Gamma} \overline{\chi}(\gamma)k(\tau_2^{-1} \gamma \tau_1).
 \end{align*}
It remains to show that
\begin{align*}
    k(g) \ll \mathbf{1}_{|a|+|b|+|c|+|d| \leq 6}.
\end{align*}
Denoting $g= \k[\varphi] \a[e^{-\varrho}] \k[\vartheta], h= \k[\varphi_0] \a[e^{-\varrho_0}] \k[\vartheta_0],$ $k_1(g)=f(\varrho) F(\varphi+\vartheta)$, we have
 \begin{align*}
     k(g) &= (k_0 \ast k_0) (g) = C^8 K^{4} L^2 \int_G k_1(h^{-1} g) k_1(h) \d h \\
     &\leq  C^8 K^{4} L^2\int_G f(h^{-1} g) k_1(h) \d h 
     \\
     &\ll  C^8 K^{4} L^2 \int_{K \times  K} \int_{\R_{>0}} f(\varrho') F(\varphi_0 + \vartheta_0 ) f(\varrho_0) \sinh(\varrho_0)\d \varrho_0\d \k[\varphi_0 ]  \d \k[\vartheta_0 ] \\
    & \ll C^7 K^{4} L \int_{K} \int_{\R_{>0}} f(\varrho')  f(\varrho_0) \sinh(\varrho_0)\d \varrho_0\d \k[\varphi_0 ] , 
 \end{align*}
 where for some $\k_1,\k_1',\k_2'$ we have 
 \begin{align*}
     \k_1' \a[e^{-\varrho}] \k_2' = \a[e^{-\varrho_0}]\k_1 \a[e^{-\varrho'}].
 \end{align*}
 Here $\varrho_0,\varrho'$ are supported on $[-(CK)^{-1},(CK)^{-1}]$ and 
 \begin{align*}
     \cosh(\varrho) &= 1+ 2u( \a[e^{-\varrho_0}]\k_1 \a[e^{-\varrho'}] i,i) = 1+ 2u( \k_1 \a[e^{-\varrho'}] i,\a[e^{\varrho_0}]i) \\
     &= 1+ 2u( \k_1  e^{-\varrho'}i, e^{\varrho_0}i) =  1+ 2u( i+ O((CK)^{-1}), i+ O((CK)^{-1})) \\
     & = 1+ O((CK)^{-2}),
 \end{align*}
which implies that $k(g)$ is supported on $|\varrho| \ll (CK)^{-1}$.
Therefore, using $\sinh(\varrho_0) \ll |\varrho_0|$ we get
 \begin{align*}
      k(g) \ll C^5 K^2 L \mathbf{1}_{|\varrho| \ll (CK)^{-1}}.
 \end{align*}
Taking $C$ sufficiently large this is by \eqref{eq:ulowerbound} supported on $a^2+b^2+c^2+d^2 \leq 2+1/100$ which implies  $|a| + |b| +|c| +|d| \leq 6$ since $4 \cdot \sqrt{2+1/100} < 6$.
 \end{proof}

\section{Discrete average of Poincar\'e series on $\Gamma \backslash G$}\label{sec:discaverage}
The goal in this section is to combine the previous sections to prove the following technical result, which is a precursor of Theorems \ref{thm:mainblackbox} and \ref{thm:twisteddetwbound}. Recall that
\begin{align*}
 \Gamma_2(q_1,q_2) = \bigg \{ \mqty(a & b \\ c &d) \in \SL( \Z):  q_1 | b, \, q_2 | c \bigg\} = \a[q_1] \Gamma_0(q_1q_2) \a[q_1]^{-1}. 
\end{align*}
For this group we define the Hecke operator by
\begin{align} \label{eq:gamma2heckedef}
  \mathcal{T}_h f(g) :=  \frac{1}{\sqrt{h}} \sum_{ad=h} \chi(a) \sum_{b  \pmod{d}} f\bigg(\frac{1}{\sqrt{h}}\mqty(a & b q_1 \\ & d)g\bigg),
\end{align}
which by conjugation by $\a[q_1]$ corresponds to the one for $\Gamma_0(q_1q_2)$ defined in \eqref{eq:heckeopdef}. For $V \in \B(q,\chi)$, $t >0$, and coefficients $\beta_h$ we denote
\begin{align}\label{eq:RVdef}
  R_V(t,\beta) := \sum_{\sigma\in \{\pm\}}\sum_{\gcd(h,q)=1} \beta_h \lambda_V(h) \sum_{0< |n| \leq t} \varrho_V(\sigma n)  
\end{align}
and similarly for $\nu \in \C$ and $\c \in \mathfrak{C}(q,\chi)$ define
\begin{align}\label{eq:Rcnudef}
     R_{\c,\nu }(t,\beta) :=   \sum_{\sigma\in \{\pm\}}\sum_{\gcd(h,q)=1} \beta_h \lambda_{\c,\nu}(h) \sum_{0< |n| \leq t} \varrho_{\c,\nu}( \sigma n). 
\end{align}

\begin{theorem} \label{thm:Gmainblackbox}
  Let $\Gamma=\Gamma_2(q_1,q_2)$ for some $q_1,q_2 \geq 1$, $q=q_1q_2$, and let $T \subseteq G= \SL(\R)$ be a finite set. Let $\alpha: T \to \C$ and $\beta:\Z_{> 0} \to \C$. Let $\chi$ be a Dirichlet character to the modulus $q$ and denote the corresponding group character on $\Gamma_2(q_1,q_2)$ by $\overline{\chi}(\smqty(a & b \\ c & d)) :=  \overline{\chi}(d)$. 

Let $A,C,D >0$, $\delta \in (0,1)$ with $AD > \delta$. Let $J \geq 3$, $f \in C_\delta^{2J+1}(A,C,D)$ and let $F: G \to \C$ denote  $F(\smqty(a&b \\ c& d))=f(a,c,d)$.
Denote 
\begin{align*}
    P_{f}(\tau) := \sum_{\gamma \in \Gamma} \overline{\chi}(\gamma) F(\gamma \tau), 
\end{align*}
\begin{align*}
   S_{f,\alpha,\beta,\chi}(\Gamma,T):= \sum_{\tau \in T}  \sum_{\gcd(h,q)=1} \alpha(\tau) \beta_h \mathcal{T}_h P_{f}(\tau), 
\end{align*}
and
\begin{align}\label{eq:GBB,M}
     M_{f,\alpha,\beta,\chi}(\Gamma,T)  &:= \mathbf{1}_{\chi \, \mathrm{principal}}\bigg(\sum_{\gcd(h,q)=1}\frac{\beta_h \sigma_1(h)}{\sqrt{h}}\bigg)\bigg(\sum_{\tau \in T} \alpha(\tau) \bigg) \frac{1}{|\Gamma \backslash G|}  \int_{G} F(g) \d g.
\end{align}
Let 
\begin{align*}
   Z  :=\max\{ A^{\pm 1}, C^{\pm 1}, D^{\pm 1}, \delta^{-1}\} .
\end{align*}
Define the decay function
\begin{align*}
    E_J(t,\nu):= \frac{1}{ (1+ |t|+|\nu|)^{2J-1}}.
\end{align*} 
Then for  $L:=C/D$, $N:=q_1\frac{CD}{AD+1},$ and for any $\eps > 0$ we have
\begin{align}\label{eq:Gmaintarget}
   | S_{f,\alpha,\beta,\chi}(\Gamma,T)  -   M_{f,\alpha,\beta,\chi}(\Gamma,T) |  \ll_{\eps,J}  Z^\eps \delta^{-O(1)}    (AD)^{1/2} \mathcal{K}_{\alpha,\chi}(\Gamma,T,L)^{1/2}\mathcal{R}_{\beta,\chi}(q_1,q_2,N)^{1/2},
\end{align}
where 
\begin{align*}
\mathcal{R}_{\beta,\chi}(q_1,q_2,N) := &\max_{t > 0}\frac{1}{N} \sum_{\substack{V \in\mathcal{B}(q,\chi)}} E_J(t/N,\nu_V)   \bigg( 1+\mathbf{1}_{\Re(\nu_V) < 1/2}\bigg(\frac{CD q_1 }{t}\bigg)^{2 \Re(\nu_V)} \bigg) |R_V(t,\beta)|^2 \\
&+ \max_{t > 0} \frac{1}{N} \sum_{\c \in \mathfrak{C}(q,\chi)}\int_{(0)} E_J(t/N,\nu)  |R_{\c,\nu }(t,\beta)|^2 |\d \nu| \, + \,\|\beta \|_1^2  N^{-1}
\end{align*}
with the two $R$ terms given by \eqref{eq:RVdef} and \eqref{eq:Rcnudef}.
Further
\begin{align*}   \mathcal{K}_{\alpha,\chi}(\Gamma,T,L):=   \sum_{\tau_1, \tau_2 \in T} \alpha(\tau_1)  \overline{\alpha(\tau_2)}\sum_{\gamma \in \Gamma} \overline{\chi}(\gamma) k_L(\tau_2^{-1} \gamma\tau_1  )
\end{align*}
for some $k_L:G \to [0,1]$ which satisfies  
\begin{align*}
  k_L(\smqty(a& b \\ c & d)) \leq \mathbf{1}_{|a|+|b| L +|c|/L+|d| \leq 6}.
\end{align*}
\end{theorem}

\begin{remark}
    Here we have assumed that $(h,q)=1$. The case $(h,q)> 1$ can be sometimes be reduced to the coprime case. For instance, if $(h,q)= r$, then
\begin{align*}
    \sum_{\substack{ ad-bc = h \\c \equiv 0 \pmod{q}}} f(a,c,d)= \sum_{r_1 r_2 = r}   \sum_{\substack{ ad-bc = h/r \\c \equiv 0 \, (q/r)}} f(ar_1,c r,dr_2).
\end{align*}
Another option is to factorize $h=kh'$ with $k=(h,q^\infty)$ and incorporate for $k$ the Hecke orbits into the set of orbits $T$, as is done in the proof of Theorem \ref{thm:twisteddetwbound}. 
\end{remark}

Note that the ranges $A,C,D$ can be smaller than $1$, since in practice they correspond to the normalized variables $a/\sqrt{h}, c/\sqrt{h},d/\sqrt{h}$ for the determinant equation $ad-bc=h$. It is even possible that $AD < 1$, in the case that $h$ is bigger than the original range for $ad$. 

We begin the proof of Theorem \ref{thm:Gmainblackbox} by first reducing to the case $\Gamma=\Gamma_0(q)$ and $L=1$.
\subsection{Reduction to the case $\Gamma=\Gamma_0(q)$}
Recall that
\begin{align} \label{eq:groupconjugation}
 \Gamma_2(q_1,q_2) = \mathrm{a}[q_1]  \Gamma_0(q_1q_2) \mathrm{a}[1/q_1]   
\end{align}
Thus, denoting
\begin{align*}
    T_0 :=  \mathrm{a}[1/q_1] T, \quad f_0(a,c,d)=F_0(g) := F(\mathrm{a}[q_1] g) = f(a \sqrt{q_1}, c/ \sqrt{q_1}, d/ \sqrt{q_1}), \quad \alpha_0(\tau) = \alpha( \mathrm{a}[q_1]\tau), \\
    A_0 = \frac{A}{\sqrt{q_1}}, \quad C_0 = C \sqrt{q_1}, \quad D_0 = D \sqrt{q_1}, \quad L_0 = L = C/D, \quad N_0 = \frac{C_0D_0}{A_0D_0 +1} =  \frac{CD q_1 }{AD +1} = N, 
\end{align*}
we have
\begin{align*}
  P_{f}(\tau) := \sum_{\gamma \in \Gamma_2(q_1,q_2)} \overline{\chi}(\gamma) F(\gamma \tau) =  \sum_{\gamma \in \Gamma_0(q_1q_2)} \overline{\chi}(\gamma) F(\mathrm{a}[q_1] \gamma\mathrm{a}[1/q_1] \tau) =   P_{f_0}(\tau_0 ) , \quad \tau_0 \in T_0,
\end{align*}
where $f_0$ satisfies the assumptions of Theorem \ref{thm:Gmainblackbox} with $(A,C,D)$ replaced by $(A_0,C_0,D_0)$
Thus, the case $\Gamma = \Gamma_2(q_1,q_2)$ follows from the case $\Gamma = \Gamma_0(q_1q_2)$, after observing that by \eqref{eq:groupconjugation}
\begin{align*}
    \mathcal{K}_{\alpha_0,\chi}(\Gamma_0(q_1q_2),T_0,L_0) = \mathcal{K}_{\alpha,\chi}(\Gamma_2(q_1,q_2),T,L).
\end{align*}
Note also that the conjugation by $\a[q_1]$ maps the Hecke operator \eqref{eq:gamma2heckedef} for $\Gamma_2(q_1,q_2)$ to the one defined in \eqref{eq:heckeopdef} for $\Gamma_0(q_1q_2)$.

\subsection{Rescaling the smooth weight}
The next step of the proof is to rescale the weight such that it suffices to consider the case $C=D$. This has as consequence that in the sequel we only need to consider $\ell$ small, more precisely, $|\ell|\ll \delta^{-1} Z^\eps$. Without this rescaling, losses from $L$ as in Proposition \ref{prop:boundbykernel} would be too costly for many applications. 

Recall we are given 
\begin{align*}
 F(g)=f(a,c,d) ,\quad g = \mqty(a & b \\ c& d)  
\end{align*}
with $f( a, c, d)$ supported for $(|a|,|c|,|d|) \in [A,2A]\times [C,2C]\times [D,2D]$, satisfying
\[
\partial_a^{J_1} \partial_c^{J_2} \partial_d^{J_3} f \ll_{J_1,J_2,J_3} (\delta A)^{-J_1} (\delta C)^{-J_2} (\delta D)^{-J_3}
\]
for any choice of $J_1,J_2,J_3 \leq 4J+1$.
Let $L:= C/D$ and define
\begin{align*}
    F_0(g):=F(g \a[L]),
\end{align*}
then 
\begin{align*}
    F_0(g)=f(a \sqrt{L}, c\sqrt{L}, d/\sqrt{L})
\end{align*}
is supported on $[A_0,2A_0]\times [C_0,2C_0]\times [D_0,2D_0]$ with 
\begin{align*}
    A_0&=A/\sqrt{L}\\
    C_0&=C/\sqrt{L}\\
    D_0&=D \sqrt{L}.
\end{align*}
In particular $C_0/D_0=1$ and we still have
\begin{align}\label{eq:f0derivbound}
    \partial_a^{J_1} \partial_c^{J_2} \partial_d^{J_3} F_0 \ll_{J_1,J_2,J_3} (\delta A_0)^{-J_1} (\delta C_0)^{-J_2} (\delta D_0)^{-J_3}.
\end{align}
Then 
\begin{align*}
    P_f(\tau) = P_{f_0} (\tau \a[L^{-1}]),
\end{align*}
so that by replacing the set $T$ with $T \a[L^{-1}]$,  we see that $P_{f_0}$ and $T \a[L^{-1}]$ fulfill the conditions of Theorem \ref{thm:Gmainblackbox} with $L=1$. Note that $AD$, $CD$, and $N=\frac{CD}{AD+1}$ are all invariant under the above rescaling of $L$. Furthermore, $k_L(g)$ is of the form
\begin{align*}
    k_L(g) = k_1( \a[L] g \a[L^{-1}])
\end{align*}
with
\begin{align*}
k_1(\smqty(a & b \\ c& d)) \ll 1_{|a|+|b|+|c|+|d|\leq 6}.
\end{align*}
Thus, it suffices to prove Theorem \ref{thm:Gmainblackbox} with $L=C/D =1$.

\subsection{Proof of Theorem \ref{thm:Gmainblackbox} for $\Gamma=\Gamma_0(q)$ and $L=C/D=1$}
Recall that the Iwasawa coordinates of $g = \smqty(a & b \\ c & d)$ are given by \eqref{eq:matrixtoiwasawa}
 \begin{align*}
     x= \frac{ac+bd}{c^2+d^2}, \quad y = \frac{1}{c^2+d^2}, \quad \theta = \arctan(-c/d).
 \end{align*}
 We claim that $F(g) =F(\n[x]\a[y]\k[\theta])$ satisfies the assumptions of Proposition \ref{prop:Jacquetinnerproduct} with
 \begin{align*}
     X \asymp\frac{AD +1}{CD} = \frac{1}{N}\quad \text{and} \quad Y \asymp \frac{1}{CD}, \quad \delta \mapsto \delta_1 := \min \{\delta,\delta AD \} \geq \delta^2. 
 \end{align*}
 Indeed, denoting $(a,c,d)=(\phi_1(x,\theta,y),\phi_2(x,\theta,y),\phi_3(x,\theta,y))$, using \eqref{eq:iwasawatomatrix} we compute for any $J_1+J_2+J_3 \leq 2J+1$
\begin{align*}
      \partial_x^{J_1} \partial_y^{J_2}  \partial_\theta^{J_3} F(\n[x]\a[y]\k[\theta]) = &\bigg(\frac{\partial \phi_1}{\partial x} \partial_a +\frac{\partial \phi_2}{\partial x} \partial_c +\frac{\partial \phi_3}{\partial x} \partial_d \bigg)^{J_1} \bigg(\frac{\partial \phi_1}{\partial y} \partial_a +\frac{\partial \phi_2}{\partial y} \partial_c +\frac{\partial \phi_3}{\partial y} \partial_d \bigg)^{J_2}  \\
      & \times \bigg(\frac{\partial \phi_1}{\partial \theta} \partial_a +\frac{\partial \phi_2}{\partial  \theta} \partial_c +\frac{\partial \phi_3}{\partial  \theta} \partial_d \bigg)^{J_3}f(a,c,d) \\
      \ll_{J_1,J_2,J_3} & \delta^{-J_1}\bigg(  \frac{1}{Y^{1/2}   A }\bigg)^{J_1} \delta^{-J_2}\bigg(\frac{1}{Y^{1/2}  A} + \frac{X}{Y^{3/2} A } + \frac{1}{Y^{3/2} C} +\frac{1}{Y^{3/2} D}\bigg)^{J_2} \\
      & \times \delta^{-J_3}\bigg( \frac{Y^{1/2}}{A} + \frac{X}{Y^{1/2} A} + \frac{1}{Y^{1/2} C} + \frac{1}{Y^{1/2} D} \bigg)^{J_3}
      \\ 
      \ll_{J_1,J_2,J_3} &\delta^{-J_1} \bigg( \frac{C}{A} \bigg)^{J_1} \delta^{-J_2}\bigg( \frac{C}{A} + CD \bigg)^{J_2}  \delta^{-J_3}\bigg( 1 + \frac{1}{AD} \bigg)^{J_3}
      \\
     \ll_{J_1,J_2,J_3} & (\delta_1  X)^{-J_1} (\delta_1 Y)^{-J_2} (\delta_1)^{-J_3},
\end{align*}
making use of $C =D$ and denoting $\delta_1 := \min \{\delta,\delta AD \}$. Note also $Y/X \leq 1$.

We now apply spectral expansion for $P_{f}$. By Proposition \ref{prop:specdecompL2} we have
\begin{align*}
    \nonumber P_{f}(g)=& \mathbf{1}_{\chi \,\mathrm{principal}}\frac{1}{|\Gamma \backslash G|} \langle P_{f} ,1 \rangle_{\Gamma \backslash G} +\sum_{V\in \mathcal{B}(q,\chi)} \sum_{\ell \equiv \kappa \pmod{2}}\langle P_{f} ,\varphi_V^{(\ell)} \rangle_{\Gamma \backslash G}  \varphi_V^{(\ell)}(g)\\
        &+\sum_{\mathfrak{c}\in \mathfrak{C}(q,\chi)} \sum_{\ell \equiv \kappa \pmod{2}} \frac{1}{4 \pi i} \int_{(0)} \langle P_{f}, E^{(\ell)}_{\mathfrak{c}}(\cdot,\nu) \rangle_{\Gamma \backslash G}  E^{(\ell)}_{\mathfrak{c}}(g,\nu) \d \nu.
\end{align*}
Therefore, we get
\begin{align}\label{eq:Pf0specdecomp}
\begin{split}
    S_{f,\alpha,\beta,\chi}(\Gamma,T)=& \sum_{\tau \in T}  \sum_{\gcd(h,q)=1} \alpha(\tau) \frac{\beta_h \sigma_1(h)}{\sqrt{h}} \mathbf{1}_{\chi \,\mathrm{principal}}\frac{1}{|\Gamma \backslash G|} \langle P_{f} ,1\rangle_{\Gamma \backslash G}+\mathcal{C}_{f,\alpha,\beta,\chi}(\Gamma,T)\\
    &+\mathcal{E}_{f,\alpha,\beta,\chi}(\Gamma,T),
    \end{split}
\end{align}
where by the action of the Hecke operators as in \eqref{eq:heckecusp} the cuspidal spectrum contribution is
\begin{align}\label{eq:specdecompcusp}  \mathcal{C}_{f,\alpha,\beta,\chi}(\Gamma,T)=\sum_{V\in \mathcal{B}(q,\chi)}\sum_{\gcd(h,q)=1}\beta_h\lambda_V(h) \sum_{\ell \equiv \kappa \pmod{2}}\langle P_{f} ,\varphi_V^{(\ell)} \rangle_{\Gamma \backslash G} \sum_{\tau\in T}\alpha(\tau) \varphi_V^{(\ell)}(\tau)
\end{align}
and similarly by \eqref{eq:heckeeis} the continuous spectrum contribution is
\begin{align}\label{eq:specdecompeis}
\mathcal{E}_{f,\alpha,\beta,\chi}(\Gamma,T)=\frac{1}{4\pi i}\sum_{\mathfrak{c}\in \mathfrak{C}(q,\chi)}\int_{(0)}  \sum_{\gcd(h,q)=1}\beta_h \lambda_{\mathfrak{c},\nu}(h) \sum_{\ell \equiv \kappa \pmod{2}} \langle P_{f}, E^{(\ell)}_{\mathfrak{c}}(\cdot,\nu) \rangle_{\Gamma \backslash G}  \sum_{\tau\in T}\alpha(\tau) E^{(\ell)}_{\mathfrak{c}}(\tau,\nu) \d \nu.
\end{align}

\subsubsection{Main term}
The projection onto the constant function in \eqref{eq:Pf0specdecomp} gives the main term in \eqref{eq:GBB,M}. Indeed, for $\chi$ principal we have by unfolding
\begin{align*}
    \langle P_{f} ,1 \rangle_{\Gamma \backslash G}&=\sum_{\gamma \in \Gamma} \int_{\Gamma \backslash G} F(\gamma g) \d g= \int_{ G} F(g) \d g.
\end{align*}
 Thus, using the fact that the Hecke operator runs over $\sum_{d|h}\sum_{b  \pmod{d}} 1=\sigma_1(h)$ many points, we get
\begin{align*}
   \sum_{\tau \in T}  \sum_{\gcd(h,q)=1} \alpha(\tau) \frac{\beta_h \sigma_1(h)}{\sqrt{h}} \mathbf{1}_{\chi \,\mathrm{principal}}\frac{1}{|\Gamma \backslash G|} \langle P_{f}(\tau) ,1\rangle&= M_{f,\alpha,\beta,\chi}(\Gamma,T).
\end{align*}
Consequently, the proof of Theorem \ref{thm:Gmainblackbox} is reduced to estimating $\mathcal{C}_{f,\alpha,\beta}(\Gamma,T)$ and $\mathcal{E}_{f,\alpha,\beta}(\Gamma,T)$.

\subsubsection{Cuspidal Spectrum}
We now estimate $\mathcal{C}_{f,\alpha,\beta,\chi}(\Gamma,T)$ as given by \eqref{eq:specdecompcusp}. Unfolding gives for the projections
\begin{align*}
    \langle P_{f} ,\varphi_V^{(\ell)} \rangle_{\Gamma \backslash G} &=\sum_{\gamma \in \Gamma} \int_{\Gamma \backslash G} \overline{\chi}(\gamma) F(\gamma g) \overline{\varphi_V^{(\ell)}(g)}  \d g =\int_{G} F( g) \overline{\varphi_V^{(\ell)}(g)}  \d g.
\end{align*}
We plug in the Fourier series expansion of Proposition \ref{prop:fourierexpansionbasis} to get
\begin{align*}
    &\int_{G} F( g) \overline{\varphi_V^{(\ell)}(g)}  \d g\\
    =&\left|\pi^{-2\nu_V}\frac{\Gamma(|\ell|+\nu_V+1/2)}{\Gamma(|\ell|-\nu_V+1/2)} \right|^{1/2}\sum_{n\neq 0} \frac{\overline{\varrho_V(n)}}{\sqrt{|n|}}  \int_{G} F( g)  \overline{\mathcal{A}^{\sgn(n)}\phi_{\ell}(\mathrm{a}[|n|]g,\nu_V)}\d g.
\end{align*}
We now apply Proposition \ref{prop:Jacquetinnerproduct}. More precisely, for $V$ in the regular and exceptional spectrum (that is $\Re(\nu_V)<1/2$), we observe that 
\begin{align*}
\left|\pi^{-2\nu_V}\frac{\Gamma(|\ell|+\nu_V+1/2)}{\Gamma(|\ell|-\nu_V+1/2)} \right|^{1/2}\ll 1
\end{align*}
and apply \eqref{eq:JIPMaas}. For the remaining part of the cuspidal spectrum, that is for $V$ in the discrete series, we apply \eqref{eq:JIPHolom}. Hence, denoting
\begin{align*}
    a_{V,\ell}(  n) :=  \overline{\varrho_V(\sigma n)} \sum_{\gcd(h,q)=1} \beta_h \lambda_V(h) \sum_{\tau \in T} \alpha(\tau) \varphi_V^{(\ell)}(\tau),
\end{align*}
we obtain (with $D_J(t,\nu_V,\ell)$ as in \eqref{eq:decaydef})
\begin{align*}
\mathcal{C}_{f,\alpha,\beta,\chi}(\Gamma,T)   \ll_{\eps,J}     \max_{t>0}  \frac{Z^\eps X}{Y^{1/2}} \sum_{\substack{V\in \mathcal{B}(q,\chi) \\  }} \sum_{\ell \equiv \kappa \pmod{2}} &(1+\mathbf{1}_{\Re(\nu_V) < 1/2}(  Yt  )^{- \Re(\nu_V)} ) D_J(t,\nu_V,\ell)
\\ & \times \bigg| \sum_{\tau \in T} \alpha(\tau) \varphi_V^{(\ell)}(\tau)\bigg| |R_V(t,\beta)|.
\end{align*}
 By using $Y=\frac{1}{CD}$, $X=1/N = \frac{AD+1}{CD}$ we get
\begin{align*}  \mathcal{C}_{f,\alpha,\beta,\chi}(\Gamma,T)   \ll_{\eps,J}& \max_{t >0} \frac{Z^\eps \delta^{1/2} \delta_1^{-1/2}(AD)^{1/2}}{N^{1/2}} \\
  &\times\sum_{\substack{V \in \B(q,\chi)  }} \sum_{\ell \equiv \kappa \pmod{2}}(1+\mathbf{1}_{\Re(\nu_V) < 1/2}( \tfrac{CD}{t}  )^{ \Re(\nu_V)})  D_J(t,\nu_V,\ell)  \\ & \times
\bigg| \sum_{\tau \in T} \alpha(\tau) \varphi_V^{(\ell)}(\tau)\bigg| |R_V(t,\beta)|. 
\end{align*}
We apply Cauchy-Schwarz in the form
\begin{align}\nonumber
    &\sum_{V,\ell}  \frac{D_J(t,\nu_V,\ell)}{N^{1/2}}\bigg| \sum_\tau \sum_h \sum_n \bigg| \\
    \leq& \bigg(  \sum_{V,\ell} D_J(t,\nu_V,\ell) (1+\log t)^2 \bigg| \label{eq:THEcauchyschwarz}\sum_\tau \bigg|^2 \bigg)^{1/2} \bigg(  \sum_{V,\ell} \frac{D_J(t,\nu_V,\ell)}{N (1+\log t)^2} \bigg|  \sum_h \sum_n \bigg|^2  \bigg)^{1/2}.
\end{align}
Here $\sum_\ell \frac{D_J(t,\nu_V,\ell)}{(1+\log t)^2}  \ll \delta^{-O(1)} Z^\eps E_J(t/N,\nu_V)$ so that the second factor is bounded by \begin{align*}
   \delta^{-O(1)} \mathcal{R}_{\Gamma,\beta}(X,N)^{1/2}.
\end{align*} 
Using Proposition \ref{prop:boundbykernel} the first factor is bounded by 
\begin{align*}
    &\ll_{\eps,J} Z^\eps \delta^{-O(1)}\bigg(\sum_{V,\ell} \min_{J_0\in \{0,1,\dots,J\}} \frac{1}{(1+|\delta \nu| + |\delta \ell|)^{2J_0-1}}  \bigg| \sum_{\tau \in T} \alpha(\tau) \varphi_V^{(\ell)}(\tau)\bigg| ^2\bigg)^{1/2} \\
   & \ll_\eps Z^\eps \delta^{-O(1)}\mathcal{K}_{f,\alpha}(\Gamma,T,1)^{1/2},
\end{align*}
which converges for $J \geq 3$. Thus, we obtain
\begin{align*}
\mathcal{C}_{f,\alpha,\beta,\chi}(\Gamma,T)  \ll_{\eps,J} Z^\eps \delta^{-O(1)}   (AD)^{1/2} \, \mathcal{K}_{f,\alpha}(\Gamma,T,1)^{1/2}\mathcal{R}_{\Gamma,\beta}(X,N)^{1/2}.
\end{align*}
\eqref{eq:Gmaintarget}.

\subsubsection{Eisenstein Series}
Similarly as above, we compute $\langle P_{f}, E^{(\ell)}_{\mathfrak{c}}(\cdot,\nu) \rangle_{\Gamma \backslash G} $ by unfolding. The contribution from $n \neq 0$ is bounded \emph{mutatis mutandis} as in the previous section. The contribution from $n= 0$ is
\begin{align*}
    \mathcal{E}_{f,\alpha,\beta,\chi,0}(\Gamma,T):=\frac{1}{4\pi i}\sum_{\mathfrak{c}\in \mathfrak{C}(q,\chi)}\int_{(0)}  \sum_{\ell \equiv \kappa \pmod{2}} \sum_{\gcd(h,q)=1}\beta_h \lambda_{\mathfrak{c},\nu}(h)  \langle F, E^{(\ell)}_{\mathfrak{c},0}(\cdot,\nu)\rangle_{ G}  \sum_{\tau\in T}\alpha(\tau) E^{(\ell)}_{\mathfrak{c}}(\tau,\nu) \d \nu
\end{align*}
where
\begin{align*}
   E^{(\ell)}_{\mathfrak{c},0}(g,\nu)  :=1_{\mathfrak{c}=\infty}\phi_\ell(g,\nu)+\varrho_{\mathfrak{c},\nu}^{(\ell)}(0)\phi_\ell(g,-\nu).
\end{align*}
Using $|\lambda_{\mathfrak{c},\nu}(h) | \ll_\eps h^\eps$ and Proposition \ref{prop:eisenstein0prop} we get
\begin{align*}
\mathcal{E}_{f,\alpha,\beta,\chi,0}(\Gamma,T) \ll_{\eps,J} \frac{Z^\eps \|\beta\|_1 \delta^{1/2} \delta_1^{-1/2}(AD)^{1/2}}{N^{1/2}} \sum_{\mathfrak{c}\in \mathfrak{C}(q,\chi)}\int_{(0)}  \sum_{\ell \equiv \kappa \pmod{2}} D_J(1,\nu,\ell)  \\ 
\times \left( \mathbf{1}_{\mathfrak{c} = \infty} + |\varrho_{\mathfrak{c},\nu}^{(\ell)}(0)|\right)
\bigg|\sum_{\tau\in T}\alpha(\tau) E^{(\ell)}_{\mathfrak{c}}(\tau,\nu)\bigg| |\d \nu|.
\end{align*}
We have
\begin{align*}
    |\varrho_{\mathfrak{c},\nu}^{(\ell)}(0)| \leq& \mathbf{1}_{\mathfrak{c} = \infty}+|\mathbf{1}_{\mathfrak{c} = \infty}+\varrho_{\mathfrak{c},\nu}^{(\ell)}(0)|   \\
    = &\mathbf{1}_{\mathfrak{c} = \infty} + \bigg| \int_0^1 E^{(\ell)}_{\mathfrak{c}}(\n[x],\nu) \d x\bigg| \\
    \leq & \mathbf{1}_{\mathfrak{c} = \infty} +  \int_0^1 
 \bigg|E^{(\ell)}_{\mathfrak{c}}(\n[x],\nu)\bigg| \d x.
\end{align*}
Thus, for some $x \in \R$ we have
\begin{align*}
\mathcal{E}_{f,\alpha,\beta,\chi,0}(\Gamma,T) \ll_{\eps,J} \frac{Z^\eps \|\beta\|_1 \delta^{1/2} \delta_1^{-1/2}(AD)^{1/2}}{N^{1/2}} \sum_{\mathfrak{c}\in \mathfrak{C}(q,\chi)}\int_{(0)}  \sum_{\ell \equiv \kappa \pmod{2}}D_J(1,\nu,\ell)  \\ 
\times  \left( \mathbf{1}_{\mathfrak{c} = \infty} +  \bigg|E^{(\ell)}_{\mathfrak{c}}(\n[x],\nu)\bigg| \right)   \bigg|\sum_{\tau\in T}\alpha(\tau) E^{(\ell)}_{\mathfrak{c}}(\tau,\nu)\bigg| |\d \nu|
\end{align*}
The contribution from $\mathbf{1}_{\mathfrak{c} = \infty}$ is bounded by Cauchy-Schwarz and Proposition \ref{prop:boundbykernel} by
\begin{align*}
    \ll_{\eps,J}  \frac{Z^\eps \|\beta\|_1 \delta^{-O(1)}(AD)^{1/2}}{N^{1/2}} \mathcal{K}_{f,\alpha}(\Gamma,T,1)^{1/2}.
\end{align*} 
The contribution from $\bigg|E^{(\ell)}_{\mathfrak{c}}(\n[x],\nu)\bigg|$ is bounded by applying Cauchy-Schwarz and Proposition \ref{prop:boundbykernel} twice 
\begin{align*}
  \ll_{\eps,J}  \frac{Z^\eps \|\beta\|_1 \delta^{-O(1)}(AD)^{1/2}}{N^{1/2}} \mathcal{K}(\Gamma,\{\n[x]\})^{1/2}\mathcal{K}_{f,\alpha}(\Gamma,T,1)^{1/2}  + Z^{-\eps J},
\end{align*}
where
\begin{align*}
\mathcal{K}(\Gamma,\{\n[x]\})  \ll  \delta^{-O(1)} \sum_{\gamma \in \Gamma} k_1(\n[-x] \gamma \n[x]) \ll \delta^{-O(1)}
\end{align*}
by using $k_1(\smqty(a& b \\ c& d)) \leq \mathbf{1}_{|a|+|b|  +|c|+|d| \leq 6}.$  \qed
\section{The spectral large sieve}\label{sec:spectrallarge}
In this section we bound the $\mathcal{R}_{\beta,\chi}(\Gamma,N,AD)$ part of Theorem \ref{thm:Gmainblackbox} with the help of the spectral large sieve. For the the unexceptional part of the spectrum, we have the following result \cite[Proposition 4.7]{Drappeau}.
\begin{proposition}\label{prop:speclasie}
    Let $K\ge 1$, $N\ge 1/2$, $\eps>0$ be real, and $a_n$ be a sequence of complex numbers supported on $n\sim N$. Each of the expressions
    \begin{align}
    &\sum_{\substack{ V \in \B(q,\chi) \\|\nu_V|\leq K}} \left|\sum_n a_n \varrho_V(n) \right|^2\\
       & \sum_{c\in \mathfrak{C}(q,\chi)} \int_{-iK}^{iK} \left|\sum_n a_n \varrho_{\mathfrak{c},\nu}(n) \right|^2 |\d \nu|
    \end{align}
    is bounded by
    \begin{align*}
        \ll_\eps (K^2+N^{1+\eps}\mathrm{cond}(\chi)^{1/2}/q) \|  a_n \|_2^2.
    \end{align*}
\end{proposition} 

\begin{remark}\label{rem:improvedchidependcy}
    It is possible that the term $\mathrm{cond}(\chi)^{1/2}$ can be improved by a more careful argument, characterizing the set of bad cases in the Weil type bound of Knightly and Li \cite[Theorem 9.3]{knightlyli}.
\end{remark}

Deshoulliers and Iwaniec \cite{deshoullieriwaniec} also proved a weighted version of the large sieve, useful for the exceptional spectrum, which was generalized by Drappeau \cite{Drappeau}. Their result was recently improved by Duker Lichtman\footnote{After the completion of this manuscript, Pascadi \cite{pascadi2024large} has shown an improved exceptional large sieve for certain coefficients. Incorporating it into our $\mathcal{R}$ estimates will improve the $\theta$ dependency at various places throughout this work.}, using the Kim-Sarnak bound on the largest possible exceptional eigenvalue (see \cite[Remark at the end of Section 4.2]{Drappeau} for a statement of the result with 
a multiplier $\chi$). We use the following notation to state the results. Recall that $\theta_q = \max\{ \Re(\nu_V) < 1/2 \}$  and $\theta=\max_{q} \theta_q$. By the work of Kim-Sarnak \cite{KimSarnak} we know that $\theta \leq 7/64.$

\begin{proposition}\label{prop:exclasie}
Let $a_n$ be a sequence of complex numbers supported on $n\sim N$.    We have
    \begin{align*}
        \sum_{\substack{V\in \B(q,\chi) \\ \Re(\nu_V) \in (0,1/2)}} X^{2 \nu_V} \left|\sum_{n} a_n   \varrho_V(n) \right|^2\ll \left(1+\bigg(\frac{NX}{q}\bigg)^{2\theta_q} \right)\left(1+ \mathrm{cond}(\chi)^{1/2}\bigg(\frac{N^{1+\eps}}{q}\bigg)^{1-2\theta_q} \right)\|  a_n\|_2^2
    \end{align*}
    and for any $N<N_1\leq 2N$
    \begin{align*}
       \sum_{\substack{q\sim Q \\ \mathrm{cond}(\chi) | q}}    \sum_{\substack{V\in \B(q,\chi) \\ \Re(\nu_V) \in (0,1/2)}} X^{2 \nu_V} \left|\sum_{N<n\leq N_1}  \varrho_V(n) \right|^2 \ll_\eps (QN)^\eps  \bigg(\frac{Q} {\mathrm{cond}(\chi)}  + N +  X^{2\theta} N^{2\theta} Q^{1-4\theta}\bigg) N
    \end{align*}
\end{proposition}
Here also it may be possible to improve on the dependency on $\mathrm{cond}(\chi)$.
\subsection{Main corollary}
In this section  we give bounds for the quantity $\mathcal{R}_{h,\chi}(q_1,q_2,N)$ in Theorem \ref{thm:Gmainblackbox}  using the spectral large sieve bounds. Recall that $\vartheta_q$ denotes the best exponent towards the Ramanujan-Peterson conjecture, that is,  smallest exponent so that $|\lambda_V(h)| \ll_\eps |h|^{\vartheta_q +\eps}$.  
\begin{corollary} \label{cor:Gmainblackbox}
Suppose that the assumptions in Theorem \ref{thm:Gmainblackbox} hold and let $ \mathcal{R}_{\beta,\chi}(q,N)$ be as defined there.
Assuming that $\beta_h$ is supported for $|h| \leq H$, we have
\begin{align*}
 & \mathcal{R}_{\beta,\chi}(q_1,q_2,N) \ll_\eps  \,  Z^\eps \bigg(\|\beta\|_1|H|^{2\vartheta_q}  \left(1+\bigg(\frac{CD}{q_2}\bigg)^{2\theta_q} \right)\left(1+ \mathrm{cond}(\chi)^{1/2}\bigg(\frac{N}{q}\bigg)^{1-2\theta_q}\right)  +\frac{\|\beta\|_1^2}{ N} \bigg), \\
   & \mathcal{R}_{h,\chi}(q_1,q_2,N) \ll_\eps  \,   Z^\eps \bigg(\|\beta\|_2^2 \left(1+\bigg(\frac{ H CD}{q_2}\bigg)^{2\theta_q} \right)\left(1+ \mathrm{cond}(\chi)^{1/2}\bigg(\frac{H N}{q}\bigg)^{1-2\theta_q}\right)   +\frac{\|\beta\|_1^2}{ N}\bigg), \\
 &\sum_{\substack{q=q_1q_2 \sim Q \\ q_1 \sim Q_1 \\ \mathrm{cond}(\chi) | q}}  \mathcal{R}_{h,\chi}(q_1,q_2,N) \ll_\eps  \,  Z^\eps \bigg(\|\beta\|_1|H|^{2\vartheta_q}   \bigg(\frac{Q} {\mathrm{cond}(\chi)}  + N + (CD Q_1 )^{2\theta} Q^{1-4\theta}\bigg) +     \frac{\|\beta\|_1^2 Q} { N\mathrm{cond}(\chi)} \bigg).
\end{align*}
\end{corollary}

\begin{proof}
First and third follow immediately from the point-wise bound  $|\lambda_V(h)| \leq H^{\vartheta_q}$ and Propositions \ref{prop:speclasie} and \ref{prop:exclasie}. The second bound also follows from these propositions, after we combine the coefficients by using the multiplicativity relation for $\gcd(h,q)=1$
    \begin{align*}
        \lambda_V(h)\varrho_V(n) = \sum_{d|(h,n)} \chi(d)  \varrho_V(nh/d^2),
    \end{align*}
    so that we can apply the spectral large sieve with the weights
\begin{align*}
    \alpha_m=\sum_{n\leq N, h\sim H}\beta_h \sum_{\substack{d|(h,n)\\m=nh/d^2}}\chi(d)=\sum_{d^2\leq NH/m}\chi(d) \sum_{h\sim H/d}\beta_{hd} \sum_{\substack{n\leq N/d\\ nh=m}}1.
\end{align*}
By a divisor bound
\begin{align*}
    \sum_{m}| \alpha_m|^2&\leq \sum_{m}\left|\sum_{h}\sum_{d}|\beta_{hd}|\sum_{\substack{n\leq N/d\\ nh=m}} 1 \right|^2 \\
  &= \sum_{m}  \sum_{d_1,d_2}\sum_{h_1,h_2 | m} |\beta_{d_1 h_1}| |\beta_{d_2 h_2}|  \sum_{\substack{n_j\leq N/d_j\\ n_jh_j=m}}1 \\
   &\leq \sum_{m}  \sum_{d_1,d_2}\sum_{h_1,h_2 | m} (|\beta_{d_1 h_1}|^2 +|\beta_{d_2 h_2}|^2)   \sum_{\substack{n_j\leq N/d_j\\ n_jh_j=m}} 1  \\
  & \ll_\eps Z^\eps  \sum_{d_1} \sum_{h_1} |\beta_{d_1 h_1}|^2 \sum_{d_2} \sum_{n \leq  N/d_2}  1  \ll_\eps Z^{\eps} N \|\beta\|_2^2.
 \end{align*}
\end{proof}

\section{Proof of Theorem \ref{thm:mainblackbox}}\label{sec:proofmainblack} Recall that we want to evaluate
\begin{align*}
      S_{f,\alpha}(\mathcal{M}) = \sum_{g=\smqty(a & b \\c& d) \in \mathcal{M}}  \alpha(g) f(a,c,d) ,
\end{align*}
where $\mathcal{M} \subseteq \SL(\R)$, $f$ satisfies the assumptions of Theorem \ref{thm:Gmainblackbox}, and
\begin{align*}
    L= C/D \quad \text{and} \quad N \frac{CD q_1}{AD+ 1} = \delta^{-O(1)} \frac{q_1 C}{A}.
\end{align*}
The claim then follows from Theorem \ref{thm:Gmainblackbox} ($h=1,\chi$ principal) and Corollary \ref{cor:Gmainblackbox} since by denoting $g= \tau_2^{-1} \gamma\tau_1$
\begin{align*}
    \mathcal{K}_{\alpha,\chi}(\Gamma,T,L)&=   \sum_{\tau_1, \tau_2 \in T} \alpha(\tau_1)  \overline{\alpha(\tau_2)}\sum_{\gamma \in \Gamma} k_L(\tau_2^{-1} \gamma\tau_1  ) \\
    & = \sum_{g \in \mathcal{M}^{-1}\mathcal{M}} k_L(\tau_2^{-1} \gamma\tau_1  )  \sum_{\substack{\tau_1, \tau_2 \in T \\ \tau_2 g \tau_1^{-1} \in \Gamma}} \alpha(\tau_1)  \overline{\alpha(\tau_2)}  \\
   & = \sum_{g \in \mathcal{M}^{-1}\mathcal{M}} k_L(\tau_2^{-1} \gamma\tau_1  )  \sum_{ \substack{\tau\in \Gamma \backslash \mathcal{M} 
 \\ \tau g \in \mathcal{M}}} \alpha(\tau) \overline{\alpha(\tau g)}  \\
    & \leq \sum_{\substack{g =\smqty(a & b \\c & d) \in \mathcal{M}^{-1} \mathcal{M} \\  |a|+|b| L +|c|/L+|d| \leq 6}}\Bigl|\sum_{ \substack{\tau\in \Gamma \backslash \mathcal{M} 
 \\ \tau g \in \mathcal{M}}} \alpha(\tau) \overline{\alpha(\tau g)} \Bigr|.
\end{align*}
\qed

\section{Determinant equations with an automorphic twist}

In this section we state and prove a more flexible technical version of Theorem \ref{thm:mainblackbox} and provide some helping information that make it easier to apply. To be more precise, the main theorem of this section allows us to count solutions to a determinant equation $ad-bc=hk$ twisted by a function
\begin{align*}
  \alpha(\smqty(a & b \\ c& d)):= \alpha_0(\smqty(a & b \\ c& d)) \chi_1(a) \psi_1(b) \chi_2(c) \psi_2 (d)
\end{align*}
where  $\chi_1,\psi_1$ (resp. $\chi_2,\psi_2$) are Dirichlet-characters to the modulus $q_1$ (resp. $q_2$) and for any matrix $\smqty(a & b \\ c & d) \in \MM_2(\Z)$ with $q_1| b,q_2 | c$ we have $\alpha_0(\smqty(a & b \\ c& d) g) = \alpha_0(g)$. Note that then for any such $\smqty(a & b \\ c & d)$  we have for $\chi= \chi_1 \psi_1 \overline{\chi_2} \overline{\psi_2}$
\begin{align*}
      \alpha(\smqty(a & b \\ c& d) g ) = \chi(a) \chi_2\psi_2(ad- bc) \alpha(g) =  \chi(a) \chi_2\psi_2(ad) \alpha(g).
\end{align*}
Here we require that $(h,q_1 q_2)=1$. This motivates the following generalisation of Definition \ref{def:groupaction}.
\begin{definition}[Automorphic function with character and determinant twist]\label{def:autfunctionchi}
Let $q_1,q_2 \in \Z_{>0}$ and let $\chi$ be a Dirichlet character to a modulus dividing $q_1q_2$. For multiplicative coefficients $\xi_h$ we define
\begin{align*} 
    \mathcal{A}(q_1,q_2,\chi,\xi):= \{\alpha:\MM_2(\Z) \to \C: \, l_g \alpha = \chi(a) \xi_{\det g} \alpha \,\, \text{for all} \,\, g=\smqty(a & b \\ c& d)& \in \MM_{2}(\Z), \, q_1 | b, \, q_2 |c,  \\
    &\gcd(\det g,q)=1 \}.
\end{align*}
\end{definition}
\begin{theorem} \emph{(Determinant equation twisted by an automorphic function).} \label{thm:twisteddetwbound}
    Let $q_1,q_2 \in \Z_{> 0}$. For non-zero integers $h,k$ denote
\begin{align*}
    \MM_{2,h,k}(\Z) := \bigg\{ \mqty(a &b \\ c &d) \in \MM_{2,hk}(\Z): \,\gcd (a,c,k)=\gcd(b,d,k)=1 \bigg\},
\end{align*}
Denote
 \begin{align*}
    q=q_1q_2, \quad \Gamma = \Gamma_2(q_1,q_2), \quad T := \Gamma \backslash \SL(\Z), \quad T_{1,k}:= \SL(\Z) \backslash \MM_{2,1,k}(\Z).
 \end{align*}   
Let $\chi$ be a Dirichlet character to a modulus dividing $q_1q_2$ and let $\xi_h$ be multiplicative complex coefficients. Let
\begin{align*}
    \alpha \in   \mathcal{A}(q_1,q_2,\chi,\xi).
\end{align*}
 Let $A,C,D,\delta,\eta >0$ with $AD > \delta$  and denote $Z := \max\{A^{\pm 1},C^{\pm 1},D^{\pm 1}, \delta^{-1}\}$. Let $H,K \geq  1$  and assume that $HK \leq (AD)^{1+\eta}$. Let 
 \begin{align*}
   f\in C^{7}_\delta \bigg( \frac{A}{\sqrt{HK}},\frac{C}{\sqrt{HK}},\frac{D}{\sqrt{HK}}\bigg).
 \end{align*} 
Let $\beta_h$ be supported on $|h| \in [H,2H]$ and $\gamma_k$ supported on $|k| \in [K,2K]$. Denote
 \begin{align}\label{eq:wdef}
       w(\sigma,\sigma_1,\sigma_2) =  \sum_{\tau \in T} \alpha(\tau \sigma \sigma_1)  \overline{\alpha(\tau \sigma_2)} 
   \end{align}
and assume that for some $\mathcal{K}_+ >0$ 
\begin{align} \label{eq:Kassumptiongeneral}
    \frac{1}{K} \sum_{k_1,k_2} |\gamma_{k_1} \gamma_{k_2}| \sum_{\substack{g=\smqty(a&b \\ c & d)  \in (k_1/k_2)^{1/2} \SL(\R) \\ |a| + |b|C/D + |c|D/C + |d| \leq 10}} \bigg| \sum_{\substack{\sigma_j \in T_{1,k_j} \\ \sigma_2 g \sigma_1^{-1} =\sigma\in \SL(\Z) }}  w(\sigma,\sigma_1,\sigma_2) \bigg| \ll Z^{O(\eta)} \mathcal{K}_+.
\end{align}
Then
\begin{align*}
   &\sum_{\substack{h,k \\ \gcd(h,kq)=1}} \beta_h \gamma_k \sum_{\substack{\smqty(a& b \\ c& d)\in \MM_{2,h,k}(\Z)}}  \alpha(\smqty(a & b \\ c & d)) f\bigg(\frac{a}{ \sqrt{|hk|}},\frac{c}{ \sqrt{|hk|}},\frac{d}{ \sqrt{|hk|}}\bigg) \\
   =&   \frac{\mathbf{1}_{\chi \, \mathrm{principal}}}{ \zeta(2)
 q\, \prod_{p|q}(1+p^{-1}) }\sum_{\substack{h,k \\ \gcd(h,kq)=1}} \beta_h \xi_h \sigma_1(|h|) \gamma_k \sum_{\tau \in \Gamma \backslash \MM_{2,1,k}(\Z)} \alpha(\tau)  \int_{\R^3} f(a,c,d) \frac{\d  a \d c \d d }{c}  \\
 &\hspace{80pt}+ O \left( Z^{O(\eta)} \delta^{-O(1)} (AD)^{1/2}  \|\beta \xi \|_2  \,  \mathcal{K}_+^{1/2} \left(\mathcal{R}_0 + \min_{j\in \{1,2\}}\mathcal{R}_j \right) \right),
\end{align*}
where 
\begin{align*}
    \mathcal{R}_0 &:= \frac{ \|\beta \xi \|_1  A^{1/2} }{\|\beta \xi \|_2  q_1^{1/2} C^{1/2}},  \\
    \mathcal{R}_1 &:=\frac{\|\beta  \xi \|_1}{\|\beta \xi \|_2} H^{\vartheta_q}  \left(1+\bigg(\frac{CD}{HK q_2 }\bigg)^{\theta_q} \right)\left(1+ \mathrm{cond}(\chi)^{1/4}\bigg(\frac{C}{A q_2}\bigg)^{1/2-\theta_q}\right) ,    \\
    \mathcal{R}_2 &:=\left(1+\bigg(\frac{ CD}{Kq_2}\bigg)^{\theta_q} \right)\left(1+ \mathrm{cond}(\chi)^{1/4}\bigg(\frac{H C}{A q_2 }\bigg)^{1/2-\theta_q}\right).
\end{align*}
\end{theorem}
Note that the smooth weight is supported on $(|a|,|c|,|d|)$ of size $\asymp (A,C,D)$. In our applications we only consider $\mathcal{A}(q_1,q_2,\chi,\xi)$ with a principal character $\chi$ and with $k$ fixed, but the generality will be useful for future work.

We have the following observation that will be useful in the proof of Theorem \ref{thm:divisorperiodic}.
\begin{remark}\label{rem:k=1blackbox}
In the case that $\gamma_k$ is supported on $k=1$ we can denote $w(\sigma) = w(\sigma,I,I)$ and the assumption 
\eqref{eq:Kassumptiongeneral} may be simplified to \begin{align}\label{eq:thmtwisteddetwbound1}
\begin{split}
     \sum_{0\leq |c| \leq 6 C/D} |w(\smqty(\pm 1 &  \\ c & \pm 1))| &\ll Z^{O(\eta)} \mathcal{K}_+ \\
   \sum_{0\leq|b| \leq 6 D/C} |w(\smqty(\pm 1 & b \\  & \pm 1))| &\ll Z^{O(\eta)} \mathcal{K}_+ .
   \end{split}
\end{align}
\end{remark}

To make the application of Theorem \ref{thm:twisteddetwbound} and in particular the calculation of $w(\sigma,\sigma_1,\sigma_2)$ easier, we now give a parametrization for $T = \Gamma \backslash \SL(\Z)$. Define the projective line over $\Z/p^k \Z$ 
\begin{align*}
    \mathbb{P}_{p^k}^{1} := \{(x,y) \in (\Z/p^k \Z)^2: x \, \text{or} \, y \in  (\Z/p^k \Z)^{\times}\} / \sim,
\end{align*}
where we define an equivalence relation by $(x_1,y_1) \sim (x_2,y_2)$ if there exists $\lambda \in (\Z/p^k \Z)^\times$ such that $(x_1,y_1) =  (\lambda x_2,\lambda y_2)$. This can be identified with the set of size $p^k+p^{k-1}$
\begin{align*}
    \{(x,1): x \in \Z/p^k \Z \} \cup   \{(1,y): y \in \Z/p^k \Z, \,\, p |y \}.
\end{align*}
For $q \in \Z_{>0}$ we define
\begin{align*}
     \mathbb{P}_{q}^{1} := \prod_{p^k||q}  \mathbb{P}_{p^k}^{1}.
\end{align*}
By the Chinese remainder theorem we may identify this with the set
\begin{align*}
      \{(x,y) \in (\Z/q \Z)^2: \gcd(x,y,q)=1\} / \sim,
\end{align*}
where the equivalence relation is as before with $ \lambda \in (\Z / q \Z)^\times$.
Note that the normalization in $q$ in the expected main term  in Theorem \ref{thm:twisteddetwbound} satisfies
\begin{align*}
    q \prod_{p | q} (1+p^{-1}) = | \mathbb{P}_{q}^{1}|.
\end{align*}

Let $\Z^2(q) := \{(a,b) \in \Z: \gcd (a,b,q)=1\}$ and  define projections
for $q= p_1^{k_1} \cdots p_m^{k_m}$
\begin{align*}
    \pi_{p^{k}}: &  \Z^2(q) \to  \mathbb{P}_{p^k}^{1}, \quad (a,b) \mapsto 
        (a,b) \pmod{p^k} \\
     \pi_{q}:&   \Z^2(q) \to \mathbb{P}_{q}^{1}, \quad \pi_q := \pi_{p^{k_1}} \times \cdots \times  \pi_{p^{k_m}}.
\end{align*}
By Chinese remainder theorem we may continue to denote elements of  $\mathbb{P}_{q}^{1}$ as pairs $(a,b) \in \mathbb{P}_{q}^{1}$, $a,b \in \Z/q \Z$.
Given $q_1,q_2$ we define then the projection 
\begin{align*}
    \pi_{q_1,q_2}: \SL(\Z) \to  \mathbb{P}_{q_1}^{1} \times  \mathbb{P}_{q_2}^{1}:  \mqty(a & b \\ c &d) \mapsto ( \pi_{q_1}(a,b) ,\pi_{q_2}(c,d)).
\end{align*}
This map is surjective iff $q_0=\gcd(q_1,q_2)=1$. The image is characterized by the condition $\gcd(ad-bc,q_0)=1$.  The map $   \pi_{q_1,q_2}$ is invariant under the action of $\Gamma_2(q_1,q_2)$ and we get a bijection 
\begin{align} \label{eq:projectivereps}
\varpi_{q_1,q_2}:  T=  \Gamma_2(q_1,q_2) \backslash \SL(\Z) \to    \mathrm{Im}(\pi_{q_1,q_2}).
\end{align}
This is well-defined since $\pi_{q_1,q_2}$ is invariant under the action of $\Gamma_2(q_1,q_2)$ from the left. Hence, we have the following lemma for $\xi \equiv 1$, which we require for the proof of Theorems \ref{cor:Lfunction} and \ref{cor:Lfunctiononechar}.
\begin{lemma} \label{lem:wsimplification}
     Let $q,q_1,q_2$, $\alpha_0$, $\chi_j,\psi_j$, and $\alpha(\smqty(a & b \\ c& d))= \alpha_0(\smqty(a & b \\ c& d)) \chi_1(a) \psi_1(b) \chi_2(c) \psi_2 (d)$ with $\alpha_0$ being $\Gamma_2(q_1,q_2)$ invariant and assume that $\alpha \in \mathcal{A}(q_1,q_1,\chi,1)$. Denote $q_0 = \gcd(q_1,q_2)$.
     Then for $w$ given by \eqref{eq:wdef} we have under the identification \eqref{eq:projectivereps} 
   \begin{align*}
          w(\sigma,I,I) = \sum_{(a_2,b_2) \in \P^{1}_{q_1} } \sum_{\substack{(c_2,d_2) \in \P^{1}_{q_2} \\ \gcd(a_2d_2-b_2c_2,q_0)=1}  } \alpha_0(\tau_2 \sigma ) \overline{\alpha_0(\tau_2 )}   
      & \, \chi_1  (a a_2 + c b_2) \psi_1(ba_2+db_2)\overline{\chi_1 } (a_2)  \overline{\psi_1 }(b_2)  \\
   \times& \,\chi_2 (a c_2+c d_2) \psi_2 (b c_2 + d d_2) \overline{ \chi_2} (c_2) \overline{  \psi_2} (d_2).
   \end{align*}
Furthermore, for any $\sigma \in \SL(\Z)$ we have $|w(\sigma,I,I)| \leq \sum_{\tau \in T} |\alpha_0(\tau)|^2$.
\end{lemma}

\subsection{Proof of Theorem \ref{thm:twisteddetwbound}} \label{section:twist}
By symmetry we may restrict to the part $h,k \geq 1$ (otherwise we can multiply the determinant equation throughout by $-1$ and let $(c,d) \mapsto (-c,-d)$). Recall that
\begin{align*}
    \MM_{2,h,k}(\Z) = \bigg\{ \mqty(a &b \\ c &d) \in \MM_{2,hk}(\Z): \, \gcd(a,c,k)=\gcd(b,d,k)=1 \bigg\}.
\end{align*}
Note that the action from left by $\SL(\Z)$ preserves the conditions $\gcd(a,c,k)=\gcd(b,d,k)=1$ so that we may set 
\begin{align*}
  T_{h,k} := \SL(\Z) \backslash   \MM_{2,h,k}(\Z).
\end{align*}
Recall that in the statement of Theorem \ref{thm:twisteddetwbound}  we denote $T = \Gamma \backslash \SL(\Z)$, $\Gamma=\Gamma_2(q_1,q_2)$, $q=q_1 q_2$. By definition
\begin{align*}
    \MM_{2,h ,k}(\Z) =  \SL(\Z)  T_{h,k} = \Gamma \, T \, T_{h,k}.
\end{align*}
Denote further
\begin{align*}
    T_{h}^{(q_1)} = \bigg\{ \mqty(a & bq_1 \\ & d): a,d \in \Z_{>0}, \quad ad = h, \quad b \pmod{d} \bigg\}.
\end{align*}
By $(h,kq) = 1$ we have
\begin{align} \label{eq:swaoheckeandT}
      \MM_{2,h, k}(\Z) =   \Gamma \, T \,  T_{h}^{(q_1)} T_{1,k} =    \Gamma \, T_{h}^{(q_1)}\, T \,  T_{1,k}.
\end{align}
Indeed, to show that $\Gamma \, T \,  T_{h}^{(q_1)} = \Gamma \,T_{h}^{(q_1)} \, T  $ we note that for $\sigma_j \in T_{h}^{(q_1)}, \tau_j \in T$
\begin{align*}
    \Gamma \sigma_1 \tau_1 =   \Gamma \sigma_2 \tau_2 \Rightarrow \sigma_1 \tau_1 \tau_2^{-1} \sigma_2^{-1} \in \Gamma \Rightarrow \tau_1 = \tau_2  \Rightarrow  \sigma_1 \sigma_2^{-1} \in SL_2(\Z) \Rightarrow \sigma_1 = \sigma_2.
\end{align*}
Here in the second implication we rely on the fact that  $q_1,q_2$ respectively divide the top right, bottom left entries of $\sigma_j$. Therefore, the orbits in $T_{h}^{(q_1)} \, T $ are distinct and has the same cardinality as $T \,  T_{h}^{(q_1)}$, so the two sets of orbits must be equal.

Since $\alpha \in  \mathcal{A}(q_1,q_2,\chi,\xi)$, for $\smqty(a & bq_1 \\ & d) \in  T_{h}^{(q_1)}$ we have  for any $ \sigma \in \MM_{2}(\Z)$ 
\begin{align*}
   \alpha ( \smqty(a & bq_1 \\ & d) \sigma )  = \xi_h \chi(a) \alpha(\sigma)
\end{align*}
and for any $\gamma \in \Gamma$ we have $\alpha(\gamma \sigma) = \overline{\chi}(\gamma)\alpha(\sigma)$. Therefore,  denoting
\begin{align*}
 P_{f}(\tau) :=  \sum_{\gamma \in \Gamma} \overline{\chi}(\gamma) f(\gamma \tau),
\end{align*}
we have
\begin{align*}
     \sum_{\substack{\smqty(a& b \\ c& d)\in \MM_{2,h,k}}}  \alpha(\smqty(a & b \\ c & d)) f\bigg(\frac{a}{ \sqrt{|hk|}},\frac{c}{ \sqrt{|hk|}},\frac{d}{ \sqrt{|hk|}}\bigg)&= \sum_{\tau \in \frac{1}{ \sqrt{hk}} \Gamma \backslash \MM_{2,h,k}(\Z) }\alpha(\tau) P_f (\tau)  \\
     &= \xi_h \sqrt{h}\sum_{\tau \in T} \sum_{\sigma \in \frac{1}{ \sqrt{k}}T_{1,k}} \alpha(\tau \sigma) \mathcal{T}_h P(\tau \sigma),
\end{align*}
where the Hecke operator $\mathcal{T}_h$ (with character $\chi$) is defined in \eqref{eq:gamma2heckedef}. By expanding $\gcd(h,k) =1$ with the M\"obius function we get
\begin{align*}
     &\sum_{\substack{ h,k\\\gcd(h,kq)=1}} \beta_h \gamma_k \sum_{ad-bc = hk}  \alpha (\smqty(a & b \\ c & d)) f\bigg(\frac{a}{ \sqrt{|hk|}},\frac{c}{ \sqrt{|hk|}},\frac{d}{ \sqrt{|hk|}}\bigg) \\
     &= \sum_{\gcd(m,q)=1} \mu(m)   \sum_{\substack{ h,k\\\gcd(h,q)=1 \\ m | h \\ m |k}} \beta_{h} \xi_h \gamma_{k} \sqrt{h}\sum_{\tau \in T} \sum_{\sigma \in k^{-1/2}T_{1,k}} \alpha(\tau \sigma) \mathcal{T}_h P(\tau \sigma).
\end{align*}
Thus, denoting $\beta^{(m)}_h = \beta_h \xi_h \sqrt{h/H} \mathbf{1}_{m|h}, $ and $ \gamma^{(m)}_k= \gamma_k \mathbf{1}_{m|k}$, in the notation of Theorem \ref{thm:Gmainblackbox} we have
\begin{align*}
    \sum_{\substack{ h,k\\\gcd(h,kq)=1}} \beta_h \gamma_k \sum_{ad-bc = hk}  \alpha (\smqty(a & b \\ c & d)) &f\bigg(\frac{a}{ \sqrt{|hk|}},\frac{c}{ \sqrt{|hk|}},\frac{d}{ \sqrt{|hk|}}\bigg)\\
    &=\sum_{\gcd(m,q)=1} \mu(m) H^{1/2} S_{f,\alpha  \star  \gamma^{(m)} ,\beta^{(m)},\chi}(\Gamma,T_0),
\end{align*}
where
\begin{align*}
    T_0 = \bigcup_{k} k^{-1/2} T T_{1,k}
\end{align*}
and
\begin{align*}
    (\alpha \star \gamma^{(m)}) ( \tau_0 ) =  \sum_k \gamma^{(m)}_k \sum_{\substack{\tau_0 = \tau \sigma k^{-1/2}  \\ \tau \in T \\ \sigma \in T_{1,k} } }  \alpha  ( \tau \sigma). 
\end{align*}
We apply Theorem \ref{thm:Gmainblackbox} with
\begin{align*}
    f \in C^7_\delta \bigg( \frac{A}{\sqrt{HK}},\frac{C}{\sqrt{HK}},\frac{D}{\sqrt{HK}}\bigg).
\end{align*}
After substituting in the integral and going back from $\beta^{(m)},\gamma^{(m)}$ to $\beta,\gamma$, the main term in Theorem \ref{thm:Gmainblackbox} matches the one in Theorem \ref{thm:twisteddetwbound}. We get an error of size (after applying Cauchy-Schwarz on $m$)
\begin{align*}
 \ll_\eps  Z^{O(\eta)}  (AD)^{1/2} K^{-1/2}\bigg(\sum_m \mathcal{K}_{\alpha \star \gamma^{(m)} ,\chi}(\Gamma,T_0,L)\bigg)^{1/2}\bigg(\sum_m \mathcal{R}_{\beta^{(m)},\chi}(q_1,q_2,N)\bigg)^{1/2},
\end{align*}
It remains to estimate $\mathcal{K}$ and $\mathcal{R}$.

We first consider $\mathcal{K}$. We have 
\begin{align*}
   \frac{1}{K} \sum_m \mathcal{K}_{\alpha \star \gamma^{(m)},\chi}(\Gamma,T,L) = &  \frac{1}{K} \sum_m \sum_{\tau_1, \tau_2 \in T_0} (\alpha \star \gamma^{(m)})(\tau_1)  \overline{(\alpha \star \gamma^{(m)})(\tau_2)}\sum_{\gamma \in \Gamma} \overline{\chi}(\gamma) k_L(\tau_2^{-1} \gamma\tau_1  ) \\
   = &\frac{1}{K} \sum_{k_1,k_2} \sum_{m|\gcd(k_1,k_2)}\gamma_{k_1} \overline{\gamma_{k_2}} \sum_{\sigma_j \in T_{1,k_j}} \sum_{\tau_1, \tau_2 \in T} \alpha (\tau_1 \sigma_1)  \overline{\alpha(\tau_2 \sigma_2 )} \\
  &\times\sum_{\gamma \in \Gamma} \overline{\chi}(\gamma) k_L( k_2^{1/2} \sigma_{2}^{-1}\tau_2^{-1} \gamma\tau_1 \sigma_1  k_1^{-1/2})
\end{align*}
for some function $k_L:G \to [0,1]$ with 
\begin{align*}
k_L(\smqty(a & b \\ c& d)) \leq \mathbf{1}_{|a|+|b| L +|c|/L+|d| \leq 6}.
\end{align*}
By denoting $\sigma = \tau_2^{-1} \gamma \tau_1$ and $g= \sigma_{2}^{-1}\tau_2^{-1} \gamma\tau_1 \sigma_1$ we have by the triangle inequality and a divisor bound to handle $\sum_m$
\begin{align*}
     \frac{1}{K} \sum_m \mathcal{K}_{\alpha \star \gamma^{(m)},\chi}(\Gamma,T,L) \ll_\eps & \frac{Z^\eps}{K} \sum_{k_1,k_2} |\gamma_{k_1} \gamma_{k_2}| \sum_{g \in (k_1/k_2)^{1/2}\SL(\R)}   |k_L( (k_2/k_1)^{1/2} g )| \\
     &\bigg| \sum_{\substack{\sigma_j \in T_{1,k_j} \\ \sigma_2 g \sigma_1^{-1} = \sigma \in \SL(\Z) }}  \sum_{\substack{\tau_j \in T \\ \tau_2 \sigma  \tau_1^{-1} \in \Gamma }}  \overline{\chi}(\tau_2 \sigma \tau_1^{-1}) \alpha(\tau_1\sigma_1)\overline{\alpha(\tau_2\sigma_2)}\bigg| \\
     \ll_\eps & \frac{Z^\eps}{K} \sum_{k_1,k_2} |\gamma_{k_1} \gamma_{k_2}| \sum_{\substack{g = \smqty(a &b \\c& d) \in (k_1/k_2)^{1/2} \SL(\R) \\ |a| + |b|L + |c|/L +|d| \leq 10}}   \\
     &\bigg| \sum_{\substack{\sigma_j \in T_{1,k_j} \\ \sigma_2 g \sigma_1^{-1} = \sigma \in \SL(\Z) }}  \sum_{\substack{\tau_j \in T \\ \tau_2 \sigma  \tau_1^{-1} \in \Gamma }}  \overline{\chi}(\tau_2 \sigma \tau_1^{-1}) \alpha(\tau_1\sigma_1)\overline{\alpha(\tau_2\sigma_2)}\bigg| .
\end{align*}
The inner sum is $w(\sigma,\sigma_1,\sigma_2)$. Indeed, in the range of summation $  \tau_1 = \gamma \tau_2 \sigma$ for some  $\gamma \in \Gamma$ since $\Gamma \tau_1 = \Gamma \tau_2 \sigma$. Thus, by using $\alpha \in \mathcal{A}(q_1,q_2,\chi,\xi)$ we get
 \begin{align*}
  \overline{\chi}(\tau_2 \sigma \tau_1^{-1}) \alpha(\tau_1  \sigma_1)  \overline{\alpha(\tau_2  \sigma_2)  } =  \chi(\gamma) \alpha( \gamma \tau_2 \sigma  \sigma_1)  \overline{\alpha(\tau_2  \sigma_2)  }  =   \alpha(  \tau_2 \sigma  \sigma_1)  \overline{\alpha(\tau_2  \sigma_2)  }.
 \end{align*}
To complete the proof of Theorem \ref{thm:twisteddetwbound} we now estimate $\mathcal{R}$ with the help of Corollary \ref{cor:Gmainblackbox}. Applying it gives us that  (ignoring  a factor of $ Z^{O(\eta)}$)
\begin{align*}
\sum_m \mathcal{R}_{\beta^{(m)},\chi}(q_1,q_2,N) \ll & \,  \|\beta \xi \|_1|H|^{2\vartheta_q}  \left(1+\bigg(\frac{CD}{ HK q_2 }\bigg)^{2\theta_q} \right)\left(1+ \mathrm{cond}(\chi)^{1/2}\bigg(\frac{N}{q}\bigg)^{1-2\theta_q}\right) + \frac{\|\beta \xi \|_1^2}{ N} ,  \\
   \sum_m \mathcal{R}_{\beta^{(m)},\chi}(q_1,q_2,N) \ll & \,   \|\beta \xi \|_2^2 \left(1+\bigg(\frac{ CD}{ K q_2}\bigg)^{2\theta_q} \right)\left(1+ \mathrm{cond}(\chi)^{1/2}\bigg(\frac{H N}{q}\bigg)^{1-2\theta_q}\right)   +\frac{\|\beta \xi \|_1^2}{ N}.
\end{align*}
Plugging in  $ q_1 C/A \ll N \ll Z^{O(\eta)} q_1 C/A$, we obtain that up to  a factor of $ Z^{O(\eta)}$
\begin{align*}
\sum_m \mathcal{R}_{\beta^{(m)},\chi}(q_1,q_2,N) \ll & \,  \|\beta  \xi \|_1^2|H|^{2\vartheta_q}  \left(1+\bigg(\frac{CD}{ HK q_2 }\bigg)^{2\theta_q} \right)\left(1+ \mathrm{cond}(\chi)^{1/2}\bigg(\frac{C}{A q_2}\bigg)^{1-2\theta_q}\right)  +\frac{ \|\beta  \xi \|_1^2 A }{q_1 C}, \\
   \sum_m \mathcal{R}_{\beta^{(m)},\chi}(q_1,q_2,N)  \ll & \,   \|\beta  \xi  \|_2^2 \left(1+\bigg(\frac{ CD}{ K q_2}\bigg)^{2\theta_q} \right)\left(1+ \mathrm{cond}(\chi)^{1/2}\bigg(\frac{H C}{A q_2 }\bigg)^{1-2\theta_q}\right)   +\frac{ \|\beta  \xi \|_1^2 A }{q_1 C},
\end{align*}
as required.
\qed

\section{Proof of Theorem \ref{thm:divisorperiodic}} \label{sec:proofapplperiodicdivisor}
We first prove the following proposition, which is Theorem \ref{thm:divisorperiodic} with a weaker error term.
\begin{proposition} \label{prop:divisorperiodicweak}
    Let
    \begin{align} \nonumber
U_h(r_1,r_2;q) &:=\prod_{p|q} \bigg( \mathbf{1}_{r_1\equiv 0 \pmod{p}} + \mathbf{1}_{r_2\equiv 0 \pmod{p}} + \mathbf{1}_{r_2\equiv -h \pmod{p}} + \mathbf{1}_{r_1 \equiv r_2 \pmod{p}} + p^{-1} \bigg) \\
   \label{eq:Ntdef} \mathcal{N}_h(t) &:= \sum_{r_1,r_2 \pmod{q}} |t(r_1) t(r_2)| U_h(r_1,r_2;q).
\end{align}
Then the asymptotic in Theorem \ref{thm:divisorperiodic} holds with the error term
\begin{align*}
    O\bigg( X^{1/2+O(\eta)} q^{1/2}\mathcal{N}_h(t)^{1/2}   (|h|^{\theta}+(X/q)^{\theta})  \bigg).
\end{align*}
\end{proposition}
For the proof, we begin by expressing $t(n)$ as a sum of functions which are \emph{balanced} in a certain sense. Defining
\begin{align} \label{def:urp}
    u(r;p) := \mathbf{1}_{r \equiv 0 \pmod{p}} - p^{-1}, \quad u(r;q_0) := \prod_{p|q_0}u(r;p),
    \end{align}
    we have for square-free $q$
    \begin{align*}
        \mathbf{1}_{n \equiv r \pmod{p}} = \prod_{p|q} \mathbf{1}_{n \equiv r \pmod{p}} =   \prod_{p|q} (u(n-r;p) + p^{-1}) = \sum_{q_0 | q}   \frac{1}{q/q_0}   u(n-r;q_0).
    \end{align*}
Therefore,
\begin{align} \label{eq:tsplitbalanced}
\begin{split}
      t(n) &= \sum_{r \pmod{q}} t(r)  \mathbf{1}_{n \equiv r \pmod{p}}  =  \sum_{q_0 | q}  \frac{1}{q/q_0}  \sum_{r \pmod{q}} t(r) u(n-r;q_0) \\
   & = \sum_{q_0 | q} t^{\flat}(n;q_0),
   \\
    t^{\flat}(n;q_0) &:= \sum_{r_0 \pmod{q_0}} t(r_0;q_0) u(n-r_0;q_0), \\
    t(r_0;q_0)  &:= \frac{1}{q/q_0}\sum_{\substack{r \pmod{q} \\ r \equiv r_0 \pmod{q_0}}} t(r).  
\end{split}
\end{align}
Denote
\begin{align*}
    \alpha(\smqty(a & b \\c & d);q_0) := t^{\flat}(ad;q_0).
\end{align*}
The parts where some $p|q_0$ divides one of $a,b,c,d$ also need to be separated. To this end, we write for $ad-bc=-h$ with $\gcd(h,q)=1$
\begin{align}
  \nonumber   \alpha(\smqty(a & b \\c & d);q_0) &=  \alpha(\smqty(a & b \\c & d);q_0) \prod_{p|q_0} \bigg( \mathbf{1}_{p|a}+\mathbf{1}_{p|b} +  \mathbf{1}_{\substack{p|c  }}   + \mathbf{1}_{\substack{p |d  }} -\mathbf{1}_{p|\gcd(a,d)} -\mathbf{1}_{p|\gcd(b,c)}+ \mathbf{1}_{ \gcd(abcd,p)=1 }\bigg) \\
\label{eq:alphaqxdef}     &= \sum_{q_0=q_{1} q_{2} q_{3}q_{4}q_5 q_6 q_7}   \alpha(\smqty(a & b \\c & d);q_0) \prod_{i=1}^7 v_i(\smqty(a & b \\c & d);q_i)=: \sum_{q_0=q_{1} q_{2} q_{3}q_{4}q_5 q_6 q_7} \alpha(\smqty(a & b \\c & d);\mathbf{q}),
\end{align}
where
\begin{align*}
    v_1(\smqty(a & b \\c & d);q_1) = \mathbf{1}_{q_1|a}, \quad   v_2(\smqty(a & b \\c & d);q_2) = \mathbf{1}_{q_2|b},   \quad   v_3(\smqty(a & b \\c & d);q_3) = \mathbf{1}_{q_3|c},   \quad   v_4(\smqty(a & b \\c & d);q_4) = \mathbf{1}_{q_4|d},  \\
      v_5(\smqty(a & b \\c & d);q_1) = \mathbf{1}_{q_5|\gcd(a,d)}, \quad   v_1(\smqty(a & b \\c & d);q_1) = \mathbf{1}_{q_6|\gcd(b,c)}, \quad   v_1(\smqty(a & b \\c & d);q_1) = \mathbf{1}_{\gcd(abcd,q_7)=1}.
\end{align*}
Denote also
\begin{align*}
   V(\smqty(a & b \\c & d);\mathbf{q}) := \prod_{i=1}^7 v_i(\smqty(a & b \\c & d);q_i).
\end{align*}
We will apply Theorem \ref{thm:twisteddetwbound} with the group
\begin{align} \label{eq:Q1Q2def}
    \Gamma = \Gamma_2(q_1q_2 q_5 q_6 q_7,q_3q_4 q_5 q_6 q_7)=: \Gamma_2(Q_1,Q_2).
\end{align}
To prepare we need some lemmas. For any prime number $p$ we write $p_j=\gcd(p,q_j)$, $\mathbf{p}=(p_1,\dots,p_7)$, and $P_j := \gcd(p,Q_j)$.
\begin{lemma} \label{le:Rrsbound}
    Let $b\in \Z$ and let  $p=p_j$ for some $j\in\{1,2,3,4,5,6,7\}$. Define
    \begin{align*}
    \Pi: \,&\Gamma_2(P_1,P_2) \backslash \SL(\Z) \to \Z/P_1\Z \times \Z/P_2\Z, \\
        \Pi(\smqty(a_0 & b_0 \\ c_0 & d_0))  &:= \bigg( a_0 d_0 \frac{-h}{a_0 d_0-b_0c_0}, a_0 c_0 \frac{-h}{a_0 d_0-b_0c_0} \bigg).
    \end{align*}
    Let 
    \begin{align*}
        R_p(r,s,b;\mathbf{p})& :=   \sum_{\tau= \smqty(a_0 & b_0 \\ c_0 & d_0) \in \Gamma_2(P_1,P_2) \backslash \SL(\Z)} \mathbf{1}_{\Pi(\smqty(a_0 & b_0 \\ c_0 & d_0))  = (r,s)}   V(\smqty(a_0 & b_0 \\c_0 & d_0); \mathbf{p})   V(\smqty(a_0 & b_0 \\c_0 & d_0)\smqty(\pm 1 & b \\ & \pm 1); \mathbf{p}).
    \end{align*}
    Then for $\gcd(r(r+h)s,p)=1$ we have
    \begin{align*}
         R_p(r,s,b;\mathbf{p}) \equiv \mathbf{1}_{p=p_7} 
    \end{align*}
    and for $p|r(r+h)s$ we have
    \begin{align*}
         R_p(r,s,b;\mathbf{p}) \leq 1.
    \end{align*}
\end{lemma}
\begin{proof}
    Note that
    \begin{align*}
        \mqty(a_0 & b_0 \\c_0 & d_0)\mqty(\pm 1 & b \\ & \pm 1) =   \mqty(a_0 & b_0 +b a_0 \\c_0 & d_0 + b c_0).
    \end{align*}
    Therefore, for $\gcd(r(r+h)s,p)=1$ we have $\gcd(a_0b_0c_0d_0,p)=1$  and the claim follows from the fact that the map
    \begin{align*}
       \{ \smqty(a_0 & b_0 \\ c_0 & d_0) \in \Gamma_2(P_1,P_2) \backslash \SL(\Z): \gcd(a_0b_0c_0d_0,p)=1 \} &\to \{(r,s) \in \Z/P_1\Z \times \Z/P_2\Z : \, (r(r+h)s,p)=1\} \\
       \smqty(a_0 & b_0 \\ c_0 & d_0) &\mapsto  \bigg(a_0d_0 \frac{-h}{ad-bc},a_0c_0 \frac{-h}{ad-bc}\bigg)
   \end{align*}
  is a bijection. Indeed, clearly this map is surjective (consider $\smqty(a_0 & b_0 \\ c_0 & d_0) = \smqty(1 & b_0 \\ s & r) $ with $r-b_0 s  \equiv 1 \pmod{p}$) and both sets have the same cardinality. Note also that the count is empty if $p=p_j$ for $j \leq 6$ since then $V$ is supported on $p| a_0b_0c_0d_0$ and $(a_0b_0c_0d_0,p)=1$.

  For  $p|r(r+h)s$ consider first $p=p_7$. Then
  \begin{align*}
      R_p(r,s,b;\mathbf{p})  \leq  \sum_{\tau= \smqty(a_0 & b_0 \\ c_0 & d_0) \in \Gamma_2(p,p) \backslash \SL(\Z)} \mathbf{1}_{\Pi(\smqty(a_0 & b_0 \\ c_0 & d_0))=(r,s)}\mathbf{1}_{(a_0b_0c_0d_0,p)=1}  \leq 1.
  \end{align*}
  For $p=p_1$ we have
  \begin{align*}
      R_p(r,s,b;\mathbf{p})  &\leq  \sum_{\tau= \smqty(a_0 & b_0 \\ c_0 & d_0) \in \Gamma_2(p,1) \backslash \SL(\Z)}  \mathbf{1}_{a_0  \equiv 0 \pmod{p}} \leq 1
  \end{align*}
and for $p=p_6$  we have
  \begin{align*}
      R_p(r,s,b;\mathbf{p})  &\leq  \sum_{\tau= \smqty(a_0 & b_0 \\ c_0 & d_0) \in \Gamma_2(p,p) \backslash \SL(\Z)}  \mathbf{1}_{a_0 \equiv d_0 \equiv 0 \pmod{p}}    \leq 1.
  \end{align*}
The other cases are similar.
\end{proof}
We require the following lemma to give bounds of the type \eqref{eq:thmtwisteddetwbound1}.
\begin{lemma} \label{le:balancedlemma}
    Let $q$ be square-free and let $t:\Z/q\Z \to \C$. Let $q_0 | q$ and let $ \alpha(\smqty(a & b \\c & d); \mathbf{q})$ be as in \eqref{eq:alphaqxdef} for $ad-bc=-h$ and extend it to matrices with general determinant coprime to $q$ by
    \begin{align*}
        \alpha(\smqty(a & b \\c & d); \mathbf{q}) :=  \alpha\bigg(\smqty(a \frac{-h}{ad-bc} & b \frac{-h}{ad-bc}  \\c& d); \mathbf{q} \bigg),
    \end{align*}
    where $\frac{1}{ad-bc}$ is the inverse of $ad-bc$ modulo $q$.  Let $\mathcal{N}_h(t)$ be as in \eqref{eq:Ntdef}.  Then
    \begin{align*}
   \sum_{\tau \in \Gamma_2(Q_1,Q_2) \backslash \SL(\Z)} \alpha(\tau \smqty(\pm 1& b \\  & \pm 1); \mathbf{q}) \overline{\alpha (\tau; \mathbf{q})}  \ll_\eps q^\eps  \mathcal{N}_h(t) \gcd(b,q_7)
\end{align*}
\end{lemma}
\begin{proof}
Recall the definition of $u$ in \eqref{def:urp} and set
\begin{align*}
    \alpha_r(\smqty(a & b \\ c & d);\mathbf{q}) := u(ad\tfrac{-h}{ad-bc}-r;q_0)  V(\smqty(a & b \\c & d);\mathbf{q}).
\end{align*}
Then
\begin{align*}
     &\sum_{\tau \in \Gamma_2(Q_1,Q_2) \backslash \SL(\Z)} \alpha(\tau \smqty(\pm 1 & b \\  & \pm 1) ;\mathbf{q}) \overline{\alpha (\tau;\mathbf{q})}  \\
     &= \sum_{r_1,r_2 \pmod{q_0}} t(r_1;q_0) \overline{t(r_2;q_0)}    \sum_{\tau \in \Gamma_2(Q_1,Q_2)  \backslash \SL(\Z)} \alpha_{r_1}(\tau \smqty(\pm 1& b \\  & \pm 1);\mathbf{q}) \overline{\alpha_{r_2} (\tau;\mathbf{q})} \\
     &=  \sum_{r_1,r_2 \pmod{q}} t(r_1) \overline{t(r_2)}  \frac{1}{(q/q_0)^2}  \sum_{\tau \in \Gamma_2(Q_1,Q_2)  \backslash \SL(\Z)} \alpha_{r_1}(\tau \smqty(\pm 1& b \\  & \pm 1) ;\mathbf{q}) \overline{\alpha_{r_2} (\tau;\mathbf{q})},
\end{align*}
so that it suffices to prove that
\begin{align*}
W(r_1,r_2;\mathbf{q}) := \sum_{\tau \in \Gamma_2(Q_1,Q_2)  \backslash \SL(\Z)} \alpha_{r_1}(\tau \smqty(\pm 1 & b \\  & \pm 1) ;\mathbf{q}) \overline{\alpha_{r_2} (\tau;\mathbf{q})}    
 \\
 \ll_\eps  q^\eps \gcd(b,q_7)  U_h(r_1,r_2;q_0). 
\end{align*}
By the Chinese remainder theorem we have
\begin{align*}
W(r_1,r_2;\mathbf{q})  = \prod_{p|q_0} W(r_1,r_2;\mathbf{p}), 
\end{align*}
where $p_j:=\gcd(p,q_j),P_j:=\gcd(p,Q_j)$, and
\begin{align*}
    W(r_1,r_2;\mathbf{p}) = \sum_{\tau \in \Gamma_2(P_1,P_2)  \backslash \SL(\Z)} \alpha_{r_1}(\tau \smqty(\pm 1 & b \\  & \pm 1) ;\mathbf{p}) \overline{\alpha_{r_2} (\tau;\mathbf{p})} .
\end{align*}
It then suffices to prove that
\begin{align*}
  W(r_1,r_2;\mathbf{p}) \ll \gcd(b,p_7) \bigg( \mathbf{1}_{r_1\equiv 0 \pmod{p}} + \mathbf{1}_{r_2\equiv 0 \pmod{p}} + \mathbf{1}_{r_2\equiv 1 \pmod{p}} + \mathbf{1}_{r_1 \equiv r_2 \pmod{p}} + p^{-1} \bigg).
\end{align*}
We have (with $R_p(r,s,b;\mathbf{p})$ as in Lemma \ref{le:Rrsbound})
 \begin{align*} 
       W(r_1,r_2;\mathbf{p}) 
      & = \sum_{r,s \pmod{p}}u(r-r_2;p) u(r+bs-r_1;p) R_p(r,s,b;\mathbf{p}) \\
      & =W_1(r_1,r_2;\mathbf{p}) +W_2(r_1,r_2;\mathbf{p}),
   \end{align*}
   where $W_1(r_1,r_2;\mathbf{p})$ has $(r(r+h)s,p)=1$ and $W_2(r_1,r_2;\mathbf{p})$ has $p| r(r+h)s$.
  
   To bound $W_2$ we use
   \begin{align*}
       |u(r-r_2;p) u(r+bs-r_1;p)|  \leq \mathbf{1}_{\substack{ r_2\equiv r  \pmod{p} \\ r_1 \equiv r+bs \pmod{p}}}  + p^{-1} (\mathbf{1}_{ r_2 \equiv r  \pmod{p}} + \mathbf{1}_{r_1 \equiv  r+bs \pmod{p}}) + p^{-2}
   \end{align*}
and that by Lemma \ref{le:Rrsbound} $ R_p(r,s,b;\mathbf{p}) \ll 1$ to get
   \begin{align*}
     W_2(r_1,r_2;\mathbf{p}) &\ll   \sum_{\substack{r,s \pmod{p}  \\ p| r(r+h)s}} \bigg(   \mathbf{1}_{\substack{ r_2\equiv r  \pmod{p} \\ r_1 \equiv r+bs \pmod{p}}}  + p^{-1} (\mathbf{1}_{ r_2 \equiv r  \pmod{p}} + \mathbf{1}_{r_1 \equiv  r+bs \pmod{p}}) + p^{-2} \bigg) \\
       &\ll \bigg( \mathbf{1}_{r_1\equiv 0 \pmod{p}} + \mathbf{1}_{r_2\equiv 0 \pmod{p}} + \mathbf{1}_{r_2\equiv -h \pmod{p}} + \mathbf{1}_{r_1 \equiv r_2 \pmod{p}} + p^{-1} \bigg).
   \end{align*}

   To bound $W_1$, we have by Lemma \ref{le:Rrsbound} that it is non-zero only for $p=p_7$ and then it is
  \begin{align*}
   W_1(r_1,r_2;p) =    \sum_{\substack{r,s \pmod{p}  \\ \gcd(r(r+h)s,p)=1}} u(r-r_2;p) u(r+bs-r_1;p) =   &\sum_{\substack{r,s \pmod{p}  }} u(r-r_2;p) u(r+bs-r_1;p) \\
       &- \sum_{\substack{r,s \pmod{p}  \\ p| r(r+h)}} u(r-r_2;p) u(r+bs-r_1;p) .
  \end{align*}
  The second term is bounded by the same argument as for $ W_2(r_1,r_2;p)$. For the first term we note that for $\gcd(b,p)=1$ we have by a change of variables, using the definition of $u(r;p)$,
  \begin{align*}
      \sum_{\substack{r,s \pmod{p}  }} u(r-r_2;p) u(r+bs-r_1;p) = \sum_{\substack{r,s \pmod{p} }} u(r;p) u(s;p) = 0,
  \end{align*}
  and for $\gcd(b,p)=p$ we have
  \begin{align*}
       &\sum_{\substack{r,s \pmod{p}  }} u(r-r_2;p) u(r+bs-r_1;p) = \gcd(b,p)     \sum_{\substack{r \pmod{p}  }}  u(r-r_2;p)  u(r-r_1;p) \\
       & \leq \gcd(b,p)     \sum_{\substack{r \pmod{p}  }} (\mathbf{1}_{\substack{ r_2\equiv r  \pmod{p} \\ r_1 \equiv r \pmod{p}}}  + p^{-1} (\mathbf{1}_{ r_2 \equiv r  \pmod{p}} + \mathbf{1}_{r_1 \equiv  r \pmod{p}}) + p^{-2}) \\
      & \ll \gcd(b,p) (  \mathbf{1}_{r_1\equiv r_2 \pmod{p}} + p^{-1} ).
  \end{align*}
\end{proof}
By the decomposition \eqref{eq:tsplitbalanced} it remains to prove Proposition \ref{prop:divisorperiodicweak}  separately for each $\alpha(\smqty(a & b \\c & d);\mathbf{q})$  as in \eqref{eq:alphaqxdef}, since for the main terms we can reverse the decomposition \eqref{eq:tsplitbalanced} to obtain the desired main term in terms of $t(r)$.  The polynomial $P_h$ that one gets is clearly independent of $q$ and therefore must agree with the one in \cite{Motohashidivisor}.

\begin{proof}[Proof of Proposition \ref{prop:divisorperiodicweak} for $ \alpha(\smqty(a & b \\c & d);\mathbf{q})$]
   Let $\psi:\R \to [0,1]$ be a fixed smooth function, supported on $[1,2]$, and satisfying
\begin{align*}
    \int_\R \psi(1/x) \frac{\d x}{x} = 1.
\end{align*}
Inserting this thrice and making a change of variables we have
\begin{align*}
  \sum_{ \substack{ad-bc=-h}} G(ad/X) \alpha(\smqty(a & b \\c & d);\mathbf{q})&=\int_{\R^3} \sum_{\substack{ad-bc=-h}} G(ad/X)\psi(a/A)\psi(c/C)\psi(d/D)\alpha(\smqty(a & b \\c & d);\mathbf{q}) \frac{\d A \d C \d D}{ACD} \\
    &=: \int_{\substack{ A D \leq 2 X \\ C \leq 2(X+|h|)}} S_{\psi}(A,C,D) \frac{\d A \d C \d D}{ACD}.
\end{align*}
Denote also $B := (X+h)/C$.  By symmetry ($ad \leftrightarrow bc$, $a \leftrightarrow d$, $b \leftrightarrow c$) we may assume that
\begin{align*}
    C \ll  X^\eta D \ll X^\eta A \ll  X^{2\eta} B.
\end{align*}
Note that here we have used the assumption $-X/2\leq h \leq X^{1+\eta}$ to infer that $X \ll AD \ll BC \ll X^{1+\eta} .$ Recall $Q_1,Q_2$ from \eqref{eq:Q1Q2def}. If
\begin{align}\label{eq:trivialrangedivisor}
   A > X^{2\eta} Q_1 C \quad \text{or} \quad D > X^{2\eta} Q_2 C,
\end{align}
we obtain the right  main term with sufficiently small error term by Poisson summation. For example, for $ A > X^{2\eta} Q_1 C $ we have
\begin{align*}
    &\sum_{\substack{ad-bc=-h}} G(ad/X)\psi(a/A)\psi(c/C)\psi(d/D) \alpha(\smqty(a & b \\c & d);\mathbf{q}) \\
    & = \sum_{\substack{a,c,d \\ad \equiv -h \pmod {c}}} G(ad/X)\psi(a/A)\psi(c/C)\psi(d/D) \alpha(\smqty(a & b \\c & d);\mathbf{q})
\end{align*}
where $a$ has a smooth weight on $a \asymp A$. Splitting into residue classes $\pmod{ Q_1}$ and $\pmod{c}$, combining the two residue classes,  applying Poisson summation on $a$ (using $A \geq X^{2\eta}  Q_1 C$),  we get the expected main term and an error term $\ll_J X^{-J}$ for any $J>0$.

It remains to consider the complementary range of \eqref{eq:trivialrangedivisor}, where we can assume that
\begin{align} \label{eq:complementaryrange}
    X^{-\eta}\ll \frac{A}{C} \leq Q_1 X^{2\eta}, \quad X^{-\eta} \ll \frac{D}{C} \leq  Q_2 X^{2\eta} .
\end{align}
We now apply Theorem \ref{thm:twisteddetwbound} to $ S_{\psi}(A,C,D) $ with  
\begin{align*}
\Gamma &= \Gamma_2(Q_1,Q_2)\\
    \beta_{h_0}&=\mathbf{1}_{h_0=-h}, \quad \quad \gamma_k=1_{k=1}, \quad \quad \xi_h=1, \\
   \alpha(\smqty(a& b \\ c & d)) &= \alpha(\smqty(a & b \\c & d);\mathbf{q}) =  \alpha\bigg(\mqty(a\frac{-h}{ad-bc} & b\frac{-h}{ad-bc} \\c & d);\mathbf{q}\bigg) \in \mathcal{A}(Q_1,Q_2,1,1), 
\end{align*}
where the last expression for $\alpha$ extends its definition to  $\MM_2(\Z)$ with determinant coprime to $q$. As we are only interested in $k=1$, we can adopt the notation in Remark \ref{rem:k=1blackbox} and we have by Lemma \ref{le:balancedlemma}
\begin{align*}
    |w(\smqty( \pm 1 & b \\ & \pm 1))|&\ll_\eps   q^\eps (b,q_7) \mathcal{N}(t).
\end{align*}
Thus, since $X^{-\eta} \ll D/C \leq  Q_2 X^{2\eta}$, we have \eqref{eq:thmtwisteddetwbound1} with $ \mathcal{K}_+ =    Q_2 \mathcal{N}_h(t)$. Note that this matches the contribution from $b=0$ for $q_7=q_0$, which is the generic case.
We also have by $X^{-\eta} \leq A/C \leq  Q_1 X^{2\eta}, C,D \ll X^{1/2+O(\eta)}$
\begin{align*}
    \mathcal{R}_1&\leq X^{O(\eta)} h^{\vartheta_q}\left(1+\bigg(\frac{X}{h  Q_2 }\bigg)^{\theta} \right)\left(1+ \bigg(\frac{1}{  Q_2}\bigg)^{1/2-\theta}\right)\leq X^{O(\eta)}  h^{\theta}\left(1+\bigg(\frac{X}{h  Q_2 }\bigg)^{\theta} \right), \\
    \mathcal{R}_2&\leq   X^{O(\eta)} \left(1+\bigg(\frac{ X}{ Q_2}\bigg)^{\theta} \right)\left(1+ \bigg(\frac{h }{ Q_2 }\bigg)^{1/2-\theta}\right),
\end{align*}
and
\begin{align*}
   \mathcal{R}_0 = X^{O(\eta)} \frac{ \|\beta\|_1  A^{1/2} }{\|\beta\|_2   Q_1^{1/2} C^{1/2}}\leq 
   X^{O(\eta)} .
\end{align*}
Therefore, by Theorem \ref{thm:twisteddetwbound} the error term is for $h \leq  Q_2$, using $\mathcal{R}_2$,
\begin{align*}
 \ll  X^{1/2+\theta+O(\eta)}  q_0^{1/2-\theta}\mathcal{N}_h(t)^{1/2}  
\end{align*}
and for $h \geq  Q_2$, using $\mathcal{R}_1$,
\begin{align*}
\ll  X^{1/2 +O(\eta)}  q_0^{1/2} \mathcal{N}_h(t)^{1/2}  h^{\theta} + X^{1/2+\theta +O(\eta)}  q_0^{1/2-\theta} \mathcal{N}_h(t)^{1/2}.
\end{align*}
Thus, in either case we get
\begin{align*}
    \ll  X^{1/2+O(\eta)}  q_0^{1/2}\mathcal{N}_h(t)^{1/2}   h^{\theta}  + X^{1/2+\theta+O(\eta)}  q_0^{1/2-\theta}\mathcal{N}_h(t)^{1/2}.
\end{align*}
The main term may be computed by reversing the smooth  dyadic partition and reversing the decomposition \eqref{eq:alphaqxdef}.
\end{proof}
It is now quick to deduce Theorem \ref{thm:divisorperiodic} from Proposition \ref{prop:divisorperiodicweak}.
\begin{proof}[Proof of Theorem \ref{thm:divisorperiodic}]
    We require another decomposition of $t(n)$. To this end, define
    \begin{align*}
        u_0(n;q_0) &:= \mathbf{1}_{n \equiv 0 \pmod{q_0}}, \quad       u_h(n;q_1): = \mathbf{1}_{n \equiv -h \pmod{q_1}}, \\
        v(n;p) &:=  \mathbf{1}_{n \not \equiv 0,-h \pmod{p}}, \quad v(n;q_2) = \prod_{p|q_2}  v(n;p).
    \end{align*}
    Then
    \begin{align*}
        t(n) = t(n) \prod_{p|q} \bigg(u_0(n;p) + u_h(n;p) + v(n;p) \bigg) = \sum_{q=q_0 q_1 q_2} t(n)  u_0(n;q_0) u_h(n;q_1) v(n;q_2). 
    \end{align*}
    This gives a decomposition of $t(n)$ into $  d_3(q)\ll_\eps q^\eps$ functions $t(n)  u_0(n;q_0) u_h(n;q_1) v(n;q_2).$ Then Theorem \ref{thm:divisorperiodic} follows by applying Proposition \ref{prop:divisorperiodicweak} separately for each of these, since
    \begin{align*}
        \mathcal{N}_h &\bigg(t(n)  u_0(n;q_0) u_h(n;q_1) v(n;q_2)\bigg) \\
       & = \sum_{r_1,r_2 \pmod{q}} |t(r_1) t(r_2)| \prod_{p_0|q_0} u_0(r_1;p_0) u_0(r_2;p_0) U_h(r_1,r_2;p_0)  \\
    &   \hspace{110pt} \times \prod_{p_1|q_1}  u_h(r_1;p_1) u_h(r_2;p_1) U_h(r_1,r_2;p_1) \\
     & \hspace{110pt}  \times \prod_{p_2|q_2} v(r_1;p_2) v(r_2;p_2) U_h(r_1,r_2;p_2) \\
      & \ll  \sum_{r_1,r_2 \pmod{q}} |t(r_1) t(r_2)| \prod_{p|q} \bigg( \mathbf{1}_{r_1 \equiv  r_2 \pmod{p}} + p^{-1} \bigg) \\
      & \ll \sum_{q=q_1 q_2} \frac{1}{q_2}  \sum_{\substack{r_1,r_2 \pmod{q}  \\ r_1 \equiv r_2 \pmod{q_1}}} |t(r_1) t(r_2)| \\
      & \ll \sum_{q=q_1 q_2} \frac{1}{q_2}  \sum_{\substack{r_1,r_2 \pmod{q}  \\ r_1 \equiv r_2 \pmod{q_1}}} (|t(r_1)|^2 + t(r_2)|^2) \ll_\eps q^\eps \|t\|_2^2.
    \end{align*}
\end{proof}

\section{Proof of Theorems  \ref{cor:Lfunction} and \ref{cor:Lfunctiononechar}} \label{sec:proofapplLfunc}

For the proof we require the following character sum bound, which will be used to bound $w(\sigma)=w(\sigma,I,I)$.  For any prime $p$ and Dirichlet character $\chi$ denote the local conductor
\begin{align*}
    \cond(\chi;p) := \gcd(\cond(\chi),p^\infty) = \max\{p^k | \cond(\chi)\}.
\end{align*}
The following statement is more general than required for Theorems  \ref{cor:Lfunction} and \ref{cor:Lfunctiononechar} in that it considers the case of four different characters. The restriction to square-free moduli can be removed with additional work.
\begin{lemma} \label{le:charactersumbound}
 Let $Q_1,Q_2 \in \Z_{>0}$ be square-free and  let $\chi_1,\psi_1,\chi_2,\psi_2$ be Dirichlet characters to modulus dividing $Q_1,Q_2$,respectively, and denote $Q_0=\gcd(Q_1,Q_2)$. Let
    \begin{align*}
    \alpha(\smqty(a &b \\ c &d)) =& \chi_1(a) \psi_1(b) \chi_2(c) \psi_2(d), \\
    w(\sigma) = & \sum_{\tau \in \Gamma_2(Q_1,Q_2) \backslash \SL(\Z)} \alpha(\tau \sigma) \overline{\alpha}(\tau).
    \end{align*}
    Then
    \begin{align*}
         \sum_{0<  |c| \leq C} |w(\smqty(\pm 1 &  \\ c & \pm 1))| &  \ll_\eps C(Q_1Q_2)^{1+\eps}  \prod_{\substack{p| Q_j \\ p \nmid Q_0}} \frac{1}{\cond(\chi_j;p)}  \prod_{\substack{ p \mid Q_0}} \frac{1}{\max_j \{\cond(\chi_j;p)\}}   \\
   \sum_{0<|b| \leq  B} |w(\smqty(\pm 1 & b \\  & \pm 1))| & \ll_\eps B  (Q_1Q_2)^{1+\eps}  \prod_{\substack{p| Q_j \\ p \nmid Q_0}} \frac{1}{\cond(\psi_j;p)}  \prod_{\substack{ p \mid Q_0}} \frac{1}{\max_j \{\cond(\psi_j;p)\}}.
    \end{align*}
\end{lemma}
\begin{proof}
The proofs are essentially the same so we consider only the sum over $b$. By Lemma \ref{lem:wsimplification} we have
\begin{align*}
    w(\smqty(\pm 1 & b \\  & \pm 1))= \chi_1 \chi_2 (\pm 1)\sum_{(a_2,b_2) \in \P^{1}_{Q_1} } \sum_{\substack{(c_2,d_2) \in \P^{1}_{Q_2} \\ \gcd(a_2d_2-b_2c_2,Q_0)=1}  } &\psi_1(ba_2\pm b_2)  \overline{\psi_1 }(b_2) |\chi_1  (a_2 )|^2 \\
   &\times\psi_2 (b c_2 \pm d_2) \overline{  \psi_2} (d_2) |\chi_2 (c_2)|^2 
\end{align*}
Recall that a Dirichlet character $\chi$ modulo $q$ is a product of Dirichlet characters modulo $p^k || q$
\begin{align*}
    \chi = \prod_{p |q }\chi^{(p)}.
\end{align*}
For any  $q_1$ with $\gcd(q_1,q/(q,q_1))=1$  we define
\begin{align*}
    \chi^{(q_1)} :=  \prod_{p |q_1 } \chi^{(p)}.
\end{align*}
Denoting
\begin{align*}
  S(q_1,q_2) :=    \sum_{(a_2,b_2) \in \P^{1}_{q_1} } \sum_{\substack{(c_2,d_2) \in \P^{1}_{q_2} \\ \gcd(a_2d_2-b_2c_2,\gcd(q_1,q_2))=1}  }  &\psi_1^{(q_1)} (ba_2\pm b_2)  \overline{\psi_1}^{(q_1)}  (b_2) |\chi_1^{(q_1)}  (a_2 )|^2 \\
   &\times\psi_2^{(q_2)}  (b c_2 \pm d_2) \overline{  \psi_2}^{(q_2)}  (d_2) |\chi_2^{(q_2)}  (c_2)|^2,
\end{align*}
we have by the Chinese remainder theorem
\begin{align*}
    w(\smqty(\pm 1 & b \\  & \pm 1))= \chi_1 \chi_2 (\pm 1)  \prod_{\substack{ \\ p | Q_0}}  S(p,p)  \prod_{\substack{p_j | Q_j \\ p_j \nmid Q_0}}  S(p_1,p_2).
\end{align*}
For $p_j \nmid Q_0$ we have $p_1 \neq p_2$ so that
\begin{align*}
     S(p_1,p_2) = &\bigg(\sum_{(a_2,b_2) \in \P^1_{p_1} }  \psi_1^{(p_1)} (ba_2\pm b_2)  \overline{\psi_1}^{(p_1)}  (b_2) |\chi_1^{(p_1)}  (a_2 )|^2 \bigg)  \\
&\times\bigg(\sum_{\substack{(c_2,d_2) \in \P^1 _{p_2 } }} \psi_2^{(p_2)}  (b c_2 \pm d_2) \overline{  \psi_2}^{(p_2)}  (d_2) |\chi_2^{(p_2)}  (c_2)|^2 \bigg),  
\end{align*}
If  $\gcd (\cond (\psi_j),p_j)|b$  then we use the trivial bound
\begin{align*}
   \sum_{(x,y) \in \P^1_{p_j}}  \psi_j^{(p_j)} (bx\pm y)  \overline{\psi_j}^{(p_j)}  (y) |\chi_j^{(p_j)}  (x )|^2   \ll p_j.
\end{align*}
Otherwise $\gcd(yb,p)=1$, $\psi_j^{(p_j)} $ is non-trivial,  and by writing $r = bx/y$ we get
\begin{align*}
    &\bigg|\sum_{(x,y) \in \P^1_{p_j}}  \psi_j^{(p_j)} (bx\pm y)  \overline{\psi_j}^{(p_j)}  (y) |\chi_j^{(p_j)}  (x )|^2\bigg|= \bigg|\sum_{r \in \Z/p_j \Z} \psi_j^{(p_j)} (r\pm 1)  |\chi_j^{(p_j)}  (r )|^2  \bigg|  \ll 1.
\end{align*}
 Combining the bounds we have for $p_1 \neq p_2$
\begin{align*}
    S(p_1,p_2) \ll  ( 1+ p_1  \mathbf{1}_{\cond(\psi_1;p_1) |b} )( 1+ p_2  \mathbf{1}_{\cond(\psi_2;p_2) |b} )
\end{align*}

For $p| Q_0$ we write
\begin{align*}
    S(p,p) =   S_{0}(p,p) -  S_{1}(p,p)
\end{align*}
where 
\begin{align*}
     S_0(p,p) = &\bigg(\sum_{(a_2,b_2) \in \P^1_{p}}  \psi_1^{(p)} (ba_2\pm b_2)  \overline{\psi_1}^{(p)}  (b_2) |\chi_1^{(p)}  (a_2 )|^2 \bigg)  \\
&\times\bigg(\sum_{\substack{(c_2,d_2) \in \P^1 _{p } }} \psi_2^{(p)}  (b c_2 \pm d_2) \overline{  \psi_2}^{(p)}  (d_2) |\chi_2^{(p)}  (c_2)|^2 \bigg),  
\end{align*}
and
\begin{align*}
     S_1(p,p) = \sum_{(a_2,b_2) \in \P^1_{p}} & \psi_1^{(p)} (ba_2\pm b_2)  \overline{\psi_1}^{(p)}  (b_2) |\chi_1^{(p)}  (a_2 )|^2   \\
&\sum_{\substack{(c_2,d_2) \in \P^1 _{p}  \\ p| a_2 d_2-b_2 c_2} } \psi_2^{(p)}  (b c_2 \pm d_2) \overline{  \psi_2}^{(p)}  (d_2) |\chi_2^{(p)}  (c_2)|^2,  
\end{align*}
By the previous argument we have
\begin{align*}
    S_0(p,p) \ll   ( 1+ p  \mathbf{1}_{\cond(\psi_1;p) |b} )( 1+ p  \mathbf{1}_{\cond(\psi_2;p) |b} ).
\end{align*}
For $S_1(p,p)$ we note that $p|a_2 d_2-b_2c_2$ means precisely that $(a_2,b_2) \sim (c_2,d_2)$ as points in $\P^1_p$ so that trivially
\begin{align*}
     S_1(p,p)  \ll p
\end{align*}
Therefore,
\begin{align*}
      S(p,p) \ll ( 1+ p  \mathbf{1}_{\cond(\psi_1;p) |b} )( 1+ p  \mathbf{1}_{\cond(\psi_2;p) |b} ) + p.
\end{align*}
Defining a multiplicative function $F(\cdot,b)$ by 
\begin{align*}
    F(p^k; b) := \begin{cases}
    1+p\mathbf{1}_{\cond(\psi_j;p) |b}  \quad & p | Q_j, p \nmid Q_0, \\
    ( 1+ p  \mathbf{1}_{\cond(\psi_1;p) |b} )( 1+ p  \mathbf{1}_{\cond(\psi_2;p) |b} ) + p   \quad &p | Q_0 \end{cases},
\end{align*}
and $F(p^k,b)=1$ for $p \nmid Q_1Q_2$, we have
\begin{align*}
  | w(\smqty(\pm 1 & b \\  & \pm 1))| \ll_\eps (Q_1Q_2)^{\eps}  F(Q_1 Q_2; b).
\end{align*}
Then  $F(Q_1 Q_2; b) = F(Q_1 Q_2, \gcd(b,(Q_1Q_2)))$ and for $d_1 d_2|Q_1Q_2$ with $\gcd(d_1,d_2)=1$  we have \begin{align*}
    F(Q_1 Q_2; d_1 d_2) = F(Q_1 Q_2; d_1 ) F(Q_1 Q_2; d_2)
\end{align*}
Therefore, we get
\begin{align*}
    \sum_{0<b \leq B}  | w(\smqty(\pm 1 & b \\  & \pm 1))| &\ll_\eps (Q_1Q_2)^{\eps} \sum_{0<b \leq B}   F(Q_1 Q_2; b)  =(Q_1Q_2)^{\eps} \sum_{d | Q_1 Q_2} F(Q_1 Q_2; d)  \sum_{\substack{0<b \leq B \\ \gcd(b,Q_1Q_2)=d}} 1 \\
    &\ll_\eps (Q_1Q_2)^{\eps}  B \sum_{d | Q_1 Q_2} \frac{F(Q_1 Q_2; d)}{d} \leq (Q_1Q_2)^{\eps} B \prod_{p| Q_1 Q_2} \bigg( \sum_{0 \leq k \leq 2 } \frac{F(Q_1 Q_2; p^k)}{p^k} \bigg)    \\
   & \ll_\eps B (Q_1Q_2)^{1+\eps} \prod_{\substack{p| Q_j \\ p \nmid Q_0}} \frac{1}{\cond(\psi_j;p)}  \prod_{\substack{ p \mid Q_0}} \frac{1}{\max_j \{\cond(\psi_j;p)\}},
\end{align*}
since for $p|Q_j,p\nmid Q_0$ we have $\frac{F(Q_1 Q_2; p^k)}{p^k}  \leq 2\frac{p}{\cond(\psi_j;p)}$, and for $p| Q_0$ we have for both $j\in\{1,2\}$ that $\frac{F(Q_1 Q_2; p^k)}{p^k}  \leq 5\frac{p^2}{\cond(\psi_j;p)}$.
\end{proof}
\begin{proof}[Proof of Theorem \ref{cor:Lfunction}]
Let $\eta>0$ be small and let $U:\R \to [0,1]$ be a smooth function such that
\begin{align*}
    U(\xi) = 1, \quad |\xi| \leq T^{-1+\eta} \quad \text{and} \quad  U(\xi) = 0, \quad |\xi| \leq 2T^{-1+\eta}.
\end{align*}
Assume that $U^{(j)}(\xi) \ll_j |\xi|^{-j}$ for $\xi > 0$.
Let $V:(0,\infty) \to [0,1]$ be a fixed smooth function with
\begin{align*}
    V(\xi) + V(\xi^{-1}) = 1, \quad \xi >0, \quad \text{and} \quad V(\xi) = 0, \quad \xi \geq 2.
\end{align*}
We denote
 \begin{align*}
          d_{\chi_1,\chi_2}(n)=\sum_{ad=n}\chi_1(a)\chi_2(d).
      \end{align*}
      Then,
      by similar reductions by the approximate functional equation as in \cite[Section 5]{berke}, it suffices to show an improved version of \cite[Proposition 5.5]{berke}. 
      That is, it suffices to show that 
\begin{align}\label{eq:M2asymp}
    M^{(2)}_{\chi_1,\chi_2}=\int P^{(2)}_{\chi_1,\chi_2}(\log t)\omega(t/T) \d t+O((q_1q_2)^{3/4}T^{1/2+\theta+O(\eta)}),
\end{align}
where 
\begin{align*}
    M^{(2)}_{\chi_1,\chi_2}(\omega) = &\int_\R \sum_{\substack{n_1,n_2 \geq 1 \\ n_1 \neq n_2} } \frac{d_{\chi_1,\chi_2}(n_1) d_{\overline{\chi_1},\overline{\chi_2}}(n_2)}{\sqrt{n_1n_2}} e\bigg( \frac{t}{2 \pi } \log \frac{n_2}{n_1}\bigg) U\bigg(\frac{n_2-n_1}{n_2}\bigg) V\bigg(\frac{n_1 }{ t \sqrt{q_1q_2} } \bigg) \omega(t/T) \d t \\
    =: &\int_\R S_{\chi_1,\chi_2} (t,U,V) \omega(t/T) \d t
\end{align*}
and $P^{(2)}_{\chi_1,\chi_2}$ is a polynomial of degree at most two whose coefficients depend only on $\chi_1,\chi_2,V$.

We now bring $ S_{\chi_1,\chi_2} (t,U,V) $ into the right shape for the application of Theorem \ref{thm:twisteddetwbound}. We start this process by sorting after the determinant. Denoting $h= n_2-n_1$, $n_2=ad$ and $n_1=bc$ we get
\begin{align*}
     S_{\chi_1,\chi_2}(t,U,V)=\sum_{h \neq 0} \sum_{\smqty(a& b \\ c & d)\in \MM_{2,h}(\Z)} \chi_1(a)\overline{\chi_1}(b)\overline{\chi_2}(c) \chi_2(d) g_{t,h}(a,c,d)
\end{align*}
where
\begin{align*}
    g_{t,h}(a,c,d)= e\bigg( -\frac{t}{2 \pi } \log( 1-\frac{h}{ad})\bigg) U\bigg(\frac{h}{ad}\bigg) V\bigg( \frac{ad}{t \sqrt{q_1 q_2}}\bigg) \frac{1}{h\sqrt{ad/h}\sqrt{ad/h-1}}.
\end{align*}
Here by symmetry we restrict to $h \geq 1$.  By Mellin inversion we have
\begin{align*}
    V\bigg( \frac{ad}{t \sqrt{q_1 q_2}}\bigg) = \frac{1}{2\pi } \int_{(0)} \bigg( \frac{ad}{t \sqrt{q_1 q_2}}\bigg)^{-s} \widetilde{V}(s) \d s,
\end{align*}
where integration by parts gives the bounds
\begin{align*}
   \widetilde{V}(s) = \int_0^\infty V(\xi) \xi^{s-1} \d \xi  \ll_{\sigma,J} (1+|s|)^{-J}.  
\end{align*}
Thus,
\begin{align*}
     S_{\chi_1,\chi_2}(t,U,V) =  \frac{1}{2\pi } \int_{(0)}  (t\sqrt{q_1q_2})^{s}   \widetilde{V}(s)  S_{\chi_1,\chi_2,s}(t,U)\d s
\end{align*}
with
\begin{align*}
    S_{\chi_1,\chi_2,s}(t,U) =  \sum_{h \geq 1} h^{s-1} \sum_{\smqty(a& b \\ c & d)\in \MM_{2,h}(\Z)} \chi_1(a)\overline{\chi_1}(b)\overline{\chi_2}(c) \chi_2(d) f_{t,s}\bigg(\frac{a}{\sqrt{h}},\frac{c}{\sqrt{h}},\frac{d}{\sqrt{h}}\bigg),
\end{align*}
where
\begin{align*}
    f_{t,s}(a,c,d)  = e\bigg( -\frac{t}{2 \pi } \log( 1-\frac{1}{ad})\bigg) U\bigg(\frac{1}{ad}\bigg)  (ad)^{-s} \frac{1}{\sqrt{ad}\sqrt{ad-1}}.
\end{align*}
The contribution from $|s| \geq T^\eta$ is negligible so we consider $|s| \leq T^\eta$. Inserting a dyadic decomposition for $h$ and a smooth dyadic decomposition for $a/\sqrt{h},c/\sqrt{h},d/\sqrt{h}$ we obtain sums of the form
\begin{align*}
     S_{\chi_1,\chi_2,s,t}(A,C,D,H) := \frac{1}{AD}  \sum_{h \sim H}  H h^{s-1} \sum_{\smqty(a& b \\ c & d)\in \MM_{2,h}(\Z)} \chi_1(a)\overline{\chi_1}(b)\overline{\chi_2}(c) \chi_2(d) f\bigg(\frac{a}{\sqrt{h}},\frac{c}{\sqrt{h}},\frac{d}{\sqrt{h}}\bigg).
\end{align*}
Here we need to consider the ranges of variables
\begin{align} \label{eq:HADrange}
    H \ll T^\eta AD/T, \quad T^{1-\eta} \ll AD,BC \ll T \sqrt{q_1q_2},
\end{align} 
and the involved weight
\begin{align*}
    f(a,c,d) :=   \frac{AD}{H} f_{t,s}(a,c,d) \psi\bigg(a \frac{\sqrt{H}}{A}\bigg) \psi\bigg(c\frac{\sqrt{H}}{C}\bigg) \psi\bigg(d\frac{\sqrt{H}}{D}\bigg)\in C^{J}_{\delta}\bigg( \frac{A}{\sqrt{H}},\frac{C}{\sqrt{H}},\frac{D}{\sqrt{H}}\bigg)
\end{align*}
with $\delta^{-1} \ll T^{O_J(\eta)}$ fulfills the assumption of Theorem \ref{thm:twisteddetwbound}.

By splitting into congruence classes and using Poisson summation on the longest variable, we get a main term with arbitrary power saving in the error term as long as
\begin{align}
    \max\{A/q_1,D/q_2\} > Z^\eta \min \{B,C\} \quad \text{or} \quad \max\{B/q_1,C/q_2\} > Z^\eta  \min\{A,D\}.
\end{align}
In the complementary range we may assume that
\begin{align} \label{eq:ratioassumption}
   q_2^{-1} Z^{-\eta} <  C/D \leq q_2 Z^\eta, \quad     q_1^{-1} Z^{-\eta} <  C/A \leq q_2 Z^\eta.
\end{align}
The choice of $\Gamma$ for applying Theorem \ref{thm:twisteddetwbound} depends on common divisors between $h$ and $q_1 q_2$. We denote $k=\gcd(h,(q_1q_2)^\infty)$ so that $\gcd(h,q_i)=\gcd(k,q_i)$ -- at a first pass the reader may wish to simplify the argument by assuming that $k=1$, which is the generic case. We have after a change of variables $h \mapsto hk$
\begin{align} \nonumber
     S_{\chi_1,\chi_2,s,t}(A,C,D,H) &=  \frac{1}{AD}\sum_{k|(q_1q_2)^\infty} k^{s}  \sum_{\substack{h \sim H/k \\ \gcd(h,q_1q_2)=1}} \frac{H}{k} h^{s-1} \sum_{\smqty(a& b \\ c & d)\in \MM_{2,h,k}(\Z)} \chi_1(a)\overline{\chi_1}(b)\overline{\chi_2}(c) \chi_2(d) \\
     & \hspace{150pt}\times f\bigg(\frac{a}{\sqrt{hk}},\frac{c}{\sqrt{hk}},\frac{d}{\sqrt{hk}}\bigg) \\ \label{eq:ADnormalization}
     &=: \frac{1}{AD}\sum_{k|(q_1q_2)^\infty} k^{s}  S_{\chi_1,\chi_2,s,t,k}(A,C,D,H/k).
\end{align}
For a fixed $k$ denote $r_j := q_j/(k,q_j)$. Since $q_j$ are square-free, we have $\gcd(r_1r_2,k)=1$ and we can fix a choice of representatives
\begin{align*}
    \SL(\Z) \backslash \MM_{2,1,k}(\Z) =T_{1,k} =\bigg\{ \mqty(1 & f r_1r_2 \\ & k), \quad 1\leq f < k, \quad \gcd(f,k)=1\bigg\}.
\end{align*}
Then for any $\tau= \smqty(a_0 & b_0 \\ c_0 & d_0) \in \MM_{2,h}(\Z)$ and $\sigma=\smqty(1 & f r_1r_2 \\ & k) \in T_{1,k}$ we have
\begin{align} \label{eq:tausigmadef}
     \tau \sigma =   \mqty(a_0 & b_0 \\ c_0 & d_0) \mqty(1 & f r_1r_2 \\ & k) =  \mqty(a_0 & a_0f r_1r_2+ b_0k \\ c_0 & c_0 f r_1r_2 + d_0k). 
\end{align}
Denoting the unique decomposition of characters
\begin{align*}
    \chi_j = \chi^{(k)}_{j} \chi_{j}^{(r_j)}, \quad \chi^{(k)}_{j} \in \widehat{(\Z/\gcd(k,q_j)\Z)^\times}, \quad\chi_{j}^{(r_j)} \in \widehat{(\Z/r_j\Z)^\times},
\end{align*}
we have for $g = \smqty(a & b \\ c& d) = \tau \sigma$ as above
\begin{align*}
    \chi_1(a)\overline{\chi_1}(b)\overline{\chi_2}(c) \chi_2(d)  = & \chi^{(k)}_{1} (f_j r_1 r_2) \chi^{(k)}_{2} (f_j r_1 r_2)   \chi_{1}^{(r_j)}(k a_0) \overline{\chi_{1}^{(r_1)}}(b_0) \chi_{2}^{(r_2)}(k c_0) \overline{\chi_{1}^{(r_2)}}(d_0) \mathbf{1}_{(a,b,k)=1} \mathbf{1}_{(c,d,k)=1} \\
    =: & \, \alpha_0(g) \mathbf{1}_{\gcd(a,b,k)=1} \mathbf{1}_{\gcd(c,d,k)=1}
\end{align*}
Expanding the conditions $\gcd(a,b,k)=1$ and $\gcd(c,d,k)=1$ via the M\"obius function we have (since by $\gcd(h,k)=1$ we have $k_1k_2 |k$)
\begin{align*}
 S_{\chi_1,\chi_2,s,t,k}(A,C,D,H) &= \sum_{k_1k_2 | k} \mu(k_1) \mu(k_2)  S_{\chi_1,\chi_2,s,t,k,k_j}(A,C,D,H), \\
 S_{\chi_1,\chi_2,s,t,k,k_j}(A,C,D,H) &:=  \sum_{\substack{h \sim H/k \\ \gcd(h,q_1q_2)=1}} \frac{H}{k} h^{s-1}\sum_{\substack{g=\smqty(a& b \\ c & d)\in \MM_{2,h,k}(\Z) \\ k_1|\gcd(a,b) \\ k_2 | \gcd(c,d)}} \alpha_0(g) f\bigg(\frac{a}{\sqrt{hk}},\frac{c}{\sqrt{hk}},\frac{d}{\sqrt{hk}}\bigg).
\end{align*}
Here for $\smqty(a&b \\ c& d)=\tau \sigma$ as in \eqref{eq:tausigmadef} we have
\begin{align} \nonumber
\alpha(\smqty(a & b \\ c & d)) &:=  \alpha_0(\smqty(a & b \\ c & d)) \mathbf{1}_{\substack{k_1|\gcd(a,b )\\ k_2 |\gcd(c,d) }} = \mathbf{1}_{\substack{k_1 | a_0 \\ k_2 | c_0}} \alpha_0(\tau \sigma)   \\ \nonumber
&= \chi^{(k)}_{1} (f_j r_1 r_2) \chi_{1}^{(r_j)}(k )\chi^{(k)}_{2} (f_j r_1 r_2) \chi_{2}^{(r_2)}(k)   \chi_{1}^{(r_j)}( a_0) \overline{\chi_{1}^{(r_1)}}(b_0) \chi_{2}^{(r_2)}( c_0) \overline{\chi_{1}^{(r_2)}}(d_0) \mathbf{1}_{\substack{k_1 | a_0 \\ k_2 | c_0}}  \\ \label{eq:alphakrk}
&=: \alpha_k (\sigma) \alpha_{r_j,k_j}(\tau).
\end{align}
This is for $\det \smqty(a & b \\ c & d)=hk$ invariant under the action $\Gamma_2(r_1 k_1 ,r_2k_2)$, that is, for $\gamma \in \Gamma_2(r_1k_1,r_2k_2) $  and for any $g \in \MM_{2,h,k}(\Z)$ we have $\alpha_0(\gamma g) = \alpha_0(g)$. This is clear since the action by $\gamma$ does not change the Hecke representative $\sigma \in T_{1,k}$ and the weight $\alpha_{r_j,k_j}(\tau)$ is invariant. It then follows that $\alpha \in \mathcal{A}(r_1k_1,r_2k_2,1,1)$. Thus, we are ready to apply Theorem \ref{thm:twisteddetwbound} for each fixed $k$ and with $\alpha$ as given above,
\begin{align*}
    \beta_h = \mathbf{1}_{h \sim H/k}\mathbf{1}_{\gcd(h,q)=1}\frac{H}{k} h^{s-1}\quad \text{and} \quad \gamma_{k_0} = \mathbf{1}_{k_0 = k}.
\end{align*}
We claim that \eqref{eq:Kassumptiongeneral} holds in the form
\begin{align} \label{claim:LfunctionKbound}
   \frac{1}{k}\sum_{\substack{g = \smqty(a &b \\ c &d ) \in \SL(\R) \\ |a| + |b|C/D + |c|D/C + |d| \leq 10 }} \bigg| \sum_{\substack{\sigma_1,\sigma_2 \in T_{1,k} \\ \sigma_2 g \sigma_1^{-1} =\sigma \in \SL(\Z) }} w(\sigma,\sigma_1,\sigma_2)\bigg| \ll_\eps   Z^{\eps} k \bigg(r_1r_2 + (r_1,r_2) ( C/D + D/C)\bigg).
\end{align}
Assuming this and applying Theorem \ref{thm:twisteddetwbound} we get
\begin{align*}
    S_{\chi_1,\chi_2,s,t,k}(A,C,D,H) = \mathrm{MT}_{\chi_1,\chi_2,s,t,k} + O(T^{O(\eta)} (AD)^{1/2} \|\beta\|_2 \mathcal{K}_+^{1/2} (\mathcal{R}_0 + \mathcal{R}_2 )  )
\end{align*}
with
\begin{align*}
     \|\beta\|_2 &\ll \frac{H^{1/2}}{k^{1/2}}, \\
     \mathcal{K}_+^{1/2} &\ll k^{1/2} (r_1r_2)^{1/2} + k^{1/2} (r_1,r_2)^{1/2} (C/D+D/C)^{1/2}, \\
     \mathcal{R}_0 &\ll \frac{H^{1/2}}{k^{1/2}}  \frac{A^{1/2}}{r_1^{1/2} C^{1/2}}, \\
     \mathcal{R}_2&\ll \bigg( 1+ \bigg(\frac{CD}{k r_2}\bigg)^{\theta}\bigg) \bigg( 1+ \bigg(\frac{HC}{k A r_2}\bigg)^{1/2-\theta}\bigg).
\end{align*}
By \eqref{eq:ratioassumption}, writing $CD= AD C/A \leq AD q_2 Z^\eta$ and $H \leq T^\eta AD/T$ we get
\begin{align*}
     \|\beta\|_2 &\ll \frac{H^{1/2}}{k^{1/2}} \ll k^{-1/2} (AD)^{1/2} T^{-1/2} , \\
     \mathcal{K}_+^{1/2} &\ll k^{1/2} (r_1r_2)^{1/2} + k^{1/2} (r_1,r_2)^{1/2} q_2^{1/2}, \\
     \mathcal{R}_0 &\ll H^{1/2}k^{-1/2}  (q_1/r_1)^{1/2} \ll (AD)^{1/2}   T^{-1/2} k^{-1/2}  (q_1/r_1)^{1/2} , \\
     \mathcal{R}_2&\ll  \bigg(\frac{AD q_2}{k r_2}\bigg)^{\theta}  \bigg(\frac{H q_2}{k r_2}\bigg)^{1/2-\theta} = (AD)^{1/2} k^{-1/2} T^{-1/2+\theta}  (q_2/r_2)^{1/2}.
\end{align*}
Thus, by \eqref{eq:HADrange} we get a total contribution to the error term (after summing over $k |q^{\infty}$ and integrating over $t \asymp T$, recalling the normalization by $\frac{1}{AD}$ in \eqref{eq:ADnormalization})
\begin{align*}
   & \ll T^{O(\eta)} \frac{T}{AD} (AD)^{1/2}  \|\beta\|_2 E^{1/2} (\mathcal{R}_0 + \mathcal{R}_2 )  \\
   & \ll   T^{O(\eta)}  T^\theta (AD)^{1/2} ((q_1/r_1)^{1/2} + (q_2/r_2)^{1/2})((r_1r_2)^{1/2} + (r_1,r_2)^{1/2} q_2^{1/2}) \\
   & \ll  T^{O(\eta)}   T^\theta  (AD)^{1/2}  ( \sqrt{q_1q_2} + q_2   ) \\
     & \ll  T^{O(\eta)}   T^{1/2+\theta} (q_1q_2)^{3/4},
\end{align*}
since by symmetry we may assume that $q_2 \leq q_1$. The main term can be calculated by reversing the previous steps. This proves Theorem \ref{cor:Lfunction}  under the assumption of \eqref{claim:LfunctionKbound}.

\subsubsection{Proof of \eqref{claim:LfunctionKbound}}
Denoting $L:=C/D$, we need to bound
\begin{align*}
\mathcal{K} := \sum_{\substack{g = \smqty(a &b \\ c &d ) \in  \SL(\R) \\ |a| + |b|L + |c|/L + |d| \leq 10 }} \bigg| \sum_{\substack{\sigma_1,\sigma_2 \in T_{1,k} \\ \sigma_2 g \sigma_1^{-1} =\sigma \in \SL(\Z) }} w(\sigma,\sigma_1,\sigma_2)\bigg|.
\end{align*}
By \eqref{eq:alphakrk}
\begin{align*}  w(\sigma,\sigma_1,\sigma_2)&=\sum_{\substack{\tau\in \Gamma_2(r_1k_1,r_2k_2) \backslash \SL(\Z)}  } \alpha_0(\tau \sigma \sigma_1)\overline{\alpha_0(\tau \sigma_2)} \\
&= \alpha_k(\sigma_1) \overline{\alpha_k(\sigma_2)} \sum_{\substack{\tau\in \Gamma_2(r_1k_1,r_2k_2) \backslash \SL(\Z)}  } \alpha_{r_j,k_j}(\tau \sigma)\overline{\alpha_{r_j,k_j}(\tau)}   \\
& = \mathbf{1}_{\substack{\sigma = \smqty(a & b \\ c & d) \\ k_2 | c ,\, k_1 | d } }  \,\alpha_k(\sigma_1) \overline{\alpha_k(\sigma_2)} \sum_{\substack{\tau\in \Gamma_2(r_1,r_2) \backslash \SL(\Z)}  }  \alpha^{(r_1,r_2)}(\tau \sigma)\overline{\alpha^{(r_1,r_2)}(\tau)},
\end{align*}
where we have denoted
\begin{align*}
  \alpha^{(r_1,r_2)}(\smqty(a_0 & b_0 \\ c_0 & d_0)) :=  \chi_{1}^{(r_j)}( a_0) \overline{\chi_{1}^{(r_1)}}(b_0) \chi_{2}^{(r_2)}( c_0) \overline{\chi_{1}^{(r_2)}}(d_0).
\end{align*}
We thus have 
\begin{align*}
|w(\sigma,\sigma_1,\sigma_2)| \leq \bigg|\sum_{\substack{\tau\in  \Gamma_2(r_1,r_2) \backslash \SL(\Z)}  } \alpha^{(r_1,r_2)}(\tau \sigma) \overline{\alpha^{(r_1,r_2)}(\tau) } \bigg| =: w_0(\sigma), 
\end{align*}
which is as in Lemma \ref{le:charactersumbound}. Recall that $q_j$ are square-free so that $(r_j,k)=1$, and we use representatives for $T_{1,k}$ of the form
\begin{align*}
    \sigma_j = \mqty(1 & f_j r_1 r_2 \\ & k), \quad 1 \leq f_j < k, \quad \gcd(f_j,k)=1.
\end{align*}
 By a change of variables  we get
\begin{align*}
      \mathcal{K} 
 &=    \sum_{\sigma \in \SL(\Z)} w_0(\sigma) \sum_{\substack{\sigma_j =\smqty(1 & f_j r_1 r_2 \\ & k) \\ \sigma_2^{-1} \sigma \sigma_1 = \smqty(a_0 & b_0 \\c_0 & d_0 ) \\ |a_0| + |b_0| L + |c_0|/L + |d| \leq 10}} 1 
    =  \mathcal{K}_{1} + \mathcal{K}_{2} + \mathcal{K}_3, \\
      \mathcal{K}_1 &:=  \sum_{\substack{\sigma= \smqty(a & b \\ c& d) \in \SL(\Z) \\c=0}} w_0(\sigma) \sum_{\substack{\sigma_j =\smqty(1 & f_j r_1 r_2 \\ & k) \\ \sigma_2^{-1} \sigma \sigma_1 = \smqty(a_0 & b_0 \\c_0 & d_0 ) \\ |a_0| + |b_0| L + |c_0|/L + |d| \leq 10}} 1   , \\
       \mathcal{K}_2  &:=    \sum_{\substack{\sigma= \smqty(a & b \\ c& d) \in \SL(\Z) \\ c \neq 0}} w_0(\sigma) \sum_{\substack{\sigma_j =\smqty(1 & f_j r_1 r_2 \\ & k) \\ \sigma_2^{-1} \sigma \sigma_1 = \smqty(a_0 & b_0 \\c_0 & d_0 ) \\ |a_0| + |b_0| L + |c_0|/L + |d| \leq 10 \\  b_0= 0}} 1,  \\
     \mathcal{K}_3 &:=  \sum_{\substack{\sigma= \smqty(a & b \\ c& d) \in \SL(\Z) \\ c \neq 0}} w_0(\sigma) \sum_{\substack{\sigma_j =\smqty(1 & f_j r_1 r_2 \\ & k) \\ \sigma_2^{-1} \sigma \sigma_1 = \smqty(a_0 & b_0 \\c_0 & d_0 ) \\ |a_0| + |b_0| L + |c_0|/L + |d| \leq 10 \\  b_0 \neq 0}} 1.
\end{align*} 
Here for $\sigma = \smqty(a &b \\ c& d)$
\begin{align*}
   \mqty(a_0 & b_0 \\c_0 & d_0 )=  \sigma_2^{-1} \sigma \sigma_1  = \mqty(a-cf_2 r_1 r_2/k & f_1 r_1 r_2 a-f_2 r_1 r_2 d+kb-cf_1f_2 (r_1r_2)^2/k\\ c/k & d+cf_1 r_1 r_2/k).
\end{align*}

We have by Lemma \ref{le:charactersumbound} 
\begin{align*}
   \mathcal{K}_{1} \leq &   \sum_{\smqty(\pm 1 & b \\ & \pm 1) \in \SL(\Z)} w_0(\smqty(\pm 1 & b \\ & \pm 1) ) \sum_{\substack{ f_1,f_2 \leq k \\ |f_1 r_1 r_2 - f_2 r_1r_2 \pm k b| \leq 10 /L}} 1  \\
   \ll &  k  \sum_{\smqty(\pm 1 & b \\ & \pm 1) \in \SL(\Z)}  \sum_{\substack{ |f| \leq k \\ |f r_1 r_2 \pm k b| \leq 10 /L}} w_0(\smqty(\pm 1 & b \\ & \pm 1) ) \\
   \ll & k \sum_{|b_0| \leq 10 /L} w_0(\smqty(\pm 1 & b_0 \\ & \pm 1) ) \ll_\eps Z^{\eps} k \bigg(r_1 r_2 +  \frac{\gcd(r_1,r_2)}{L }\bigg)
\end{align*}
by combining $b_0 = f r_1 r_2  \pm k b$ since $\gcd(k,r_1r_2)=1$ and $b \mapsto w_0(\smqty(\pm 1 & b \\ & \pm 1) )$ is $r_1r_2$ periodic in $b$. 

To bound $\mathcal{K}_2$ note that $b_0 = 0$ implies that $b \equiv 0 \pmod{r_1r_2}$ and $c \equiv 0 \pmod{k}$, so that
\begin{align*}
   \mathcal{K}_{2} = &   \sum_{\substack{\smqty(a &  b \\ c & d) \in \SL(\Z) \\  0< |c| \leq 10 k L \\ b \equiv 0 \pmod{r_1r_2} \\ c \equiv 0 \pmod{k}}] } w_0(\smqty(a &  b \\ c & d) ) \sum_{\substack{ f_1,f_2 \leq k \\ |a - c f_2 r_1r_2/k | \leq 10 \\   | d + c f_1 r_1r_2/k |  \leq 10 \\ b_0 = 0 }} 1   
  \end{align*} 
 We have for $r_1r_2 | b, \gcd(ad,r_1r_2)=1,$ and $k|c$ that $w_0(\smqty(a &  b \\ c & d) )  = w_0(\smqty(1 &  0 \\ c & 1) ) =w_0(\smqty(1 &  0 \\ c/k & 1) )$, which depends only on $\gcd(c,r_1r_2)$. Thus, by Lemma \ref{le:charactersumbound} we get
  \begin{align*}
      \mathcal{K}_{2} \ll  &   \sum_{\substack{ 0< |c| \leq 10 k L \\ c \equiv 0 \pmod{k} } } w_0(\smqty(1 &  0 \\ c & 1) ) \sum_{\substack{a,b,d \\ ad-bc=1}}\sum_{\substack{ f_1,f_2 \leq k \\ |a - c f_2 r_1r_2/k | \leq 10 \\   | d + c f_1 r_1r_2/k |  \leq 10 \\ b_0 = 0 }} 1  \\
      \ll & \sum_{\substack{ 0< |c| \leq 10  L  } } w_0(\smqty(1 &  0 \\ c & 1) )\sum_{\substack{a,b,d \\ ad-bc k =1}}\sum_{\substack{ f_1,f_2 \leq k \\ |a - c f_2 r_1r_2 | \leq 10 \\   | d + c f_1 r_1r_2 |  \leq 10 \\ r_1r_2(f_2d-f_1 a) + c f_1f_2 (r_1r_2)^2 = kb }} 1
      \\
      \ll & \sum_{\substack{ 0< |c| \leq 10  L  } } w_0(\smqty(1 &  0 \\ c & 1) )\sum_{\substack{ f_1,f_2 \leq k }}\sum_{\substack{a,d \\ |a - c f_2 r_1r_2 | \leq 10 \\   | d + c f_1 r_1r_2 |  \leq 10 }} 1 \\
       \ll & k^2 \sum_{\substack{ 0< |c| \leq 10  L  } } w_0(\smqty(1 &  0 \\ c & 1) ) \ll_\eps Z^{\eps} k^2 L (r_1,r_2).
  \end{align*}

Finally, we have by the trivial bound $w_0(\sigma) \ll_\eps (r_1r_2)^{1+\eps}$ that
\begin{align*}
      \mathcal{K}_{3} \ll & (r_1r_2)^{1+\eps}   \sum_{\substack{\smqty(a & b \\ c & d) \in \SL(\Z) \\ c \neq 0  }}\sum_{\substack{\sigma_1,\sigma_2 \in T_{1,k} \\   \sigma_2^{-1} \sigma \sigma_1 = \smqty(a_0 & b_0 \\c_0 & d_0 ) \\ |a_0| + |b_0| L + |c_0|/L + |d| \leq 10 \\ b_0 \neq 0  }} 1.
     \\  \ll & \,  (r_1r_2)^{1+\eps}   \sum_{\substack{\smqty(a_1 & b_1 \\ c_1 & d_1) \in \MM_{2,k}(\Z) \\ c \neq 0 \\ }}\sum_{\substack{\sigma_2 \in T_{1,k} \\   \sigma_2^{-1} \smqty(a_1 & b_1 \\ c_1 & d_1) = \smqty(a_0 & b_0 \\c_0 & d_0 ) \\ |a_0| + |b_0| L + |c_0|/L + |d| \leq 10 \\ b_0 \neq 0}} 1. 
\end{align*}
Here
\begin{align*}
   \mqty(a_0 & b_0 \\c_0 & d_0 )  =  \frac{1}{k} \mqty(k & -f_2 r_1r_2 \\ & 1) \mqty(a_1 & b_1 \\ c_1 & d_1)  = \frac{1}{k} \mqty(k a_1-f_2 r_1r_2  c_1 & k b_1-f_2 r_1r_2  d_1 \\ c_1 & d_1) = \frac{1}{k} \mqty(a_2 & b_2 \\ c_2 &  d_2).
\end{align*}
Thus, summing over $f_2$ and combining the variables (using $(k,r_1r_2)=1$) we get
\begin{align*}
  \mathcal{K}_{3} & \ll (r_1r_2)^{1+\eps} \sum_{\substack{a_2 d_2 -b_2 c_2 = k^2 \\ |a_2|/k + |b_2|L/k + |c_2|/(kL) + |d_2| /k \leq 10  \\ b_2 c_2 \neq  0}} 1 \\
   & \ll  (r_1r_2)^{1+\eps}  \sum_{\substack{a_2,d_2 \ll  k \\ a_2d_2 \neq k^2}} \tau(a_2 d_1 -k^2) \\
    & \ll  k^2(r_1r_2)^{1+\eps} .
\end{align*}
\end{proof}
\begin{proof}[Proof of Theorem \ref{cor:Lfunctiononechar}]
For Theorem \ref{cor:Lfunctiononechar} where $\chi_1 = \chi_2, q=q_1=q_2$ the previous argument is improved as follows. By using the approximate functional equation for $L(s,\chi)$ \cite[Theorem 5.3]{IKbook}  (instead of the square $L(s,\chi)^2$ as in \cite{berke})
\begin{align*}
    L(s,\chi) = \sum_{n} \frac{\chi(n)}{n^s} V_s(\frac{n}{\sqrt{q}}) + \eps(\chi,s) \sum_{n} \frac{\overline{\chi}(n)}{n^{1-s}} V_{1-s}(\frac{n}{\sqrt{q}}),
\end{align*}
we get (denoting $V_1(n,s) = V_s(\frac{n}{\sqrt{q}}) \chi(n) n^{-s},V_2(n,s) = \eps(\chi,s)V_{1-s}(\frac{n}{\sqrt{q}})  \overline{\chi}(n)n^{1-s}$)
\begin{align*}
    |L(s,\chi)|^4 = \sum_{j_1,\dots,j_4 \in \{1,2\} } \sum_{n_1,\dots,n_4}  \prod_{k \leq 4}V_{j_k}(n_k).
\end{align*}
The contribution from $(j_1,j_2,j_3,j_4) $ which are not a permutation of $(1,1,2,2)$ may be handled trivially, for example, for $(j_1,j_2,j_3,j_4)=(1,2,2,2)$ we get  by  integration by parts
\begin{align} \label{eq:ipL}
\begin{split}
        \int \bigg(\frac{n_2n_3n_4}{n_1}\bigg)^{it} V_s(\frac{n_1}{\sqrt{q}}) V_{1-s}(\frac{n_2}{\sqrt{q}})  &V_{1-s}(\frac{n_3}{\sqrt{q}})  V_{1-s}(\frac{n_4}{\sqrt{q}}) \eps(\chi,s)^3 \omega(t/T) \d t  \\
    &\ll_J T \bigg(1+ \frac{|n_2n_3 n_4-n_1|  T^{1-O(\eta)}}{n_1}\bigg)^{-J}
\end{split}
\end{align}
Therefore, the part where $n_1 > T^{1-O(\eta)}$ contributes 
\begin{align*}
 \ll 1+  T \sum_{\substack{n_1,\dots,n_4 \ll T^\eta \sqrt{q T} \\ n_1 > T^{1-O(\eta)}  \\ n_2n_3 n_4 = n_1 (1+ O(T^{-1 + O(\eta)})) } } \frac{1}{\sqrt{n_1 n_2 n_3 n_4}} \ll& 1+  T (qT)^{O(\eta)} \sum_{\substack{T^{1-O(\eta)}< n_1 \ll T^\eta \sqrt{q T}  } } \bigg(\frac{1 }{n_1^2}  + \frac{1 }{T n_1}\bigg) \\
 \ll &(qT)^{O(\eta)}.
\end{align*}
For $n_1 < T^{1-O(\eta)}$ we get a negligible contribution outside the diagonal $n_1=n_2n_3n_4$ by \eqref{eq:ipL}, and the diagonal gives a contribution to the main term since there $\chi(n_1) \overline{\chi (n_1n_2n_3)} = \mathbf{1}_{(n_1n_2n_3n_4,q) = 1}$.

In the remaining part we can assume by symmetry that $(j_1,j_2,j_3,j_4)= (1,2,2,1)$ and we may truncate the variables $(a,b,c,d)=(n_1,n_2,n_3,n_4)$ to ranges $A,B,C,D \ll T^\eta \sqrt{q T}$. Denote $X=AD \ll q T$ so that $H \leq T^\eta X/T$. Then we may assume that $X > T^{1-O(\eta)}$ (otherwise $h=0$) and further that $A,C > \sqrt{X}$ (otherwise we can swap $b \leftrightarrow c$, $a\leftrightarrow d$). This implies that $Z^{-\eta} (X/qT)^{1/2} < C/A < Z^{\eta} (qT/X)^{1/2}$, which is stronger than \eqref{eq:ratioassumption}. The bound for the error term in the application of Theorem \ref{thm:twisteddetwbound} is consequently improved to
\begin{align*}
 \ll   T^{O(\eta)}   T^{1/2+\theta}\,  q.
\end{align*}
\end{proof}

\begin{remark}
For two different characters (or even four different characters) one may also use approximate functional equations for each $L$-function separately. This possibly introduces a nebentypus character $\chi$ and the error term will depend on its conductor, via the bounds for $\mathcal{R}_2$ in Theorem \ref{thm:twisteddetwbound}. This may lead to better results in these cases. In particular if, as indicated in Remark \ref{rem:improvedchidependcy}, the conductor loss is reduced or even eliminated.  We do not pursue this further here.
\end{remark}

\bibliography{SL2bib}
\bibliographystyle{abbrv}

\end{document}